\documentclass[11pt,reqno]{amsart}
\usepackage{amsmath, amssymb, amsthm, amsfonts, mathrsfs, hyperref}
\usepackage[table]{xcolor}
\usepackage{stmaryrd} %\llbracket, \rrbracket
\usepackage{multirow}
\allowdisplaybreaks
\usepackage{bbm} % \mathbbm
\usepackage{leftidx} %\ltrans, \leftidx
\usepackage{mathabx} %\widecheck
\usepackage{mathtools} %\prescript
\usepackage{rotating} %\rotatebox
\usepackage{youngtab,young}

\usepackage{todonotes}
\usepackage{arydshln}
\makeatletter
\renewcommand*\env@matrix[1][*\c@MaxMatrixCols c]{%
  \hskip -\arraycolsep
  \let\@ifnextchar\new@ifnextchar
  \array{#1}}
\makeatother
\usepackage{pgf,tikz}
\usetikzlibrary{arrows, positioning, calc, chains}
\tikzset{
	ch/.style={circle,draw,on chain,inner sep=2pt},
	chj/.style={ch,join},
	every path/.style={shorten >=4pt,shorten <=4pt}
	}
\newcommand{\dnode}[2][chj]{%
	\node[#1,label={below:#2}] (#1) {};}
\newcommand{\dnodenj}[1]{%
	\dnode[ch]{#1}}
\newcommand{\dydots}{%
	\node[chj,draw=none,inner sep=1pt] {\dots};}
\usepackage[margin = 0.8in]{geometry}
\newcommand{\alg}{\textup{alg}}
\newcommand{\bA}[1]{    \begin{align}  #1  \end{align}}
\newcommand{\bAn}[1]{    \begin{align*}  #1  \end{align*}}
\newcommand{\ba}[1]{    \begin{array}   #1   \end{array}}
\newcommand{\bbA}{\mathbb{A}}
\newcommand{\bfA}{\mathbf{A}}
\newcommand{\bc}[1]{     \begin{cases}   #1   \end{cases}}

\newcommand{\bD}{\mbf D}

\newcommand{\be}{\mbf e}
\newcommand{\bsp}[1]{    \begin{split}  #1  \end{split}}
\newcommand{\bbf}{\mbf f}

\newcommand{\bH}{\mbf H}

\newcommand{\bK}{\mbf K}

\newcommand{\bL}{\mbf L}

\newcommand{\bm}[1]{    \left[\begin{smallmatrix}    #1    \end{smallmatrix}\right]}

\newcommand{\Bp}[1]{\Big{(} #1\Big{)}}

\newcommand{\bS}{\mbf S }

\newcommand{\bU}{\mbf U}
\newcommand{\bv}{\mbf v}

\newcommand{\CL}{C^{\bL}}

\newcommand{\co}{\textup{col}}

\newcommand{\cP}{\mathcal{P}}

\newcommand{\D}{\mathscr{D}} %distinguished coset representative
\newcommand{\dBj}{\dot{\mathfrak B}}
\newcommand{\diag}{\textup{diag}}
\newcommand{\dK}{\dot{\bK}}
\newcommand{\dKKi}{\dot{\KK}^\imath_{n}}
\newcommand{\dKKj}{\dot{\KK}^\jmath_{n}}
\newcommand{\dKKjp}{\dot{\KK}^{>}_n}
\newcommand{\ds}{\displaystyle}
\newcommand{\dU}{\dot{\bU}}
\newcommand{\dUUi}{\dot{\UU}^\imath}
\newcommand{\dUUj}{\dot{\UU}^\jmath}

\newcommand{\fa}{\mathfrak{a}}

\newcommand{\fc}{\mathfrak{c}}
\newcommand{\fd}{\mathfrak{d}}

\newcommand{\fgl}{\mathfrak{gl}}

\newcommand{\fS}{\mathfrak{S}}

\newcommand{\HH}{\mathbb{H}}

\newcommand{\Hom}{\textup{Hom}}

\newcommand{\id}{\mathbbm{1}}
\newcommand{\II}{\mathbb{I}}

\newcommand{\IIj}{\mathbb{I}^\jmath}
\newcommand{\inv}{^{-1}}
\newcommand{\iXi}{\Xi^\imath}
\newcommand{\iXit}{\~{\iXi}}
\newcommand{\JJ}{\mathbb{J}}
\newcommand{\KK}{\mathbb{K}}
\newcommand{\ld}{\lambda}
\newcommand{\Ld}{\Lambda}
\newcommand{\lrb}[2]{{#1 \brack #2}}

\newcommand{\LR}[1]{\left\llbracket #1 \right\rrbracket}

\newcommand{\mA}{\mathcal{A}}
\newcommand{\mB}{\mathcal{B}}
\newcommand{\mC}{\mathcal{C}}

\newcommand{\mbf}{\mathbf}
\newcommand{\mrm}{\mathrm}
\newcommand{\NN}{\mathbb{N}}

\newcommand{\otw}{\textup{otherwise}}
\newcommand{\p}[1]{\leftidx{_{p}}{#1}}
\newcommand{\pp}[1]{\leftidx{_{\breve{p}}}{#1}}
\newcommand{\QQ}{\mathbb{Q}}

\newcommand{\ro}{\textup{row}}

\newcommand{\rw}{\rightarrow}

\newcommand{\sgn}{\mathrm{sgn}}

\newcommand{\SSj}{\mathbb{S}^{\jmath}_{n,d}}
\newcommand{\SSi}{\mathbb{S}^{\imath}_{n,d}}

\newcommand{\Stab}{\textup{Stab}}

\newcommand{\tfor}{\textup{for }}
\newcommand{\tif}{\textup{if }}
\newcommand{\tinv}{\mathrm{inv}}
\newcommand{\tneg}{\mathrm{neg}}

\renewcommand{\tt}{\theta}
\newcommand{\Tt}{\Theta}

\newcommand{\UU}{\mathbb{U}}
\newcommand{\UUi}{\mathbb{U}^\imath}
\newcommand{\UUj}{\mathbb{U}^\jmath}

\newcommand{\Xit}{\~{\Xi}}
\newcommand{\ZZ}{\mathbb{Z}}

\renewcommand{\^}[1]{\widehat{#1}}
\renewcommand{\~}[1]{\widetilde{#1}}
\renewcommand{\=}[1]{\overline{#1}}

\newenvironment{eqa}{\begin{eqnarray}}{\end{eqnarray}}
\newenvironment{enu}{\begin{enumerate}}{\end{enumerate}}
\usepackage{enumitem }
\newenvironment{enua}{\begin{enumerate}[label=$(\alph*)$]
}{\end{enumerate}}
\newenvironment{eq}{\begin{equation}}{\end{equation}}
\newenvironment{itm}{\begin{itemize}}{\end{itemize}}
\numberwithin{equation}{subsection}
\theoremstyle{definition}
\newtheorem{Def}{Definition}[subsection] %[subsection]

\newtheorem{exa}[Def]{Example}
\newtheorem{rem}[Def]{Remark}
\newtheorem{rmk}[Def]{Remark}

\theoremstyle{plain}
\newtheorem{prop}[Def]{Proposition}
\newtheorem{thm}[Def]{Theorem}
\newtheorem{lem}[Def]{Lemma}

\newtheorem{lemma}[Def]{Lemma}
\newtheorem{cor}[Def]{Corollary}

\newtheorem*{thmA}{Theorem A}
\newtheorem*{thmB}{Theorem B}
\newtheorem*{thmC}{Theorem C}
\newtheorem*{thmD}{Theorem D}
\title[Schur algebras and QSPs with unequal parameters]{Schur algebras and quantum symmetric pairs with unequal parameters}
\author[C.~Lai and L.~Luo]{Chun-Ju Lai and Li Luo}
\address{Department of Mathematics, University of Georgia, Athens, GA 30605}
\email{cjlai@uga.edu (Lai)}
\address{ %${}^{\dagger}$
    Department of mathematics, Shanghai Key Laboratory of Pure Mathematics and Mathematical Practice, East China Normal University, Shanghai 200241, China}
\email{lluo@math.ecnu.edu.cn (Luo)}
\begin{document}

\begin{abstract}
We study the (quantum) Schur algebras of type B/C corresponding to the Hecke algebras with unequal parameters.
We prove that the Schur algebras afford a stabilization construction in the sense of Beilinson-Lusztig-MacPherson that constructs a multiparameter upgrade of the quantum symmetric pair coideal subalgebras of type AIII/AIV with no black nodes.
We further obtain the canonical basis of the Schur/coideal subalgebras, at the specialization associated to any weight function.
These bases are the counterparts of Lusztig's bar-invariant basis for Hecke algebras with unequal parameters.
In the appendix we provide an algebraic version of a type D Beilinson-Lusztig-MacPherson construction which is first introduced by Fan-Li from a geometric viewpoint.
\end{abstract}

\maketitle
\setcounter{tocdepth}{1}
\tableofcontents
%=========================================================
\section{Introduction}
%=========================================================
\subsection{Background}
%=========================================================
The quantum groups introduced by Drinfeld and Jimbo have played a central role in representation theory and many other branches of mathematics.
Equally important are Lusztig's modified (or idempotented) quantum groups (cf. \cite{Lu93})
that admit the canonical bases, which are analogs of the Kazhdan-Lusztig bases for the Hecke algebras.
In \cite{BLM90}, a geometric construction of the modified quantum group $\dU(\fgl_n)$ is given by Beilinson-Lusztig-MacPherson.
Their construction is now referred as the BLM or stabilization construction after a stabilization property of the family of the (quantum) Schur algebras of type A.
In this paper, by a (equal-parameter)\footnote{Our goal} {\it stabilization construction of type $X$} we mean a construction of an algebra $\dK^X_n$ over $\ZZ[v,v\inv]$ such that
\enu
\item There is a family of quantum Schur algebras $\bS^X_{n,d}$, which are the centralizing algebras to the action of the Hecke algebra $\bH^X_d$ of type $X_d$, for all $n,d$;
\item The family $\{\bS^X_{n,d}~|~ d\in \NN\}$ admits a stabilization property, namely, the algebra $\dK^X_n = \mathop{\textup{Stab}}\limits_{\infty\leftarrow d} \bS^X_{n,d}$ is well-defined. As a consequence, there is a basis of $\dK^X_n$ that is compatible with the Kazhdan-Lusztig bases for $\bH^X_d$, and the canonical bases of $\bS^X_{n,d}$ for all $d$.
\endenu
%One can show that the stabilization property holds true if the Schur algebra is replaced by the centralizing partner of the Hecke algebra of other type in a Schur-type duality.
The stabilization constructions have been developed for classical type and for certain affine type (see Table~\ref{tab:BLM} for the references) --
%There are geometric approaches using partial flags and counting over finite fields (see \cite{BKLW14} for type B/C; \cite{FL14} for type D; \cite{Lu99} for affine type A; \cite{FL3W17} for affine type C).
there are geometric approaches using partial flags and counting over finite fields developed;
while there also are algebraic approaches in the framework of the Hecke algebras using combinatorics on Coxeter groups.

\begin{table}[h]
\caption{Known BLM/stabilization constructions}\label{tab:BLM}
\begin{tabular}{c|cccccccc}
type& finite A& finite B/C & finite D& affine A& affine C
\\ \hline
geometric&\cite{BLM90}&\cite{BKLW18}&\cite{FL15}&\cite{Lu99}&\cite{FL3Wa}
\\
algebraic&\cite{DDPW08}& ? & ? &\cite{DF15}&\cite{FL3Wb}
\end{tabular}
\end{table}
We remark that the algebraic approach for finite type B/C is more or less a special case for affine type C;
while the algebraic approach for type D will be given in the appendix of this present paper.

The stabilization construction in general produces not the Drinfeld-Jimbo's quantum groups but Letzter-Kolb's quantum symmetric pairs (cf. \cite{Le02, Ko14}). For example, the stabilization constructions of type A and B/C lead to the quantum symmetric pairs of type AIII/IV with no black nodes.
%=========================================================
\subsection{A new direction}
%=========================================================
A recent work by Bao-Wang-Watanabe brings to the author's attention that a multiparameter Schur duality (cf. \cite{BWW18}) plays a governing role among the Schur dualities of classical type.
They also introduce a multiparameter upgrade of quantum symmetric pairs of type AIII/AIV with no black nodes.

While it is unclear how to proceed a geometric approach with unequal parameters since dimension counting does not make sense in an obvious way,
an algebraic/combinatorial approach seems viable.
The goal of this article is to provide a stabilization construction with respect to the Schur duality with unequal parameters in {\it loc. cit.}
We show that the multiparameter stabilization algebras constructed are the coideal subalgebras appearing in the quantum symmetric pairs of type AIII/AIV with no black nodes.
As an application, we construct, for the first time, the canonical bases for the type B/C Schur algebras with unequal parameters associated to any weight function, using Lusztig's bar-invariant basis \cite{Lu03} with unequal parameters.

The following diagram explains briefly the connection between the stabilization construction of type B/C for equal and unequal parameters (here $\mbf{c} = \gcd(\bL(s_0), \bL(s_1))$, and there are two distinct cases where $\bullet$ can be replaced by $\imath$ or $\jmath$):

\begin{table}[h]
\caption{Relation between Schur duality of type B/C at various specializations}
\[
\ba{{llcccccc}
\ds
\mathbb{S}^\bullet_{n,d} \curvearrowright \mathbb{V}^{\otimes d}\curvearrowleft \mathbb{H}_d&\textup{over }\ZZ[u^{\pm1}, v^{\pm1}]
\\[5pt]
\ds
\Downarrow\textup{ specialization at } u = \bv^{\bL(s_0)}, v = \bv^{\bL(s_1)}
\\[5pt]
\ds
\mathbb{S}^{\bullet,\bL}_{n,d} \curvearrowright \mathbb{V}_\bL^{\otimes d}\curvearrowleft \mathbb{H}^\bL_d&\textup{over }\ZZ[\bv^{\pm\mbf{c}}]
\\[5pt]
\Downarrow \textup{ specialization at } u=v=\bv ~(\textup{i.e., } \bL = \ell)
\\[5pt]
\mathbf{S}^{\bullet}_{n,d} \curvearrowright \mathbf{V}^{\otimes d}\curvearrowleft \mathbf{H}_d&\textup{over }\ZZ[\bv^{\pm1}]
}
\]
\end{table}
\label{rmk:typeD}
%Finally, a weak Schur duality of type D is realized at the specialization $u=1, v = \bv$,
%which provides a stabilization realization of the algebras used to formulate the %Kazhdan-Lusztig theory of super type D (cf. \cite{Bao17}).
%Note that the Hecke algebra here contains the type D Hecke algebra over $\ZZ(\bv^{\pm1})$ %as a proper subalgebra, and hence the weak Schur duality is essentially different from the %type D Schur duality in the appendix.

At the specialization $u=1$, the Hecke algebra contains the type D Hecke algebra over $\ZZ(v^{\pm1})$ as a proper subalgebra. Hence the multiparameter Schur duality yields a weak Schur duality of type D which is used in \cite{Bao17} to formulate the Kazhdan-Lusztig theory for classical and super type D.
The very duality also appears in \cite{ES18} as a piece of a larger skew Howe duality of the quantum symmetric pair coideal subalgebra with itself.

%==========================================================
\subsection{Unequal parameters}
%=========================================================
While the organization of this paper follows closely to the (equal-parameter) affine type C construction  \cite{FL3Wb}, the technical lemmas therein do not generalize naively.
Below we mention some notable difficulties working with unequal parameters.

The first difficulty comes to dealing with the combinatorics of (type B/C) quantum numbers with two parameters.
The key observation here is that the (equal-parameter) quantum numbers/factorials used in the BLM-type constructions arise from the (equal-parameter) Poincare polynomials corresponding to the Weyl groups.
Hence, we compute the multiparameter upgrade for the type B/C Poincare polynomials (cf. Lemma~\ref{lem:[A]!}), and then extract from it a type B/C quantum factorial \eqref{def:c-factorial} with two parameters.

The second difficulty arises in constructing a standard basis of $\SSj$.
For the equal-parameter case such a basis element $[A]$ is obtained by multiplying a $\bv$-power to the evident basis $e_A$; while for unequal parameters, it is not obvious how to define a multiplier $u^\bullet v^\bullet$ that specializes to the original $\bv$-power.
We solve this problem by reducing it to getting an explicit formula (cf. Lemma~\ref{lem:bar}) for the leading coefficient under the bar map.
For the equal-parameter case the formula is obtained using certain identities on the dual Kazhdan-Lusztig basis due to Curtis.
However, there are no multiparameter Kazhdan-Lusztig basis known to us (yet).
Hence, we take a detour via Lusztig's bar-invariant basis $c_w$ with unequal parameters and have successfully define a standard basis that affords the entire stabilization process.

%The remaining parts requires just careful bookkeeping.

Finally, we remark that there is an unexpected behavior for our multiparameter monomial bases --
the basis elements are not bar-invariant, unlike the (equal-parameter) monomial basis elements.
As a result, we can only show the existence of canonical bases for Schur algebras at certain specialization (see Section~\ref{sec:SCB}).

%=========================================================
\subsection{Organization and main results}
%=========================================================
%We also set $\Xi_n = \bigcup_{d\in\NN} \Xi_{n,d}$.
Throughout the article the algebras are over the ground ring
$$\bbA = \ZZ[u^{\pm1},v^{\pm1}]$$ ($u,v$ are independent indeterminants) and its specializations.

We first start with the case $\bullet = \jmath$.
In Section~\ref{sec:Weyl} we recall combinatorial properties of Weyl groups of type B/C in terms of permutation matrices.
We characterize a matrix set $\Xi_{n,d}$ (see \eqref{def:Xi}) associated to certain double coset representatives.
%As a consequence, we obtain a unconventional length formula (see Lemma~\ref{lem:l(A)}) for double coset representatives via matrix entries.
We also introduce the multiparameter quantum numbers of type B/C corresponding to the Poincare polynomials.
In Section~\ref{sec:schur} we introduce the Schur algebra $\SSj$ (see \eqref{def:Sj}) with an evident basis $\{e_A\mid A\in\Xi_{n,d}\}$.
In Section~\ref{sec:bases} we introduce a standard basis $\{[A] \mid A\in\Xi_{n,d}\}$ (see \eqref{def:[A]}), and we show that, using Lusztig's basis $c_w$ for the Hecke algebras with unequal parameters, it satisfies a unitriangular condition under the bar involution.
The first main result is the following multiparameter upgrade of the multiplication formulas in \cite{BKLW18}:
%For example, the Dipper-James basis $\{e_A\}_A$ (\eqref{def:eA}), standard basis $\{[A]\}_A$ (\eqref{def:[A]}), monomial basis $\{m_A\}_A$ and canonical bases
%-----------------------------------------------------------------------------------------------------------
%\begin{prop}
%The $q$-Schur algebra $\SSj$ has a linear basis $\{e_A\}$ (see \eqref{def:eA}) parametrized by $\Xi$.
%\end{prop}
%-----------------------------------------------------------------------------------------------------------
%-----------------------------------------------------------------------------------------------------------
\begin{thmA}[Theorem~\ref{thm:multformula2}]
Let $A, B \in \Xi_{n,d}$ and $B - b(E_{h,h+1} +E_{-h,-h-1})$ is diagonal.
Let $\gamma_{B,A}^C \in \bbA$ be such that
$
[B] [A]
=\sum_C
\gamma_{B,A}^C
[C] \in \SSj$.
The explicit formula and the vanishing criterion for $\gamma_{B,A}^C$ are computed.
\end{thmA}
The multiplication formula plays an essential step towards constructing a monomial basis in the sense that a stabilization property \eqref{eq:monop} holds. %(and hence canonical basis).
\begin{thmB}[Proposition~\ref{prop:mono}, Theorem~\ref{thm:SCB}]
There exists a monomial basis $\{m_A\}$ for the Schur algebra $\SSj$ over $\bbA$. Consequently, at a specialization associated to a weight function $\bL$,
there exists a canonical basis $\{\{A\}^\bL\}$ for $\mathbb{S}^{\jmath, \bL}_{n,d}$.
\end{thmB}
In Section~\ref{sec:stabj} we show that the stabilization procedure along the line of Beilinson-Lusztig-MacPherson applies to the family of Schur algebras $\{\SSj\mid d \ge 1\}$ with a fixed $n$, which leads to the construction of stabilization algebra $\dKKj$ (cf. Corollary~\ref{cor:stab}) together with its canonical basis.
\begin{thmC}[Theorem~\ref{thm:KCB}]
There exists a monomial basis $\{m_A\}$ for the stabilization algebra $\dKKj$. As a corollary, there exists a canonical basis $\{\{A\}^\bL\}$ for $\dKKj$ at a specialization associated to a weight function $\bL$.
\end{thmC}
Section~\ref{sec:i} is dedicated to the counterparts of Theorems~B and C for the case $\bullet = \imath$ (see Theorems~\ref{thm:SiCB} and \ref{thm:KiCB}).
In Section~\ref{sec:QSP} we show that the stabilization algebras coincide with the $\fgl$-variants $\UUj, \UUi$ of the multiparameter quantum symmetric pair coideal subalgebras studied by Bao-Wang-Watanabe in \cite{BWW18} (referred as $\mathbf{U}^\jmath, \mathbf{U}^\imath$ therein). The argument is made bypassing the idempotented (or modified) quantum algebras.
\begin{thmD}[Theorems~\ref{thm:Kj=Uj} and \ref{thm:Ki=Ui}]
There are algebra isomorphisms $\dKKj \simeq \dUUj, \dKKi \simeq \dUUi$.
\end{thmD}
In the appendix we provide an algebraic version of a type D Beilinson-Lusztig-MacPherson construction which is first introduced by Fan-Li from a geometric viewpoint.
%-----------------------------------------------------------------------------------------------------------

\vspace{2mm}
\noindent {\bf Acknowledgement.}
The authors thank Huanchen Bao and Weiqiang Wang for helpful discussions. We also thank Catherina Stroppel for bringing \cite{ES18} to our attention.
The second author is partially supported by the Science and Technology Commission of Shanghai Municipality (grant No. 18dz2271000) and the NSF of China (grant No. 11671108, 11871214).
We thank the referees for detailed comments on a previous version of the manuscript.

%%-----------------------------------------------------------------------------------------------------------
%=========================================================
\section{Combinatorics on Weyl groups}\label{sec:Weyl}
%=========================================================
\subsection{Weyl groups as permutation groups}
%-----------------------------------------------------------------------------------------------------------
Let $\NN = \{0,1,2,\ldots\}$. Fix $N,n,D,d\in\NN$ such that
\begin{equation}\label{NdDd}
N = 2n+1, D=2d+1.
\end{equation}
Let $\textup{Perm}(X)$ be the group of permutations on a set $X$.
Let $(W,S)$ be the Coxeter system of type B/C by
\bA{
\label{def:W}
&W = \{g \in \textup{Perm}([-d,d]) ~|~ g(-i) = -g(i)\},
\quad
S = \{s_0, \ldots, s_{d-1}\},}
where
\begin{equation}
s_0 = (-1, 1),
\quad
s_i = (i, i+1)(-i, -i-1)
\quad
(1\leq i<d).
\end{equation}
In particular, $g(0)=0$ for any $g\in W$.
The corresponding Coxeter diagram is as below:
\[
\begin{tikzpicture}[start chain]
\dnode{$0$}
\dnodenj{$1$}
\dydots
\dnode{$d-1$}
\path (chain-1) -- node[anchor=mid] {\(=\joinrel=\joinrel=\)} (chain-2);
\end{tikzpicture}
\]
Since that any $g \in W$ is uniquely determined by $(g(1), \ldots, g(d))$,
we use the two-line/one-line notations (referred as the window notation in \cite{BB05})
\eq
g \equiv \left|\ba{{ccc}1,&\ldots&,d\\g(1),&\ldots&,g(d)}\right|_\fc \equiv |g(1), \cdots, g(d)|_\fc.
\endeq
Let $\ell:W\to\NN$ be the length function on $W$.
We introduce a truncated length function $\ell_\fc:W \to \NN$ such that $\ell_\fc(g)$ equals to the total number of $s_0$'s in a reduced expression of $g$.
The function $\ell_\fc$ is well-defined since it is the weight function (cf. \cite{Lu03}) determined by $\ell_\fc(s_0) = 1, \ell_\fc(s_i) = 0$ for $i\ge 1$.
We set $\ell_\fa = \ell - \ell_\fc$.
%-----------------------------------------------------------------------------------------------------------
\lem\label{lem:l(g)}
For $g \in W$, we have
\bA{
\label{eq:lc(g)} &\ell_\fc(g)
= \frac{1}{2}{}^\sharp\left\{ (i,j) \in [1,d]\times \{0\} \middle| \substack{i< j\\ g(i) > g(j)} \textup{ or }\substack{i > j\\ g(i) < g(j)} \right\},
\\
\label{eq:la(g)} &\ell_\fa(g)
= \frac{1}{2}{}^\sharp\left\{ (i,j) \in [1,d]\times ([-d,d]-\{0\}) \middle| \substack{i< j\\ g(i) > g(j)} \textup{ or }\substack{i > j\\ g(i) < g(j)} \right\}.
\\
\label{eq:l(g)} &\ell(g)
= \frac{1}{2}{}^\sharp\left\{ (i,j) \in [1,d]\times [-d,d] \middle| \substack{i< j\\ g(i) > g(j)} \textup{ or }\substack{i > j\\ g(i) < g(j)} \right\}.
}
\endlem
%-----------------------------------------------------------------------------------------------------------
\proof
It follows by an easy induction that $\ell_\fc(g) = {}^\sharp\{ i\in [1, d] ~|~ g(i) < 0\}$, which yields to
\eqref{eq:lc(g)} by a direct calculation.
The formula \eqref{eq:l(g)} for $\ell(g)$ is equivalent to the formula \cite[(8.2)]{BB05}.
Then there comes the formula \eqref{eq:la(g)} by $\ell_\fa(g)=\ell(g)-\ell_\fc(g)$.
\endproof
%-----------------------------------------------------------------------------------------------------------
\rmk
The expressions in Lemma~\ref{lem:l(g)}  are not the most straight-forward.
There are simpler ones, for example, $\ell_\fa = \tinv +\tneg$ and $\ell_\fc = \tneg$ following the convention in \cite{BB05}.
We will see in Lemma~\ref{lem:l(A)} the advantage of choosing such symmetrized expressions.
See also \cite[Appendix A]{FL3Wb} for similar symmetrized length formulas for finite and affine classical types.
\endrmk
%-----------------------------------------------------------------------------------------------------------
Denote the set of weak compositions of $d$ of $n+1$ parts by
\eq\label{def:Ld}
\Ld_{n,d} = \{
\ld =(\ld_n, \ldots,\ld_{1}, 2\ld_{0}+1,\ld_{1}, \ldots, \ld_n) \in\NN^{2n+1}
~|~
\textstyle\sum_{i=0}^{n} \ld_i = d
\}.
\endeq
For any $\ld \in \Ld_{n,d}$ and integer $i\in[-n,n]$, we define integer intervals $R^\ld_i$ by
\eq\label{def:R}
R^\ld_i  = \bc{

[\ld_0+\sum\limits_{1\le j <i}\ld_j +1, \ld_0+\sum\limits_{1\le j \le i}\ld_j]&\tif 0<i \leq n;
\\
[-\ld_0, \ld_0]&\tif i= 0;
\\
-R^\ld_{-i} &\tif -n\leq i < 0.
}
\endeq

%The sets $\{R^\ld_i\}_{-n\leq i\leq n}$ partition the integer interval $[-d,d]$ as pictured below:
%\[
%\ba{{|c|c|c|c|c|c|c|c|}
%\hline
%R^\ld_{-n}&\cdots&R^\ld_{-1}&R^\ld_0&R^\ld_1&\cdots&R_n^\ld
%\\
%\hline
%-d \cdots&\cdots& -\ld_0 -\ld_1 \cdots -\ld_0 - 1&-\ld_0 \cdots \ld_0& \ld_0 + 1 \cdots \ld_0 + \ld_1&\cdots&\cdots d
%\\
%\hline
%}
%\]
For any subset $X \subset [-d,d]$, let $\textup{Stab}(X)$ be the stabilizer of $X$ in $W$.
A parabolic subgroup of $W$ must be of the form
\eq\label{eq:W_ld}
W_\ld = \bigcap_{i=0}^n \textup{Stab}(R^\ld_i),\quad \mbox{for some $\ld \in \Ld_{n,d}$.}
\endeq
Precisely, $W_\ld$ is the parabolic subgroup of $W$ generated by
$S- \{s_{\ld_0}, s_{\ld_0+\ld_1}, \ldots, s_{d-\ld_n}\}$.
Denote the set of shortest right coset representatives for $W_\ld\setminus W$ by
\bA{\label{def:Dld}
\D_\ld &= \{ w\in W ~|~ \ell(wg) = \ell(w) + \ell(g) \textup{~for all~} w\in W_\ld\}
\\
&=\{w \in W ~|~ w\inv \textup{~is order-preserving on all~}R_i^\ld\}.
}

%-----------------------------------------------------------------------------------------------------------
%\rmk\label{rmk:Dld}
%We remark that for $w\inv(a)$ keep record of the position of the integer $a$ in the one-line notation of $w$.
%Therefore, $w\inv$ is order-preserving on an integer interval $[a,b] \subset [-d,d]$ if and only if
%in the one-line notation of $w$, the number $a$ is to the left of the number $a+1$, which is to the left of the number $a+2$, ... and all the way up to $b$.
%\endrmk
%-----------------------------------------------------------------------------------------------------------
Denote the set of  minimal length double coset representatives
for $W_\ld \backslash W /W_\mu$ by
\eq  \label{def:Dldmu}
\D_{\ld\mu} = \D_\ld \cap \D_\mu^{-1}.
\endeq
%-----------------------------------------------------------------------------------------------------------
%\exa
%Let $n=1, N=3, d=3, D = 7$, $\ld = (2,3,2) \in \Ld_{1,3}$. We have
%\[
%R_{-1}^\ld = \{-3,-2\},
%\quad
%R_0^\ld = \{-1,0,1\},
%\quad
%R_1^\ld = \{2,3\}.
%\]
%The corresponding parabolic subgroup is
%\[
%W_\ld = \textup{Stab}([-1,1]) \cap \textup{Stab}([2,3]) = \langle s_0, s_2 \rangle.
%\]
%We have $w \in \D_\ld$ if and only if
%\[
%w\inv(-1)< w\inv(0) < w\inv(1),
%\quad
%w\inv(2) < w\inv(3).
%\]
%By Remark~\ref{rmk:Dld}, this is equivalent to that in the one-line notation of $w$:
%\enu
%\item The position of 2 is to the left of 3;
%\item The position of -1 is to the left of 0, which is to the left of 1.
%\endenu
%Therefore, we have
%\[
%\D_\ld = \left\{\ba{{c}
%|1,2,3|_\fc, |2,1,3|_\fc,|-2,1,3|_\fc,|1,-2,3|_\fc,|-3,-2,1|_\fc,|-3,1,-2|_\fc, \\
%|1,-3,-2|_\fc, |2,3,1|_\fc,|-2,3,1|_\fc,|1,3,-2|_\fc,|3,-2,1|_\fc,|3,1,-2|_\fc}\right\}.
%\]
%\endexa

In the following we collect some standard results for Coxeter groups from \cite[Proposition 4.16, Lemma ~4.17 and Theorem~ 4.18]{DDPW08}.
%-----------------------------------------------------------------------------------------------------------
\begin{lem}\label{lem:doublecoset}
Let $\ld,\mu \in \Ld_{n,d}$ and $g \in \D_{\ld\mu}$.
\enua
\item There exists $\delta \in \Ld_{n',d}$ for some $n'$ such that
$
W_{\delta} = g^{-1} W_\ld g \cap W_\mu.
$
\item The map $W_\ld \times (\D_\delta \cap W_\mu) \rw W_\ld g W_\mu$ sending $(x,y)$ to $xgy$ is a bijection;
moreover, we have $\ell(xgy) = \ell(x) + \ell(g) + \ell(y)$.
\item The map $W_\delta \times (\D_\delta\cap W_\mu) \rw W_\mu$ sending $(x,y)$ to $xy$ is a bijection;
moreover, we have $\ell(x) + \ell(y) = \ell(xy)$.
\endenua
\end{lem}
%-----------------------------------------------------------------------------------------------------------
An essential step in deriving the multiplication formula is to understand the set $\D_\delta \cap W_\mu$, which we will see in Section \ref{sec:DW}.
%-----------------------------------------------------------------------------------------------------------
\subsection{Set-valued matrices}
%-----------------------------------------------------------------------------------------------------------
Let
%Denote the set of $\NN$-valued $N\times N$ matrices by $\Tt_N$, and denote its subset of matrices whose entries add to $D$ by
\eq\label{def:Tt}
\Tt_{N,D} := \left\{ (a_{ij})_{-n \le i,j \le n} \in \textup{Mat}_{N\times N}(\NN)~\middle|~ \sum_{ij} a_{ij} = D\right\},
\quad
\Tt_N = \bigcup_{D\in2\NN+1} \Tt_{N,D}.
\endeq
Note that the columns/rows of such a matrix are indexed by $[-n, n]$ instead of $[1, N]$.
Let
%Define its subset of centro-symmetric matrices whose $(0,0)$-th entries are odd by
\eq\label{def:Xi}
\Xi_{n,d} := \left\{ (a_{ij})\in \Tt_{N,D}~\middle|~ \ba{{l}a_{00} \in 2\ZZ+1,\\ a_{ij} = a_{-i,-j}\quad\textup{for all}\quad i,j}\right\},
\quad
\Xi_n = \bigcup_{d\in\NN} \Xi_{n,d}.
\endeq

For $A=(a_{ij}) \in \Xi_{n,d}$ we define a matrix $A^\cP = (A^\cP_{ij})$ to be the unique set-valued matrix satisfying:
\enu
\item[(P0)] The sets $(A^\cP_{ij})_{ij}$ partition $[-d,d]$;
\item[(P1)] $|A^\cP_{ij}| = a_{ij}$ for all $i,j$;
\item[(P2)] Every element in $A^\cP_{ij}$ is smaller than any element in $A^\cP_{xy}$ if $(i,j) < (x,y)$ in the lexicographical order (i.e., $(i,j) < (x,y)$ if and only if $i<x$ or $(i=x, j < y)$).
\endenu
In words, the set-valued matrix $A^\cP$ is obtained by filling integers from $-d$ to $d$ into the entries $A^\cP_{ij}$ row-by-row, top-to-bottom.
For $T \in \Tt_N$, we define its row sum vector $\ro(T)= (\ro(T)_k)_{k=-n}^n$
and column sum vector $\co(T)= (\co(T)_k)_{k=-n}^n$
%and full column sum $\co(T)$
by
\eq\label{def:ro}
\ro(T)_k = \sum\limits_{-n\le j \le n} t_{kj}
\quad\textup{and}\quad
\co(T)_k = \sum\limits_{-n\le i \le n} t_{ik}.
\endeq
%-----------------------------------------------------------------------------------------------------------
\begin{lem}\label{lem:kappa}
The following map is bijective:
\eq
\kappa: \bigsqcup_{\ld,\mu\in\Ld_{n,d}} \{\ld\} \times \D_{\ld\mu} \times \{\mu\} \to \Xi_{n,d},
\quad
\kappa(\ld,g,\mu) = (|R_i^\ld \cap g R_j^\mu|)_{ij}.
\endeq
Moreover, the inverse is given by $\kappa\inv(A) = (\ro(A), g_A, \co(A))$, where $g_A$ is the permutation sending $k$ to the $k$-th number in the column-reading of $A^\cP$ (see Example \ref{ex:kappa} below).
\end{lem}
%-----------------------------------------------------------------------------------------------------------
\proof
The surjectivity follows from $\kappa(\ro(A), g_A, \co(A))=A\ (\forall A\in\Xi_{n,d})$ by a direct calculation.

For injectivity, we assume $\kappa(\ld,g,\mu) = A = \kappa(\ld',g',\mu')$.
Then $\ld =\ld' = \ro(A)$ and $\mu = \mu' = \co(A)$ and hence $g,g'\in\D_{\ld\mu}$. It follows from $|R^{\ld}_i \cap g R^{\mu}_j|=|R^{\ld}_i \cap g' R^{\mu}_j|\ (\forall i,j\in[-n,n])$ that $g = w_{(\ld)} g' w_{(\mu)}$ for some $w_{(\ld)} \in W_{\ld}, w_{(\mu)} \in W_{\mu}$.
Therefore $g = g'$ since they are both minimal double coset representatives in $W_\ld\backslash W / W_\mu$.
\endproof
%-----------------------------------------------------------------------------------------------------------
Thanks to Lemma~\ref{lem:kappa}, we define length functions $\ell, \ell_\fc, \ell_\fa$ on $\Xi_{n,d}$ by
\eq\label{def:l(A)}
\ell(A) = \ell(g),
\quad
\ell_\fc(A) = \ell_\fc(g),
\quad
\ell_\fa(A) = \ell_\fa(g)
\quad
(\textup{for }A = \kappa(\ld,g,\mu)).
\endeq
We define index subsets of type A/C by the following:
\eq\label{def:Ia}
I_\fa = (\{0\}\times[1,n]) \sqcup ([1,n] \times [-n,n]),
\quad
I_\fc = I_\fa \sqcup \{(0,0)\}.
\endeq
For $(i,j) \in I_\fc$, we set
\eq\label{def:anat}
a^\natural_{ij} = \bc{\frac{1}{2}(a_{ij}-1) &\tif (i,j) = (0,0);\\ a_{ij}&\otw.}
\endeq
There is an alternative length formula in terms of products of matrix entries as below.
%-----------------------------------------------------------------------------------------------------------

\begin{lem}\label{lem:l(A)}
Recall $a_{ij}^{\natural}$ from \eqref{def:anat}.
The (truncated) length functions of $A$ are given by
\bA{\label{eq:l(A)}
&\ell(A) = \dfrac{1}{2}
\bigg(
\sum\limits_{(i,j) \in I_\fc }
\Bp{
\sum\limits_{\substack{x < i\\ y > j}}
+
\sum\limits_{\substack{x > i\\ y < j}}
}
a^\natural_{ij} a_{xy}
\bigg),
\quad
\ell_\fc(A) =
\dfrac{1}{2}
\Bp{
\sum\limits_{\substack{0 < x\\ 0 > y}}
+
\sum\limits_{\substack{0 > x\\ 0 < y}}
}
a_{xy},
\\
&
\ell_\fa(A)
=
\dfrac{1}{2}
\bigg(
\sum\limits_{(i,j) \in I_\fc }
\Bp{
\sum\limits_{\substack{x < i\\ y > j}}
+
\sum\limits_{\substack{x > i\\ y < j}}
}
a^{\natural\natural}_{ij} a_{xy}
\bigg),
}
where $a_{00}^{\natural\natural} =a_{00}^{\natural}-1= \frac{1}{2}(a_{00}-3)$ and $a_{ij}^{\natural\natural} =a_{ij}$ if $(i,j) \in I_\fa$.
\end{lem}
%-----------------------------------------------------------------------------------------------------------
\proof
These three formulas are paraphrases of those in Lemma~\ref{lem:l(g)}.
\endproof

Let $A=\kappa(\ld,g,\mu)\in\Xi_{n,d}$. We define a signed weak composition as below:
\begin{equation}\label{delta1}
\delta(A)=(a_{nn},\ldots,\ldots,\ldots,a_{00}^{\natural},a_{10},\ldots,a_{n0},a_{-n,1},a_{-n+1,1},\ldots,a_{n1},\ldots,\ldots,a_{-n,n},a_{-n+1,n},\ldots,a_{nn}).
\end{equation}

A direct computation shows that $\delta(A)$ is indeed a weak composition $\delta$ in Lemma \ref{lem:doublecoset}(a).

%-----------------------------------------------------------------------------------------------------------
\exa\label{ex:kappa}
Let $A =  \bm{1&3&1\\1&1&1\\1&3&1}$. We have
\[
\ro(A) = (5,3,5),
\quad
\co(A) = (3,7,3),
\quad
A^\cP  = \bm{\{-6\}&\{-5,-4,-3\}&\{-2\}\\ \{-1\}&\{0\}&\{1\}\\ \{2\}&\{3,4,5\}&\{6\}}.
\]
Column-reading of $A^\cP$ gives us a sequence $-6,-1,2,-5,-4,-3,0,3,4,5,-2,1,6$, and hence $g_A$ is the permutation
\[
g_A = |3,4,5,-2,1,6|_\fc = s_1s_0s_2s_1s_3s_2s_4s_3.
\]
%Hence,
%\[
%g_A R_0^\mu = \{-5,-4,-3,0,3,4,5\},
%\quad
%g_A R_1^\mu = \{-2,1,6\},
%\]
%which are indeed the column unions of $A^\cP$.
Indeed, we have
\bAn{
\ell(A) &= \frac{1}{2}\Big(
a^\natural_{00} (1+1) + a_{01}(0+4) + a_{1,-1}(6+0) + a_{10}(2+0) + a_{11}(0+0)
\Big)
\\
&= \frac{1}{2}(0+4+6+6+0)= 8,
\\
\ell_\fc(A)&= \frac{1}{2}( a_{1,-1} + a_{-1,1}) = 1,
\\
\ell_\fa(A)&= \frac{1}{2}\Big(
a^{\natural\natural}_{00} (2) + a_{01}(4) + a_{1,-1}(6) + a_{10}(2) + a_{11}(0)
\Big) = 7.
}
Furthermore, $\delta(A)=(1,1,1,3,0,3,1,1,1)$.
\endexa
%-----------------------------------------------------------------------------------------------------------
\subsection{Quantum combinatorics}
%-----------------------------------------------------------------------------------------------------------
We denote the quantum $v$-number by
\eq
[a] =  \ds\frac{v^{2a}-1}{v^2-1}
\quad
(a\in\ZZ).
\endeq
We denote the type-A quantum $v$-factorials by, for
$t \in \NN$ $, A = (a_{ij}) \in \Tt_N$,
\eqa\label{def:factorial}
\quad
[t]! =  \prod_{k=1}^t [k],
\quad
[A]! = \prod_{-n\le i, j \le n} [a_{ij}]!
.
\endeqa
The type-B/C analogues are defined by, for
$t\in\NN, A = (a_{ij}), B = (b_{ij})\in \Xi_n$,
\eq\label{def:c-factorial}
[2t]_\fc = [t](u^2v^{2(t-1)}+1),
\quad
[t]^!_\fc =  \prod_{k=1}^t [2k]_\fc,
\quad
[A]^!_\fc =  [a^\natural_{00}]^!_\fc  \prod\limits_{(i,j)\in I_\fa} [a_{ij}]!.
\endeq
In particular, the specialization of $[2t]_\fc$ at  $u=v$ is $[t](1+v^{2t}) = [2t]$.
Furthermore, we set, for any $a\in\mathbb{Z}$ and $b\in\mathbb{N}$,
$$\left[\begin{array}{cc}a\\b\end{array}\right]=\prod_{i=1}^b\frac{v^{2(a-i+1)}-1}{v^{2i}-1}.$$

%Since that removing zeroes in a weak composition $\delta\in \Ld_{n,d}$ will not change the parabolic subgroup $W_\delta$ (e.g., $W_{(2,0,1,0,2)} = W_{(2,1,2)}$), we assume from now %on that the weak composition $\delta = \delta(A)$ in Lemma \ref{lem:doublecoset}~(a) is zero-free.
%We may assume that set $\{(i,j)~|~ a_{ij}\neq 0\}$ has size $2r+1$. Since it is partially ordered by the column-reading, i.e.,
%\eq\label{def:col-read}
%(i,j) \le_\co (x,y) \Lrw j \le y \textup{ or }(j=y, i\le x),
%\endeq
%there is a unique bijection $\pi:\{(i,j)~|~ a_{ij}\neq 0\} \to  [-r,  r]$ that respects the column-reading.
%\eq\label{def:pi}
%\endeq
%where $ (\af_{1i}, \ldots, \af_{k_i,i})$ is the unique subsequence obtained from $(\af_{1i}, \ldots, \af_{Ni})$ by removing all zeroes.
%-----------------------------------------------------------------------------------------------------------
\lem\label{lem:[A]!}
Let $A = \kappa(\mu,g,\nu)$, and let $\delta = \delta(A)$.
%$\pi$ be defined as above.
Then $\sum\limits_{w \in W_{\delta}} u^{2\ell_\fc(w)}v^{2\ell_\fa(w)} = [A]_\fc^!$.
%\enua
%\item $\delta(A) = (a_{\pi\inv(-r)}, a_{\pi\inv(-r+1)}, \ldots, a_{\pi\inv(r)})$.
%\item $R^{\delta(A)}_{\pi(i,j)} = g\inv R^\mu_i \cap R^\nu_j.$
%\item $\sum\limits_{w \in W_{\delta(A)}} u^{2\ell_\fc(w)}v^{2\ell_\fa(w)} = [A]_\fc^!$.
%\endenua
\endlem
%-----------------------------------------------------------------------------------------------------------
\proof
%This is similar to \cite[Lemma 2.10]{FL3W}.
Let $W^\fc_d$ be the Weyl group of type C$_d$.

Recall $\delta$ in \eqref{delta1}. We have
$W_\delta \simeq W^\fc_{a^\natural_{00}}\times \prod_{(i,j)\in I_\fa } \fS_{a_{ij}}$.
For each $w\in\fS_{a_{ij}}$ we have $\ell_\fc(w) = 0, \ell_\fa(w) = \ell(w)$, and hence
\eq
\sum_{w\in \fS_{a_{ij}}} u^{2\ell_\fc(w)}v^{2\ell_\fa(w)}
= \sum_{w \in \fS_{a_{ij}}} v^{2\ell(w)} = [a_{ij}]^!.
\endeq
Thus
\[
\sum_{w \in W_{\delta}} u^{2\ell_\fc(w)}v^{2\ell_\fa(w)}
=\Bp{\sum_{w\in W^\fc_{a^\natural_{00}}} u^{2\ell_\fc(w)}v^{2\ell_\fa(w)}}\prod\limits_{(i,j)\in I_\fa} [a_{ij}]!.
\]

It suffices to show that
\eq\label{eq:[d]^!_c}
\sum_{w\in W^\fc_d} u^{2\ell_\fc(w)}v^{2\ell_\fa(w)} = [d]^!_\fc.
\endeq
%Part (a)
%Part (b) follows from Lemma \ref{lem:doublecoset}(a) and \eqref{eq:W_ld}.
Let $\ld = (0, \ldots, 0, 1, 2d-1,1,0,\ldots,0) \in \Ld_{n,d}$.
We have $W_\ld \simeq W^\fc_{d-1}$, and hence
\eq\label{}
\sum_{w\in W^\fc_d} u^{2\ell_\fc(w)}v^{2\ell_\fa(w)} =
\Big(%\left(
	\sum_{w\in W^\fc_{d-1}} u^{2\ell_\fc(w)}v^{2\ell_\fa(w)}
\Big)%\right)
\Big(%\left(
\sum_{w\in \D_\ld} u^{2\ell_\fc(w)}v^{2\ell_\fa(w)}
\Big).%\right).
\endeq
By \eqref{def:Dld}, $g \in \D_\ld$ if and only if $g\inv$ is order-preserving on $[-d+1, d-1]$. Hence,
\eq
\D_\ld = \left\{
|i_1,\cdots,i_{d-1},\pm j|_\fc\inv
~\middle|~
\ba{{c} [1,d] = \{j\} \sqcup \{i_1, \ldots i_{d-1}\},
\\ i_1 < \ldots < i_{d-1}}
\right\}.
\endeq
Consequently, we have
\eq
\sum_{w\in \D_\ld} u^{2\ell_\fc(w)}v^{2\ell_\fa(w)}
= [d](1+u^2v^{2(d-1)}) = [2d]_\fc.
\endeq
Therefore, \eqref{eq:[d]^!_c} follows from a downward iteration. The Lemma is proved.
%A downward iteration shows that
%\eq
%\sum_{w\in W^c_d} u^{2\ell_\fc(w)}v^{2\ell_\fa(w)} = [d]^!_\fc.
%\endeq
%It is known that the Poincare polynomial of the symmetric group $\fS_d$ is
%\eq
%\sum_{w \in \fS_d} v^{2\ell(w)} = [d]^!_\fa.
%\endeq
%It follows from Lemma \eqref{lem:WdWmu} that $W_\delta \simeq W^\fc_{a'_{00}}\times \prod_{(i,j)\in I_\fa } \fS_{a_{ij}}$ and hence the lemma is proved.
\endproof
%-----------------------------------------------------------------------------------------------------------
%-----------------------------------------------------------------------------------------------------------
%\exa
%Let $A = \bM{1&2&\\2&3&2\\&2&1}$. We have
%\[
%\{(i,j)~|~ a_{ij} \neq 0\} = \{(-1,-1) < (0,-1) < (-1,0) < (0,0) < (1,0) < (0,1) < (1,1)\},
%\]
%and so
%\[
%\ba{{ccc}
%\pi\inv(-3) = (-1,-1),
%\quad
%\pi\inv(-2) = (0,-1),
%\quad
%\pi\inv(-1) = (-1,0).
%\\
%\pi\inv(1) = (1,0),
%\quad
%\pi\inv(2) = (0,1),
%\quad
%\pi\inv(3) = (1,1).
%}
%\]
%Hence $\delta(A) = (1,2,2,3,2,2,1)$, and
%\[
%R^\delta_0 = \{-1,0,1\},
%\quad
%R^\delta_1 = \{2,3\},
%\quad
%R^\delta_2 = \{4,5\},
%\quad
%R^\delta_3 = \{6\}.
%\]
%On  the other hand, we have
%\[
%\bM{R^\delta_{-3}&R^\delta_{-1}&\varnothing\\
%R^\delta_{-2}&R^\delta_0&R^\delta_2\\
%\varnothing&R^\delta_1&R^\delta_3}
%=
%\ba{{c|c:c:c}
%&
%\ba{{c}R^\mu_{-1}\\ = [-6,-4]}
%&
%\ba{{c}R^\mu_0\\ = [-3,3]}
%&
%\ba{{c}R^\mu_1\\ = [4,6]}
%\\
%\hline
%\ba{{c}g\inv R^\ld_1 \\= \{-6,-3,-2\}}
%&\{-6\}&[-3,-2]&\varnothing
%\\
%\hdashline
%\ba{{c}g\inv R^\ld_1 \\= \{0,\pm1,\pm4,\pm5\}}
%&[-5,-4]&[-1,1]&[4,5]
%\\
%\hdashline
%\ba{{c}g\inv R^\ld_1 \\= \{2,3,6\}}
%&\varnothing&[2,3]&\{6\}
%}
%\]
%and indeed that $R^\delta_t = g\inv R^\ld_i \cap R^\mu_j.$
%Also, we have $W_\delta = \langle s_0, s_2,s_4\rangle$ and
%\[
%\sum_{w \in W_\delta} u^{2\ell_\fc(w)} v^{2\ell_\fa(w)}
%= (1+u^2)(1+v^2)^2 = [2]_u[2]_v^2 = [A]^!_\fc.
%\]
%\endexa
%-----------------------------------------------------------------------------------------------------------

%=========================================================
\section{Schur algebras}\label{sec:schur}
%=========================================================
\subsection{Schur algebras}
%-----------------------------------------------------------------------------------------------------------
%Let $\Sq$ be the $q$-Schur algebra with linear basis $\{e_A ~|~ A \in \Tt\}$ where
%\eq
%\Tt = \Tt_{N,D} = \left\{ A=(a_{ij}) \in \textup{Mat}_{N\times N}(\NN)~\middle|~ \sum_{ij} a_{ij} = D\right\},
%\endeq
%and $e_A:x_\mu\HH \to x_\ld \HH$ is the right $\HH$-linear map uniquely determined by
%\eq
%e_A(x_\mu) = T_{W_\ld g W_\mu} = x_\ld T_{\D_{\delta(A)}\cap W_\mu}.
%\endeq
%Let $\bbA =\ZZ[u,u\inv,v,v^{-1}]$.
The Hecke algebra $\HH = \HH(W)$  %of type $\~{C}_d$
over $\bbA$ is an algebra with a basis $\{T_g ~|~ g\in W\}$ satisfying
\bA{
\label{def:Hecke}
&T_w T_{w'} = T_{ww'}
\quad\tif
\ell(ww') = \ell(w) + \ell(w'),
\\
\label{def:T0}&(T_{s_0}+1) (T_{s_0} - u^2) = 0,
\\
\label{def:Ts}&(T_s+1) (T_s - v^2) = 0
\quad\tfor
s \in S-\{s_0\}.
}
%Equivalently, we have
%\eq
%T_{s_0}\inv = u^{-2}T_{s_0} + u^{-2}-1,
%\quad
%T_{s_0}\inv = v^{-2}T_{s_0} + v^{-2}-1
%\quad (s \in S\backslash\{ s_0\}).
%\endeq
For any subset $X \subset W$ and for $\ld\in\Ld_{n,d}$ \eqref{def:Ld}, set
\eq  \label{eq:x}
T_X = \sum_{w\in X} T_w,
\quad
T_{\ld\mu}^g = T_{(W_\ld)g(W_\mu)},
\quad
x_\ld = T_{\ld\ld}^\id =T_{W_\ld},
\endeq
where $\id$ is the identity element of $W$.
%-----------------------------------------------------------------------------------------------------------
\lem\label{lem:wx}
If $w \in W_\ld$, then $T_w x_\ld = u^{2\ell_\fc(w)}v^{2\ell_\fa(w)} x_\ld = x_\ld T_w$.
\endlem
%-----------------------------------------------------------------------------------------------------------
\proof
This reduces to the case when $w = s \in S$. It then follows from the Hecke relation \eqref{def:Hecke}.
\endproof
%-----------------------------------------------------------------------------------------------------------
For $\ld,\mu\in\Ld_{n,d}$ and $g\in \D_{\ld\mu}$, we consider a right $\HH$-linear map
$
\phi_{\ld\mu}^g \in \Hom_\HH(x_\mu \HH, \HH)$,
sending $x_\mu$ to $T^g_{\ld\mu}.$
Thanks to Lemma~ \ref{lem:doublecoset}(b),
we have $T^g_{\ld\mu} = x_\ld T_g T_{\D_\delta \cap W_\mu}$ for some $\delta\in \Ld_{n',d}$,
and hence we have constructed a right $\HH$-linear map
\eq  \label{phi}
\phi_{\ld\mu}^g \in \Hom_\HH(x_\mu\HH, x_\ld\HH),
\qquad
T_{\mu\mu}^\id \mapsto T^g_{\ld\mu}.
\endeq
The {\em Schur algebra} $\SSj$ is defined as the following $\bbA$-algebra
\eq
\label{def:Sj}
\SSj = \textup{End}_{\HH}
\Bp{
\mathop{\oplus}_{\ld\in\Ld_{n,d}} x_\ld \HH
}
= \bigoplus_{\ld,\mu \in \Ld_{n,d}} \Hom_{\HH} (x_\mu \HH, x_\ld \HH)
.
%= \Span \left\{ \phi_{\ld\mu}^g %~|~ \ld,\mu \in \Ld, g \in \D_{\ld\mu}
%\right\}.
\endeq
Thanks to Lemma \ref{lem:kappa}, for $A = \kappa(\ld,g,\mu)$ we define
\eq\label{def:eA}
e_A = \phi_{\ld\mu}^g.
\endeq
A formal argument as in \cite{Du92,G97} is applicable to our setting and gives us the following:
\begin{lem}
The set $\{ e_A ~|~ A\in\Xi_{n,d} \}$
forms an $\bbA$-basis of $\SSj$.
\end{lem}
%-----------------------------------------------------------------------------------------------------------
%\subsection{Some}
%-----------------------------------------------------------------------------------------------------------
For $T = (t_{ij})\in \Tt_N$, let $\diag(T) = (\delta_{ij} t_{ij}) \in \Tt_N$ and denote its centro-symmetrizer by
\eq\label{def:tt}
T^\tt= (t^\tt_{ij}),\quad\mbox{where}\quad t^\tt_{ij} = t_{ij} + t_{-i,-j}.
\endeq
We remark that $T^\tt \not\in \Xi_{n}$ since $t^\tt_{00}$ is even.
A matrix $B \in \Xi_{n,d}$ is called a {\em Chevalley matrix} if
\eq\label{def:Chev}
B-\diag(B) =  bE^\tt_{h,h+1},
\quad
(b\in\NN, -n\le h < n).
\endeq
An easy consequence of Lemma \ref{lem:l(A)} is that $g_B =\id$ if $B$ is Chevalley.
We assume from now on that $B$ is a Chevalley matrix, and we fix
$B = \kappa(\ld, \id, \mu)$, $A = \kappa(\mu, g, \nu)$.
%The lemma below is obtained by rewriting combinatorial facts in Section 2.
%-----------------------------------------------------------------------------------------------------------
Recall $[A]^!_\fc$ from \eqref{def:c-factorial}. We have the following identity.
\lem  \label{lem:xTx}
$x_\mu T_{g} x_\nu = [A]^!_\fc \, e_A(x_\nu).$
\endlem
%-----------------------------------------------------------------------------------------------------------
\proof
Let $\delta = \delta(A)$. By Lemma~\ref{lem:doublecoset}(c), we have $x_\nu = x_\delta T_{\D_\delta \cap W_\nu}$, and hence
\eq
x_\mu T_g x_\nu = x_\mu T_g x_\delta T_{\D_\delta \cap W_\nu} = \sum_{w\in W_\delta} x_\mu T_g T_w T_{\D_\delta \cap W_\nu}.
\endeq
By Lemma~\ref{lem:doublecoset}(a), $w \in g\inv W_\mu g \cap W_\nu \subset W_\nu$ and hence $T_gT_w = T_{gw}$ since $g\in\D_{\mu\nu} \subset \D_\nu\inv$.
Moreover, we have $gw = w' g$ for some $w' \in W_\mu$.
Since $g\in\D_{\mu\nu} \subset \D_\mu$, we have
\begin{equation}\label{ell}
\ell(g)+\ell(w)=\ell(gw)=\ell(w'g) = \ell(w') + \ell(g)
\end{equation}
and therefore $\ell(w') = \ell(w)$. Moreover, note that $\ell_\fc$ is a well-defined weight function (cf. \cite{Lu03}) determined by $\ell(s_0)=1$ and $\ell(s_i)=0 \ (i\geq1)$. Counting the number of $s_0$ appeared in a reduced form of $gw=w'g$, we have $\ell_\fc(gw)=\ell_\fc(g)+\ell_\fc(w)$ and $\ell_\fc(w'g) = \ell_\fc(w') + \ell_\fc(g)$ by \eqref{ell}. Thus $\ell_\fc(w)=\ell_\fc(w')$ (and hence $\ell_\fa(w)=\ell_\fa(w')$).
Finally, we have
\eq
\sum_{w\in W_\delta} x_\mu T_{gw}
=
\sum_{w\in W_\delta} x_\mu T_{w'}T_g
=
\sum_{w\in W_\delta}  u^{2\ell_\fc(w)}v^{2\ell_\fa(w)} x_\mu T_g
=
[A]^!_\fc x_\mu T_g,
\endeq
where the second equality follows from Lemma~\ref{lem:wx}, while the third equality follows from Lemma~\ref{lem:[A]!}.
The rest follows by the definition  $e_A(x_\nu)=x_\mu T_g T_{\D_\delta \cap W_\nu}$.
\endproof
\subsection{Multiplication formulas $\D_\delta \cap W_\mu$}  \label{sec:DW}
%-----------------------------------------------------------------------------------------------------------
%-----------------------------------------------------------------------------------------------------------
\begin{lem}\label{lem:multAw}
Fix
$B = \kappa(\ld, \id, \mu)$, $A = \kappa(\mu, g, \nu)$ and let $\delta = \delta(B)$.
Let $y^w$ be the shortest double coset representative for $W_\ld wg W_\nu$, and set $A^w = \kappa(\ld,y^w,\nu)$. Then
\eq\label{eq:multAw}
e_B e_A
=   \sum_{w\in \D_\delta \cap W_\mu} \frac{[A^w]_\fc^!}{[A]_\fc^!}
(u^2)^{\ell_\fc(w) + \ell_\fc(g) - \ell_\fc(y^w)}
(v^2)^{\ell_\fa(w) + \ell_\fa(g) - \ell_\fa(y^w)}
e_{A^w}.
\endeq
\end{lem}
%-----------------------------------------------------------------------------------------------------------
\proof
By Lemma \ref{lem:xTx} and \eqref{phi} (which implies $e_B(x_\mu)=x_\ld T_{\D_\delta \cap W_\mu}$) we see that
\eq\label{eq:mult1}
e_B e_A(x_\nu)
= e_B\Big(\frac{1}{[A]_\fc^!} x_\mu T_g x_\nu\Big)
= \frac{1}{[A]_\fc^!} e_B(x_\mu) T_g x_\nu
=  \frac{1}{[A]_\fc^!} x_\ld T_{\D_\delta \cap W_\mu} T_g x_\nu.
\endeq
Since $g \in \D_{\mu\nu} \subset \D_\mu$, so $T_w T_g = T_{wg}$ for all $w \in \D_\delta \cap W_\mu \subset W_\mu$.
For $w \in \D_\delta \cap W_\mu$,
%the element $wg$ lies in a unique double coset $W_\ld y^{w} W_\nu$ where $y^w \in \D_{\ld\nu}$. Namely,
there exists $w_\ld \in W_\ld, w_\nu \in W_\nu$ such that $wg = w_\ld y^w w_\nu$. Moreover, we have
\eq\label{eq:y^w}
\ell(wg) = \ell(w) + \ell(g) = \ell(w_\ld) +\ell(y^w) + \ell(w_\nu).
\endeq
Thus, we have
\eq\label{eq:xTx}
x_\ld T_{wg} x_\nu
= x_\ld T_{w_\ld} T_{y^w} T_{w_\nu} x_\nu
= (u^2)^{\ell_\fc(w_\ld) + \ell_\fc(w_\nu)}(v^2)^{\ell_\fa(w_\ld) + \ell_\fa(w_\nu)} x_\ld T_{y^w} x_\nu.
\endeq
%We also define
%\eq\label{eq:A^w}
%A^w = \kappa(\ld,y^w,\nu).
%\endeq
Combining the \eqref{eq:mult1}, \eqref{eq:xTx} and applying Lemma \ref{lem:xTx} on $x_\ld T_{y^w} x_\nu$, we have
\eq
e_B e_A(x_\nu)
=  \frac{1}{[A]_\fc^!} \sum_{w\in \D_\delta \cap W_\mu} x_\ld T_{wg} x_\nu
= \sum_{w\in \D_\delta \cap W_\mu} \frac{[A^w]_\fc^!}{[A]_\fc^!}
(u^2)^{\ell_\fc(wg) - \ell_\fc(y^w)}
(v^2)^{\ell_\fa(wg) - \ell_\fa(y^w)} e_{A^w} (x_\nu).
\endeq

The lemma follows from \eqref{eq:y^w}.
\endproof

\begin{prop}\label{prop:multformula1}
Suppose that $A, B, C\in\Xi_{n,d}$ and $h\in[1,n]$.

\enu
\item[(1)] If $B-bE_{h,h-1}^\theta$ is diagonal, $\co(B)=\ro(A)$, then
\begin{equation}
e_B e_A=\sum_{t}v^{2\sum_{k<l}t_la_{h,k}}\prod_{l=-n}^{n}\left[\begin{array}{cc}a_{h,l}+t_l\\t_l\end{array}\right] e_{\widecheck{A}_{t,h}},
\end{equation} where $t=(t_i)_{-n\leq i\leq n}\in\NN^N$ with $\sum_{i=-n}^n t_i=b$ such that
$\bc{
t_i\leq a_{h-1,i} & \tif h>1;
\\
t_i+t_{-i}\leq a_{h-1,i} &\tif h=1,
}$
%$$\left\{\begin{array}{ll}
%t_i\leq a_{h-1,i} & \mbox{if~} h>1;\\
%t_i+t_{-i}\leq a_{h-1,i} & \mbox{if~} h=1,
%\end{array}
%\right.$$
and $$
\widecheck{A}_{t,h}=A+\sum_{l=-n}^{n}t_lE_{h,l}^\theta-\sum_{l=-n}^{n}t_lE_{h-1,l}^\theta.$$

\item[(2)] Suppose $C-cE_{h-1,h}^\theta$ is diagonal and $\co(C)=\ro(A)$.
If $h\neq1$, then
\begin{equation}
e_C e_A=\sum_{t}v^{2\sum_{k>l}t_la_{h-1,k}}\prod_{l=-n}^{n}\left[\begin{array}{cc}a_{h-1,l}+t_l\\t_l\end{array}\right] e_{\widehat{A}_{t,h}},
\end{equation} where $t=(t_i)_{-n\leq i\leq n}\in\NN^N$ with $\sum_{i=-n}^n t_i=c$ such that $t_i\leq a_{h,i}$, and
\[
\widehat{A}_{t,h}=A-\sum_{l=-n}^{n}t_lE_{h,l}^\theta+\sum_{l=-n}^{n}t_lE_{h-1,l}^\theta.
\]
If $h=1$, then
\eq\label{ecea}
e_C e_A=
\sum_{t}
u^{2\sum_{l<0}t_l}
v^{2\sum_{k>l}
a_{0,k}t_l+2\sum_{l<k<-l}t_lt_k+\sum_{l<0}t_l(t_l-3)}
\frac{[a_{0,0}^\natural+t_0]_\fc^!}{[a_{0,0}^\natural]_\fc^![t_0]!}\prod_{l=1}^n\frac{[a_{0,l}+t_l+t_{-l}]!}{[a_{0,l}]![t_l]![t_{-l}]!}e_{\widehat{A}_{t,1}},
\endeq
where $t=(t_i)_{-n\leq i\leq n}\in\NN^N$ with $\sum_{i=-n}^n t_i=c$ such that $t_i\leq a_{1,i}$.
\endenu
\end{prop}

\proof
For Part (1), we only present the proof for the most complicated case $h=1$.
Let $\delta=\delta(B)$ and take any $t=(t_i)_{-n\leq i\leq n}\in\NN^N$ as in the assumptions.
Among those $w\in\D_\delta \cap W_\mu$ such that $A^{w}=\widecheck{A}_{t,1}$, there is a unique shortest element $w_t$ with
\begin{equation}\label{eq1}
\begin{split}
\ell(w_t)
%&=\sum\limits_{\substack{k>l\geq 0 \\ \text{or} \\ k\geq-l>0}}t_{l}(a_{0,k}-t_k)+\sum\limits_{|k|<-l}t_l(a_{0,k}-t_k-t_{-k})-\sum\limits_{l<0}\frac{(t_l-1)t_l}{2}
%\\
&=\sum\limits_{k>l}(a_{0,k}-t_k)t_l-\sum\limits_{l<k<-l}t_lt_k-\frac{1}{2}\sum\limits_{l<0}t_l(t_l-1).
\end{split}
\end{equation}
In particular, we have
\begin{equation}\label{eq:fcwt}
\ell_{\fc}(w_t)=\sum_{l<0}t_l,
\quad
%\end{equation}
%Hence
%\begin{equation}\label{eq:fawt}
\ell_{\fa}(w_t)=\sum_{k>l}(a_{0,k}-t_k)t_l-\sum_{l<k<-l}t_lt_k-\frac{1}{2}\sum_{l<0}t_l(t_l+1).
\end{equation}
By a combinatorial argument, we calculate that
\begin{eqnarray*}
\sum_{\substack{w\in\D_{\delta} \cap W_{\mu},\\A^{w}=\widecheck{A}_{t,1}}}
u^{2\ell_\fc(w)}v^{2\ell_\fa(w)}
=
u^{2\ell_\fc(w_t)}v^{2\ell_\fa(w_t)}
\left(\sum_{x+y=t_{0}}
\lrb{a^\natural_{00}}{x}
\lrb{a^\natural_{00}-x}{y}u^{2x}
(v^2)^{\frac{x(x-1)}{2}+x(a^\natural_{00}-t_{0})}\right)\prod_{l=1}^n\lrb{a_{0l}}{t_l}\lrb{a_{0l}-t_l}{t_{-l}}.
\end{eqnarray*}
Note that
\[
\begin{split}
&\sum_{x+y=t_{0}}
\lrb{a^\natural_{0,0}}{x}
\lrb{a^\natural_{0,0}-x}{y}
u^{2x}(v^2)^{\frac{x(x-1)}{2}+x(a^\natural_{0,0}-t_{0})}
\\
&=\lrb{ a^\natural_{0,0}}{t_{0}}
\sum_{x=0}^{t_{0}}
\lrb{t_{0}}{x}
v^{x(x-1)} (uv^{a^\natural_{0,0}-t_{0}})^{2x}
\stackrel{(\diamondsuit)}{=}\lrb{ a^\natural_{0,0}}{t_{0}}
\prod_{i=1}^{t_{0}}(1+v^{2(i-1)}u^2v^{2(a^\natural_{0,0}-t_{0})})
= \frac{[a_{0,0}^\natural]^!_\fc}{[a_{0,0}^\natural - t_{0}]^!_\fc [t_{0}]!},
\end{split}
\]
where ($\diamondsuit$) is due to the quantum binomial theorem
$
\sum_{x=0}^m
\lrb{m}{x}
v^{x(x-1)}z^x
=\prod_{i=0}^{m-1}(1+v^{2i}z).
$
Therefore
\begin{equation}\label{eq:v2lw}
\sum_{w\in\D_{\delta} \cap W_{\mu},A^{w}=\widecheck{A}_{t,1}}u^{2\ell_\fc(w)}v^{2\ell_\fa(w)}=u^{2\ell_\fc(w_t)}v^{2\ell_\fa(w_t)}\frac{[a_{0,0}^\natural]^!_\fc}{[a_{0,0}^\natural - t_{0}]^!_\fc [t_{0}]!}\prod_{l=1}^n\lrb{a_{0,l}}{t_l}\lrb{a_{0,l}-t_l}{t_{-l}}.
\end{equation}
Furthermore, it follows from Lemma \ref{lem:l(A)} that
\begin{align}
\label{eq:dfc}
\ell_\fc(A)-\ell_\fc(\widecheck{A}_{t,1})
&=-\sum_{l<0}t_l
\\
\label{eq:dfa}
\ell_\fa(A)-\ell_\fa(\widecheck{A}_{t,1})
%&=-\sum_{k>l}(a_{0,k}-t_k)t_l+\frac{1}{2}\sum_{l<0}t_l+\sum_{k<l}t_la_{1,k}+\frac{1}{2}\sum_{k<-l}t_lt_k
%\\
&=\sum_{k<l}t_la_{1,k}-\sum_{k>l}(a_{0,k}-t_k)t_l+\sum_{l<k<-l}t_lt_k+\frac{1}{2}\sum_{l<0}t_l(t_l+1).
\end{align}
%\begin{equation}\label{eq:dfc}
%\ell_\fc(A)-\ell_\fc(\widecheck{A}_{t,1})=-\sum_{l<0}t_l
%\end{equation}
%and
%\begin{align}
%&\ell_\fa(A)-\ell_\fa(\widecheck{A}_{t,1})=-\sum_{k>l}(a_{0,k}-t_k)t_l+\frac{1}{2}\sum_{l<0}t_l+\sum_{k<l}t_la_{1,k}+\frac{1}{2}\sum_{k<-l}t_lt_k\nonumber\\
%&=\sum_{k<l}t_la_{1,k}-\sum_{k>l}(a_{0,k}-t_k)t_l+\sum_{l<k<-l}t_lt_k+\frac{1}{2}\sum_{l<0}t_l(t_l+1).\label{eq:dfa}
%\end{align}

Part (1) then follows from combining \eqref{eq:multAw}, \eqref{eq:fcwt}--\eqref{eq:dfa}.
% we finally obtain that
%\begin{equation*}
%e_B e_A=\sum_{t}v^{2\sum_{k<l}t_la_{1,k}}\prod_{l=-n}^{n}\left[\begin{array}{cc}a_{1,l}+t_l\\t_l\end{array}\right] e_{\widecheck{A}_{t,1}}.
%\end{equation*}
For Part (2), we only present a proof for the most complicated case that $h=1$.
Let $\delta=\delta(C)$ and take any $t=(t_i)_{-n\leq i\leq n}\in\NN^N$ as in the assumptions. Among those $w\in\D_{\delta} \cap W_{\mu}$ such that $A^{w}=\widehat{A}_{t,1}$, there is a shortest element $w_t$ with
\begin{equation}\label{eq:lfalfc}
\ell_\fc(w_t)=0 \quad \mbox{and} \quad \ell_{\fa}(w_t)=\sum_{k<l}t_l(a_{1,k}-t_k).
\end{equation}
Direct computation yields to the following identities:

\begin{align}
\label{eq4}
\sum_{\substack{w\in\D_{\delta} \cap W_{\mu},\\A^{w}=\widehat{A}_{t,1}}}
u^{2\ell_\fc(w)}v^{2\ell_\fa(w)}
&=
u^{2\ell_{\fc}(w_t)}v^{2\ell_\fa(w_t)}\prod_{l=-n}^n\lrb{a_{1,l}}{t_l}=v^{2\sum_{k<l}t_l(a_{1,k}-t_k)}\prod_{l=-n}^n\lrb{a_{1,l}}{t_l},
\\
\ell_\fc(A)-\ell_\fc(\widehat{A}_{t,1})
&=\sum_{l<0}t_l,
\\
\label{eq:eq5}
\ell_\fa(A)-\ell_\fa(\widehat{A}_{t,1})
%&=\sum_{k>l}a_{0,k}t_l-\frac{3}{2}\sum_{l<0}t_l-\sum_{k<l}t_l(a_{1,k}-t_k)+\frac{1}{2}\sum_{k<-l}t_lt_k
%\\
&=\sum_{k>l}a_{0,k}t_l-\sum_{k<l}t_l(a_{1,k}-t_k)+\sum_{l<k<-l}t_lt_k+\frac{1}{2}\sum_{l<0}t_l(t_l-3).
%\nonumber
\end{align}
%\begin{equation}\label{eq4}
%\sum_{\substack{w\in\D_{\delta} \cap W_{\mu},\\A^{w}=\widehat{A}_{t,1}}}
%u^{2\ell_\fc(w)}v^{2\ell_\fa(w)}=
%u^{2\ell_{\fc}(w_t)}v^{2\ell_\fa(w_t)}\prod_{l=-n}^n\lrb{a_{1,l}}{t_l}=v^{2\sum_{k<l}t_l(a_{1,k}-t_k)}\prod_{l=-n}^n\lrb{a_{1,l}}{t_l}.
%\end{equation}
%\begin{align}
%\ell_\fc(A)-\ell_\fc(\widehat{A}_{t,1})=\sum_{l<0}t_l.
%\end{align}
%\begin{align}\label{eq:eq5}
%\ell_\fa(A)-\ell_\fa(\widehat{A}_{t,1})
%&=\sum_{k>l}a_{0,k}t_l-\frac{3}{2}\sum_{l<0}t_l-\sum_{k<l}t_l(a_{1,k}-t_k)+\frac{1}{2}\sum_{k<-l}t_lt_k\\
%&=\sum_{k>l}a_{0,k}t_l-\sum_{k<l}t_l(a_{1,k}-t_k)+\sum_{l<k<-l}t_lt_k+\frac{1}{2}\sum_{l<0}t_l(t_l-3).\nonumber
%\end{align}
Part (2) then follows from combining \eqref{eq:multAw}, \eqref{eq:lfalfc}--\eqref{eq:eq5}.
\endproof

\begin{rem}\label{specialize}
These explicit formulas match the ones in \cite{BKLW18} (resp. the unsigned ones in \cite{FL15}) if we specialize $u=v$ (resp. $u=1$).
\end{rem}

%-----------------------------------------------------------------------------------------------------------
%=========================================================
\section{Canonical bases}\label{sec:bases}
%=========================================================
%-----------------------------------------------------------------------------------------------------------
\subsection{The bar involution}
%-----------------------------------------------------------------------------------------------------------
There is an $\bbA$-algebra involution $\bar{\empty}:\HH \rw \HH$,
which sends $u\mapsto u\inv, v \mapsto v^{-1}, T_w \mapsto T_{w^{-1}}^{-1}$, for all $w\in W$.
In particular, we have, for $s \in S - \{s_0\}$,
\eq\label{eq:Tbar}
\={T_{s}} = v^{-2} T_{s} + v^{-2} - 1,
\quad
\={T_{s_0}} = u^{-2} T_{s_0} + u^{-2} - 1.
\endeq

For $\ld,\mu \in \Ld_{n,d}$ (see \eqref{def:Ld}),
let $g^+_{\ld\mu}$ be the longest element in the double coset $W_\ld g W_\mu$ for $g \in \D_{\ld\mu}$,
and let $w_\circ^\mu = \id_{\mu\mu}^+$ be the longest element in the parabolic subgroup $W_\mu = W_\mu \id W_\mu$.
The lemma below is standard (cf. \cite[Corollary 4.19]{DDPW08}).
%-----------------------------------------------------------------------------------------------------------
\begin{lem}\label{lem:WgW}
Let $A = \kappa(\ld,g,\mu)$, $\delta = \delta(A)$.
Then:
\enua
\item $g_{\ld\mu}^+ = w_\circ^\ld g w_\circ^{\delta} w_\circ^\mu$, and
$\ell(g_{\ld\mu}^+) =  \ell(w_\circ^\ld) +  \ell(g)  - \ell(w_\circ^{\delta}) +  \ell(w_\circ^\mu).$
\item $W_\ld g W_\mu = \{w \in W ~|~ g \le w \le g^+_{\ld\mu}\}$.
%\item
%$T^g_{\ld\mu}
%=  u^{\ell_\fc(g^+_{\ld\mu})} v^{\ell_\fa(g^+_{\ld\mu})} C^g_{\ld\mu}
%+ \sum\limits_{\substack{y\in \D_{\ld\mu}\\
% y < g }} c^{(\ld,\mu)}_{y,g} C^y_{\ld\mu}$,
%for $c^{(\ld,\mu)}_{y,g}\in \bbA$.
%In particular,
%$x_\mu = u^{\ell_\fc(w_\circ^\mu)}v^{\ell_\fa(w_\circ^\mu)} C'_{w_\circ^\mu}$.
\endenua
\end{lem}
%-----------------------------------------------------------------------------------------------------------
Following \cite{KL79}, denote by $\{C'_w\}$ the Kazhdan-Lusztig $\ZZ[v, v\inv]$-basis of the Hecke algebra $\HH|_{u=v}$
characterized by Conditions (C1)--(C2) below:
\itm
\item[(C1)]
$C'_w$ is bar-invariant;

\item[(C2)]
$C'_w =  v^{-\ell(w)}\sum_{y \le w} P_{yw}(v) T_y$.
\enditm
Here $\le$ is the (strong) Bruhat order, and $P_{yw}$ is the Kazhdan-Lusztig polynomial satisfying that $P_{ww}=1$ and $P_{yw}\in \ZZ[v^2]$ with $ \deg_v P_{yw} \le \ell(w)-\ell(y)-1$ for $y<w$.
Recall $T^g_{\ld\mu}$ from \eqref{eq:x} and denote
\eq
C^g_{\ld\mu} = C'_{g^+_{\ld\mu}}
\quad (g \in \D_{\ld\mu}, \ld,\mu \in \Ld_{n,d}).
\endeq
Following  \cite{Cur85}, let $\bH_{\ld\mu}$ be the $\ZZ[v, v\inv]$-submodule of $\left.\HH\right|_{u=v}$ with basis $\{ T_{\ld\mu}^g\}_{g \in \D_{\ld\mu}}$.
It is shown in {\it loc. cit.} that $\{C^g_{\ld\mu}\}_{g\in \D_{\ld\mu}}$ also forms a bar-invariant basis of $\bH_{\ld\mu}$.

%\eq
%T^g_{\ld\mu}
%\in  v^{\ell(g^+_{\ld\mu})} C^g_{\ld\mu}
%+ \sum\limits_{\substack{y\in \D_{\ld\mu}\\ y < g }} \QQ(v) C^y_{\ld\mu}.
%\endeq

It is shown in \cite[\S 5]{Lu03} that, for any weight function $\bL:W \to \NN$, there exists a bar-invariant basis $\{\CL_w\}$ (referred as $c_w$ therein) at the specialization $u = \bv^{\bL(s_0)}, v = \bv^{\bL(s_1)}$, given by
\eq\label{def:CL}
\CL_w =
u^{-\ell_\fc(w)}v^{-\ell_\fa(w)} \sum_{y \le w} p_{y,w}(\bv) \left. T_y
\right|_{u = \bv^{\bL(s_0)}, v = \bv^{\bL(s_1)}},
\endeq
where $p_{y,w}(\bv)$ is an analogue of Kazhdan-Lusztig polynomial.
%\rmk
%The existence of bar-invariant basis and Kazhdan-Lusztig polynomials over $\bbA$ without specialization are not known to the author.
%\endrmk
For $\ld,\mu \in \Ld_{n,d}$, let $\HH_{\ld\mu}$ be the $ \ZZ[u^{\pm2},v^{\pm2}] $-submodule of $\HH$ with basis $\{ T_{\ld\mu}^g\}_{g \in \D_{\ld\mu}}$.
It follows from \cite[Lemma~2.10]{CIK72} and Lemma~\ref{lem:wx} that $\HH_{\ld\mu}$ can be characterized  as below:
\eq\label{eq:HIJ2}
\HH_{\ld\mu} =
\left\{ h \in \HH \middle|
\begin{array}{l}
T_{w} h = u^{2\ell_\fc(w)}v^{2\ell_\fa(w)} h, (\forall w\in W_\ld), \\
h T_{w'}=u^{2\ell_\fc(w')}v^{2\ell_\fa(w')}h, (\forall w'\in W_\mu)
\end{array}
\right\}.
\endeq
Below we show that the bar involution is closed on $\HH_{\ld\mu}$ although lacking of bar-invariant basis.
%-----------------------------------------------------------------------------------------------------------
\lem\label{lem:bar}
Let $A = \kappa(\ld,g,\mu)$. Then $\={T^g_{\ld\mu}} \in \HH_{\ld\mu}$.
In particular,
\eq
\={T^g_{\ld\mu}}  \in  u^{-2\ell_\fc(g^+_{\ld\mu})}v^{-2\ell_\fa(g^+_{\ld\mu})} T^g_{\ld\mu}
+ \sum\limits_{\substack{y\in \D_{\ld\mu}\\ y < g }}  \ZZ[u^{\pm2},v^{\pm2}]  T^y_{\ld\mu}.
\endeq
Moreover, $u^{-\ell_\fc(w_\circ^\mu)} v^{-\ell_\fa(w_\circ^\mu)} x_\mu$ is bar-invariant.
\endlem
%-----------------------------------------------------------------------------------------------------------
\proof
First, we show that $\={x_\nu} \in \bbA x_\nu$ for all $\nu \in \Ld_{n,d}$ via bar-invariant basis $\CL_w$.
Let $\HH^{\bL}_{\ld\mu}$ be the specialization of $\HH_{\ld\mu}$ at $u = \bv^{\bL(s_0)}, v = \bv^{\bL(s_1)}$.
From \eqref{eq:HIJ2}, a direct calculation shows that
$\CL_{w_\circ^\nu} \in \HH^{\bL}_{\nu\nu}$ and hence
\eq\bsp{
\CL_{w_\circ^\nu}
&{}= u^{-\ell_\fc(w_\circ^\nu)}v^{-\ell_\fa(w_\circ^\nu)} \sum_{y \le w_\circ^\nu} p_{y,w_\circ^\nu}
\left.T_y
\right|_{u = \bv^{\bL(s_0)}, v = \bv^{\bL(s_1)}}
\in \sum_{g\in \D_{\nu\nu}} \ZZ(\bv^{\pm\bL(s_0)},\bv^{\pm\bL(s_1)})
\left. T^g_{\nu\nu}
\right|_{u = \bv^{\bL(s_0)}, v = \bv^{\bL(s_1)}}.
}\endeq
Upon comparing coefficients, we obtain
\eq
\CL_{w_\circ^\nu} =u^{-\ell_\fc(w_\circ^\nu)}v^{-\ell_\fa(w_\circ^\nu)}
\left. T^\id_{\nu\nu}
\right|_{u = \bv^{\bL(s_0)}, v = \bv^{\bL(s_1)}}.
\endeq
Note that $x_\nu = T_{\nu\nu}^\id$.
Hence, for any weight function $\bL$, we have
\eq
\left.
(\={x_\nu} -u^{-2\ell_\fc(w_\circ^\nu)}v^{-2\ell_\fa(w_\circ^\nu)} x_\nu)
\right|_{u = \bv^{\bL(s_0)}, v = \bv^{\bL(s_1)}}
= 0.
\endeq
Therefore $\={x_\nu}  = u^{-2\ell_\fc(w_\circ^\nu)} v^{-2\ell_\fa(w_\circ^\nu)} x_\nu$.
We now show that $\={T_{\ld\mu}^g} \in \HH_{\ld\mu}$.
By Lemma~\ref{lem:xTx}, we have
$
T_{\ld\mu}^g \in \ZZ[u^{\pm2},v^{\pm2}]  x_\ld T_{g} x_\mu,
$
and hence
\eq
\={T_{\ld\mu}^g} \in  \ZZ[u^{\pm2},v^{\pm2}]   \={x_\ld} \={T_{g}} \={x_\mu}
= \sum_{z \le g}  \ZZ[u^{\pm2},v^{\pm2}]  x_\ld T_z x_\mu
.
\endeq
Similar to \eqref{eq:y^w}, we have $x_\ld T_z x_\mu \in  \ZZ[u^{\pm2},v^{\pm2}]  x_\ld T_y x_\mu$ for some $y \in \D_{\ld\mu}$ such that $y\le z$.
Finally, we have $\={T_{\ld\mu}^g} \in \sum_{y \in \D_{\ld\mu}}  \ZZ[u^{\pm2},v^{\pm2}]  x_\ld T_y x_\mu  \subseteq \HH_{\ld\mu}$.
The leading coefficient is obtained by a lengthy calculation which we omit.
\endproof
%-----------------------------------------------------------------------------------------------------------
%-----------------------------------------------------------------------------------------------------------
%-----------------------------------------------------------------------------------------------------------
The bar involution $\bar{\empty}$ on $\SSj$ is defined as follows:
for each $f \in \textup{Hom}_{\HH}(x_\mu \HH, x_\ld \HH)$, let $\={f}\in\textup{Hom}_{\HH}(x_\mu \HH, x_\ld \HH)$ be the
$\HH$-linear map which sends $x_\mu$ to $\={f(\={x_\mu})}$.
%-----------------------------------------------------------------------------------------------------------
%Particularly, the map $\={e_A}$ (for $A = \kappa(\ld,g,\mu)$) is then determined by
%\eq
%\={e_A}(x_\mu) = u^{-2\ell_\fc(w_\circ^\nu)} v^{-2\ell_\fa(w_\circ^\nu)} \={T^g_{\ld\mu}}.
%\endeq

%Let $\le$ be the (strong) Bruhat order on $W$.
%
%\enditm
%

%-----------------------------------------------------------------------------------------------------------

%, that is,
%\[
%\={f}(x_\mu H) = u^{2\ell_\fa(w_\circ^\mu)}v^{2\ell_\fc(w_\circ^\mu)}\={f(x_\mu)} H,
%\quad
%\tfor
%H \in \HH.
%\]
%For $A=\kappa(\ld,g,\mu) \in \Xi_n$,
%we have $e_A(x_\mu)  =T^g_{\ld\mu}$; see \eqref{phi} and \eqref{def:eA}.
%Hence by Lemma~\ref{lem:Cur}(c) we have
%\eqnarray
%e_A(C'_{w_\circ^\mu})
%&\ds =  u^{\ell_\fc(g^+_{\ld\mu})-\ell_\fc(w_\circ^\mu)}v^{\ell_\fa(g^+_{\ld\mu})-\ell_\fa(w_\circ^\mu)} C^g_{\ld\mu}
% + \sum_{\substack{y\in\D_{\ld\mu}\\y < g}}
%u^{-\ell_\fc(w_\circ^\mu)} v^{-\ell_\fa(w_\circ^\mu)} c_{y,g}^{(\ld,\mu)} C^y_{\ld\mu},
%    \label{eq:eA}
%\\
%\={e_A}(C'_{w_\circ^\mu})
%&\ds =
%u^{\ell_\fc(w_\circ^\mu)-\ell_\fc(g^+_{\ld\mu})}
%v^{\ell_\fa(w_\circ^\mu)-\ell_\fa(g^+_{\ld\mu})}
%C^g_{\ld\mu}
% + \sum_{\substack{y\in\D_{\ld\mu}\\y < g}}
%u^{\ell_\fc(w_\circ^\mu)} v^{\ell_\fa(w_\circ^\mu)} c_{y,g}^{(\ld,\mu)} C^y_{\ld\mu}.
%    \label{eq:eAbar}
%\endeqnarray
%-----------------------------------------------------------------------------------------------------------
%-----------------------------------------------------------------------------------------------------------
\subsection{A standard basis in $\SSj$}
We define, for $ A \in \Xi_{n,d}$,
the (truncated) generalized length functions of $A$ by
\bA{\label{eq:l^(A)}
&\^{\ell}(A) = \dfrac{1}{2}
\bigg(
\sum\limits_{(i,j) \in I_\fc }
\Bp{
\sum\limits_{\substack{x \le i\\ y > j}}
+
\sum\limits_{\substack{x \ge i\\ y < j}}
}
a^\natural_{ij} a_{xy}
\bigg),
\quad
\^\ell_\fc(A) =
\dfrac{1}{2}
\Bp{
\sum\limits_{\substack{0 \le x\\ 0 > y}}
+
\sum\limits_{\substack{0 \ge x\\ 0 < y}}
}
a_{xy},
\\
&
\^\ell_\fa(A) =
\^\ell(A) - \^\ell_\fc(A)
=
\dfrac{1}{2}
\bigg(
\sum\limits_{(i,j) \in I_\fc }
\Bp{
\sum\limits_{\substack{x \le i\\ y > j}}
+
\sum\limits_{\substack{x \ge i\\ y < j}}
}
a^{\natural\natural}_{ij} a_{xy}
\bigg),
}
where $a_{00}^{\natural\natural} = \frac{1}{2}(a_{00}-3)$ and $a_{ij}^{\natural\natural} =a_{ij}$ if $(i,j) \in I_\fa$.
We shall see in Proposition~\ref{prop:Abar}  that $\^\ell_\fa(A), \^\ell_\fc(A) \in\NN$.
\rmk
The function $\^\ell$ counts the dimension of the generalized Schubert variety associated to the matrix $A$ (cf. \cite[Appendix A]{FL3Wb}), and is equal to the length of $A$ when $A$ is a permutation matrix (that is when the associated variety is a genuine Schubert variety).
\endrmk
Set
\eq \label{def:[A]}
[A] = u^{-\^\ell_\fc(A)} v^{-\^\ell_\fa(A)} e_A.
\endeq
The set
$\{[A] ~|~ A \in \Xi_{n,d}\}$ forms an $\bbA$-basis of $\SSj$, which we call the \textit{standard basis}.
For $A\in \Xi_n$, we let
\eq\label{def:sig}
\sigma_{ij}(A) =
\ds\sum_{x\le i, y\ge j} a_{xy}.
\endeq
Now we define a partial order $\le_\alg$ on $\Xi_n$ by letting, for $A, B \in \Xi_n$,
\eq \label{eq:order}
A \le_\alg B \Leftrightarrow
\ro(A) = \ro(B),\; \co(A)=\co(B), \; \text{and } \sigma_{ij}(A) \le \sigma_{ij}(B), \forall i < j.
\endeq
We denote $A <_\alg B$ if $A \le_\alg B$ and $A\neq B$.

%In the following the expression ``lower terms'' represents a linear combination of smaller elements with respect to $\le_\alg$.
%-----------------------------------------------------------------------------------------------------------
\prop\label{prop:Abar}
Let $A =\kappa(\ld,g,\mu)\in \Xi_{n,d}$. Then we have
$\={[A]}  \in [A] +\sum_{B <_\alg A} \bbA [B].$
\endprop
%-----------------------------------------------------------------------------------------------------------
\proof
By the finite type analogue of \cite[Proposition~5.3]{FL3Wb}, we have
\eq
\^\ell_\fc(A) =\ell_\fc(g_{\ld\mu}^+) - \ell_\fc(w_\circ^\mu),
\quad
\^\ell_\fa(A) =\ell_\fa(g_{\ld\mu}^+) - \ell_\fa(w_\circ^\mu).
\endeq
Hence,
\eq
[A](u^{-\ell_\fc(w_\circ^\mu)}v^{-\ell_\fa(w_\circ^\mu)} x_\mu) = u^{-\ell_\fc(g^+_{\ld\mu})}v^{-\ell_\fa(g^+_{\ld\mu})} T^g_{\ld\mu}.
\endeq
Thus, by Lemma~\ref{lem:bar}, the map $\={[A]}$ is determined by
\eq
\={[A]}(u^{-\ell_\fc(w_\circ^\mu)}v^{-\ell_\fa(w_\circ^\mu)} x_\mu)
= u^{\ell_\fc(g^+_{\ld\mu})}v^{\ell_\fa(g^+_{\ld\mu})} \={T^g_{\ld\mu}}
\in u^{-\ell_\fc(g^+_{\ld\mu})}v^{-\ell_\fa(g^+_{\ld\mu})} T^g_{\ld\mu} + \sum_{y<g}  \bbA  T^y_{\ld\mu}.
\endeq
We note that $[\kappa(\ld,y,\mu)](x_\mu) \in \bbA T^y_{\ld\mu}$. An induction on $\ell(g)$ shows that
\eq\label{eq:Abar}
\={[A]}  \in [A] + \sum_{y\in \D_{\ld\mu}, y< g} \bbA\, [\kappa(\ld,y,\mu)].
\endeq
A finite type analogue of \cite[Corollary~5.5]{FL3Wb} shows that $\kappa(\ld,y,\mu) <_\alg A$ if $y < g$.  We conclude the statement.
\endproof

Let us reformulate the multiplication formula for $\SSj$ (Proposition~\ref{prop:multformula1}) in
terms of the standard basis. % (see \cite[Lemma~3.4(a2)]{BLM90}). %, \cite[Proposition~3.7]{BKLW}, \cite{DF14}).
%-----------------------------------------------------------------------------------------------------------
\begin{thm}\label{thm:multformula2}
Suppose that $A, B, C\in\Xi_{n,d}$ and $h\in[1,n]$.

\enu
\item[(1)] If $B-bE_{h,h-1}^\theta$ is diagonal, $\co(B)=\ro(A)$, then
\begin{equation}\label{mult1}
[B] [A]=\sum_{t}u^{-\delta_{h,1}\sum_{l>0}{t_l}}v^{\beta(t)}\prod_{l=-n}^{n}\overline{\left[\begin{array}{cc}a_{h,l}+t_l\\t_l\end{array}\right]} [\widecheck{A}_{t,h}],
\end{equation} where
$t$ is summed over as in Propsition \ref{prop:multformula1}~(1), and
\begin{equation}\label{beta1}
\beta(t)=\sum_{k\leq l}t_la_{h,k}-\sum_{k<l}t_l(a_{h-1,k}-t_k)+\delta_{h,1}(\sum_{-l<k<l}t_lt_k+\sum_{l>0}\frac{t_l(t_l+3)}{2}).
\end{equation}

\item[(2)] Suppose $C-cE_{h-1,h}^\theta$ is diagonal and $\co(C)=\ro(A)$.
If $h\neq1$ then
\begin{equation}\label{mult2}
[C] [A]=\sum_{t}v^{\beta'(t)}\prod_{l=-n}^{n}\overline{\left[\begin{array}{cc}a_{h-1,l}+t_l\\t_l\end{array}\right]} [\widehat{A}_{t,h}],
\end{equation} where $t$ is summed over as in Propsition \ref{prop:multformula1}~(2), and
\begin{equation}\label{beta2}
\beta'(t)=\sum_{k\geq l}t_la_{h-1,k}-\sum_{k>l}t_l(a_{h,k}-t_k).
\end{equation}
\\
If $h=1$ then
\begin{equation}\label{mult3}
[C][A]=\sum_{t}u^{\sum_{l\leq0}t_l}v^{\beta''(t)}
\overline{\left(\frac{[a_{0,0}^\natural+t_0]_\fc^!}{[a_{0,0}^\natural]_\fc^![t_0]!}\prod_{l=1}^n\frac{[a_{0,l}+t_l+t_{-l}]!}{[a_{0,l}]![t_l]![t_{-l}]!}\right)}[\widehat{A}_{t,1}],
\end{equation}
where
\begin{equation}\label{beta3}
\beta''(t)=\sum_{k\geq l}t_la_{0,k}-\sum_{k>l}t_l(a_{1,k}-t_k)+\sum_{l<k\leq -l}t_lt_k+\sum_{l\leq 0}\frac{t_l(t_l-3)}{2}.
\end{equation}
\endenu
\end{thm}
%-----------------------------------------------------------------------------------------------------------
\proof
For Part (1), by Proposition \ref{prop:multformula1}, we have
\[
[B] [A]=\sum_{t}u^{\^\ell_\fc(\widecheck{A}_{t,h})-\^\ell_\fc(A)-\^\ell_\fc(B)}v^{\^\ell_\fa(\widecheck{A}_{t,h})-\^\ell_\fa(A)-\^\ell_\fa(B)+2\sum_{k<l}t_la_{h,k}+2\sum_{l}t_la_{h,l}}
\prod_{l=-n}^{n}\overline{\left[\begin{array}{cc}a_{h,l}+t_l\\t_l\end{array}\right]} [\widecheck{A}_{t,h}].
\]
Part (1) concludes by combining the following identities via direct computation:
\begin{align*}
&\^\ell_\fc(B)=0,
\quad
\^\ell_\fa(B)=b b_{h,h}=\sum_{l,k}t_la_{h,k},
\quad
\^\ell_\fc(\widecheck{A}_{t,h})-\^\ell_\fc(A)=-\delta_{h,1}\sum_{l>0}{t_l},
\\
&\^\ell_\fa(\widecheck{A}_{t,h})-\^\ell_\fa(A)=\sum_{k>l}t_la_{h,k}-\sum_{k<l}t_l(a_{h-1,k}-t_k)+\delta_{h,1}(\sum_{-l<k<l}t_lt_k+\sum_{l>0}\frac{t_l(t_l+3)}{2}).
\end{align*}
%Therefore,
%\[
%[B] [A]=\sum_{t}
%u^{-\delta_{h,1}\sum\limits_{l>0}{t_l}}
%v^{\sum\limits_{k\leq l}t_la_{h,k}-\sum\limits_{k<l}t_l(a_{h-1,k}-t_k)+\delta_{h,1}(\sum\limits_{-l<k<l}t_lt_k+\sum\limits_{l>0}\frac{t_l(t_l+3)}{2})}
%\prod_{l=-n}^{n}\overline{\left[\begin{array}{cc}a_{h,l}+t_l\\t_l\end{array}\right]} [\widecheck{A}_{t,h}].
%\]
For Part (2), we only present the most complicated case that $h=1$.
A direct computation shows that
\eq\label{eq:mult3bar}
\frac{[a_{0,0}^\natural+t_0]_\fc^!}{[a_{0,0}^\natural]_\fc^![t_0]!}\prod_{l=1}^n\frac{[a_{0,l}+t_l+t_{-l}]!}{[a_{0,l}]![t_l]![t_{-l}]!}\\
=u^{2t_0}v^{\sum_{l}(2a_{0,l}t_l+t_lt_{-l})-3t_0}\overline{\left(\frac{[a_{0,0}^\natural+t_0]_\fc^!}{[a_{0,0}^\natural]_\fc^![t_0]!}\prod_{l=1}^n\frac{[a_{0,l}+t_l+t_{-l}]!}{[a_{0,l}]![t_l]![t_{-l}]!}\right)}.
\endeq
%\begin{align*}
%&\frac{[a_{0,0}^\natural+t_0]_\fc^!}{[a_{0,0}^\natural]_\fc^![t_0]!}\prod_{l=1}^n\frac{[a_{0,l}+t_l+t_{-l}]!}{[a_{0,l}]![t_l]![t_{-l}]!}\\
%&=u^{2t_0}v^{2a_{00}t_0+t_0(t_0-3)+2\sum_{l>0}(a_{0,l}t_l+a_{0,l}t_{-l}+t_lt_{-l})}\overline{\left(\frac{[a_{0,0}^\natural+t_0]_\fc^!}{[a_{0,0}^\natural]_\fc^![t_0]!}\prod_{l=1}^n\frac{[a_{0,l}+t_l+t_{-l}]!}{[a_{0,l}]![t_l]![t_{-l}]!}\right)}\\
%&=u^{2t_0}v^{\sum_{l}(2a_{0,l}t_l+t_lt_{-l})-3t_0}\overline{\left(\frac{[a_{0,0}^\natural+t_0]_\fc^!}{[a_{0,0}^\natural]_\fc^![t_0]!}\prod_{l=1}^n\frac{[a_{0,l}+t_l+t_{-l}]!}{[a_{0,l}]![t_l]![t_{-l}]!}\right)}.
%\end{align*}
Part (2) follows from combining \eqref{eq:mult3bar} and the calculation below:
\begin{align*}
&\^\ell_\fc(C)=c=\sum_{l}t_l,
\quad
\^\ell_\fa(C)=\sum_{l,k}t_la_{0,k}+\frac{c(c-3)}{2},
\quad
\^\ell_\fc(\widehat{A}_{t,1})-\^\ell_\fc(A)
=\sum_{l>0}{t_l},
\\
&\^\ell_\fa(\widehat{A}_{t,1})-\^\ell_\fa(A)
=\sum_{k<l}t_la_{0,k}-\sum_{l<k}t_l(a_{1,k}-t_k)+(\sum_{-l<k<l}t_lt_k+\sum_{l>0}\frac{t_l(t_l-3)}{2}).
\end{align*}
\endproof

%-----------------------------------------------------------------------------------------------------------
\subsection{A monomial basis in $\SSj$}
%-----------------------------------------------------------------------------------------------------------
Thanks to Remark \ref{specialize}, we can use results in \cite{BKLW18} freely when we specialize $u=v$.
For $A \in \Xi_{n,d}$, we can use the algorithm
in \cite[Theorem~3.10]{BKLW18} with the fixed order therein
to produce a unique family of Chevalley matrices $\{A^{(1)}, \ldots, A^{(x)}\}$ in $\Xi_{n,d}$ for some $x = x(A) \in \NN$.
At the specialization $u=v$, a unitriangular relation is satisfied:
\eq
 \left. [A^{(1)}] \cdots [A^{(x)}]\right|_{u=v} = \textstyle[A] + \sum_{B<_\alg A}\left. \bbA [B]\right|_{u=v}.
\endeq
Denote the product of the corresponding elements in $\SSj$ by
\eq\label{def:mA}
m_A =[A^{(1)}] \cdots [A^{(x)}] \in \SSj.
\endeq
Let $I$ be the identity matrix. Since the algorithm in \cite[Theorem~3.10]{BKLW18} produces matrices $A^{(1)}, \ldots, A^{(x)}$ according to mainly the off-diagonal matrices of $A$ and then determine the diagonal entries of these $A^{(i)}$ by the row and column sums, we have that $x(A) = x(A+pI)$ and $(A+pI)^{(i)} = A^{(i)}+ pI$ for all $p\in2\NN$, i.e.,
\eq\label{eq:monop}
m_{A+pI} = [A^{(1)}+pI] \cdots [A^{(x)}+pI].
\endeq
%-----------------------------------------------------------------------------------------------------------
\prop\label{prop:mono}
For $A \in \Xi_{n,d}$ the element $m_A \in \SSj$ has the following property:
\eq\label{eq:mA}
m_A =[A] + \sum_{B<_\alg A} \bbA [B].
\endeq
Moreover,  $\{m_A\}_{A\in \Xi_{n,d}}$ form a basis of $\SSj$, which we call the monomial basis.
\endprop
%-----------------------------------------------------------------------------------------------------------
\proof
A direct proof can be pursued using the multiplication formulas (Proposition~\ref{thm:multformula2}), similar to the proofs of \cite[Theorem 3.10]{BKLW18} and \cite[Theorem 4.6.3]{FL15}.
Here we offer a simpler proof by combining \cite[Theorem 3.10]{BKLW18} and \cite[Theorem 4.6.3]{FL15} as below: now
\[
m_A =u^{\alpha(A)}v^{\beta(A)}[A] + \sum_{B<_\alg A} \bbA [B],
\quad
\textup{for some}
\quad \alpha(A),\beta(A)\in \NN
\]
It follows from \cite[Theorem 3.10]{BKLW18} (resp. \cite[Theorem 4.6.3]{FL15}) that $v^{\alpha(A)}v^{\beta(A)}=1$ (resp. $1^{\alpha(A)}v^{\beta(A)}=1$), which forces that $u^{\alpha(A)}v^{\beta(A)}=1$ and hence \eqref{eq:mA} holds.
Hence the transition matrix from $\{m_A~|~A\in \Xi_{n,d}\}$ to the standard basis $\{[A] ~|~ A \in \Xi_{n,d}\}$ is unital triangular. Therefore $\{m_A~|~A\in \Xi_{n,d}\}$ form a basis of $\SSj$.
\endproof
%-----------------------------------------------------------------------------------------------------------
\begin{rmk}
The monomial basis acts as an intermediate step toward constructing canonical basis in the one-parameter case.
Moreover, the two-parameter stabilization procedure is made possible thanks to the property \eqref{eq:monop} of monomial basis.
\end{rmk}

%-----------------------------------------------------------------------------------------------------------
\subsection{The canonical basis at the specialization}\label{sec:SCB}
%-----------------------------------------------------------------------------------------------------------

For any weight function $\bL$, let $\mbf{c} = \gcd(\bL(s_0), \bL(s_1))$. We show that the specialization of $\SSj$ at $u=\bv^{\bL(s_0)}, v=\bv^{\bL(s_1)}$ admits canonical basis with respect to $\bv^{\mbf{c}}$.
For $A \in \Xi_{n,d}$, let $[A]^\bL$ (and $m_A^\bL$, resp.) be the standard basis (and monomial basis, resp.) of the specialization of $\SSj$ at $u=\bv^{\bL(s_0)}, v=\bv^{\bL(s_1)}$.
It follows from \eqref{eq:Abar} and \eqref{eq:mA} that the following unitriangular relations hold:
\bA{
\={[A]^\bL} &{}\in [A]^\bL + \sum_{B <_\alg A} \ZZ[\bv^{\mbf{c}}, \bv^{-\mbf{c}}] [B]^\bL,
\\
\={m_A^\bL} = m_A^\bL &\in [A]^\bL + \sum_{B <_\alg A} \ZZ[\bv^{\mbf{c}}, \bv^{-\mbf{c}}] [B]^\bL.
}
If $A$ is diagonal, set $\{A\}^\bL =[A]^\bL$.
Arguing inductively on the partial order $\le_\alg$ and using a standard argument (cf. \cite[24.2.1]{Lu93})
there exists a unique element $\{A\}^\bL \in \SSj$ such that
\eq\label{eq:SCB}
\overline{ \{A\}^\bL } =\{A\}^\bL \in [A]^\bL + \sum_{B <_\alg A} \bv^{-\mbf{c}}\ZZ[\bv^{-\mbf{c}}] [B]^\bL.
\endeq

Let $\mathbb{S}^{\jmath, \bL}_{n,d}$ be the specialization of $\SSj$ at $u=\bv^{\bL(s_0)}, v=\bv^{\bL(s_1)}$.
\begin{thm} \label{thm:SCB}
There exists a canonical basis $\{\{A\}^\bL\ |\ A \in \Xi_{n,d}\}$ for $\mathbb{S}^{\jmath, \bL}_{n,d}$,
which is characterized by the property \eqref{eq:SCB}.
\end{thm}
%=========================================================
\section{Stabilization algebra $\dKKj$}  \label{sec:stabj}
In this section,
we shall establish a stabilization property for the family of Schur algebras $\SSj$ as $d$ varies,
which leads to a quantum algebra $\dKKj$.
%=========================================================
\subsection{A BLM-type stabilization}
   \label{sec:stab}
%=========================================================
Let
\eq\label{eq:Xitn}
\Xit_n = \left\{
(a_{ij})_{-n\leq i,j\leq n} \in \text{Mat}_{N\times N}(\ZZ)
\middle|
\substack{\ds a_{-i,-j} =a_{ij} (\forall i, j),
\\
\ds a_{xy}\in\NN (\forall x\neq y),
a_{00}\in2\ZZ+1}
\right\}.
\endeq
Extending the partial ordering $\le_\alg$  for $\Xi_{n}$,
we define a partial ordering $\le_\alg$ on $\Xit_n$ using the same recipe
\eqref{eq:order}.
For each $A \in \Xit_n$ and $p \in2\NN$, we write
\eq
\p{A} = A+pI \in \Xit_n.
\endeq
Then $\p{A} \in \Xi_n$ for even $p \gg 0$.
Let $\pi$ be an indeterminate (independent of $u,v$),
and $\mathcal{R}_1$ be the subring of $\QQ(u,v)[\pi,\pi\inv]$ generated by, for $a\in\ZZ, k\in \ZZ_{>0}$,
\eq\label{def:R1}
r^{(1)}_{a,k},
\quad
r^{(2)}_{a,k},
\quad
v^a,
\quad
\textup{and}
\quad
u^a,
\endeq
where
\bA{\label{def:r_i}
r^{(1)}_{a,k}(u,v,\pi) &{}= \prod_{i=1}^k \frac{v^{-2(a-i)}\pi^2-1}{v^{-2i}-1},\\
r^{(2)}_{a,k}(u,v,\pi) &{}= \prod_{i=1}^k \frac{(u^{-2}v^{-2(a-1-i)}\pi+1)(v^{-2(a-i)}\pi-1)}{v^{-2i}-1}.
}
Let $\mathcal{R}_2$ be the subring of $\QQ(u,v)[\pi, \pi\inv]$ generated by, for $a\in\ZZ, k\in \ZZ_{>0}$,
\eq\label{def:R2}
r^{(1)}_{a,k},
\quad
\={r}^{(1)}_{a,k},
\quad
r^{(2)}_{a,k},
\quad
\={r}^{(2)}_{a,k},
\quad
v^a,
\quad
\textup{and}
\quad
u^a.
\endeq
We extend the bar-involution to $\mathcal{R}_2$ by requiring $\overline{\pi}=\pi^{-1}$.
%-----------------------------------------------------------------------------------------------------------
\prop\label{prop:stab1}
Let $A_1, \ldots, A_f \in \Xit_n$ be such that $\co(A_i) = \ro(A_{i+1})$ for all $i$.
Then there exists matrices $Z_1, \ldots, Z_m \in \Xit_n$ and $\zeta_i(u,v,\pi) \in \mathcal{R}_1$
such that for  even integer $p\gg 0 ,$
\eq\label{eq:stab1}
[\p{A_1}] [\p{A_2}]  \cdots  [\p{A_f}] = \sum_{i=1}^m \zeta_i(u,v, v^{-p}) [\p{Z_i}].
\endeq
\endprop
%-----------------------------------------------------------------------------------------------------------
\proof
We assume first that $f=2$ and $A_1$ is such that $A_1-bE_{h,h-1}^\theta$ is diagonal for some $h\in[1,n]$ and some $b\geq0$. Let $A_2=A=(a_{ij})$. For each $t=(t_i)_{-n\leq i\leq n}\in\NN^N$, we define
$$\zeta_t(u,v,\pi)=
u^{-\delta_{h,1}\sum_{l>0}{t_l}}v^{\beta(t)}
\prod_{h\neq l\in[-n,n]}\overline{\left[\begin{array}{cc}a_{h,l}+t_l\\t_l\end{array}\right]}
\prod_{i=1}^{t_h}\frac{v^{-2(a_{h,h}+t_h-i+1)}\pi^2-1}{v^{-2i}-1}\in\mathcal{R}_1$$
where $\beta(t)$ is defined in \eqref{beta1}. Though $\beta(t)$ depends on $A$, it is invariant if $A$ is replaced by $\p{A}$. Therefore we have the following formula for large enough even $p$ by \eqref{mult1}:
$$[\p{A}_1][\p{A}]=\sum_{t}\zeta_t(u,v,v^{-p})[\p{\widecheck{A}}_{t,h}].$$
The statement holds in this case.

We next assume that $f=2$ and $A_1$ is such that $A_1-cE_{h-1,h}^\theta$ is diagonal for some $h\in[1,n]$ and some $c\geq0$. Let $A_2=A=(a_{ij})$. Recall $\beta'(t)$ and $\beta''(t)$ in \eqref{beta2} and \eqref{beta3}, respectively. If $h\neq1$, for each $t=(t_i)_{-n\leq i\leq n}\in\NN^N$, we define
$$\zeta_t(u,v,\pi)=
v^{\beta'(t)}\prod_{h-1\neq l\in[-n,n]}\overline{\left[\begin{array}{cc}a_{h-1,l}+t_l\\t_l\end{array}\right]}\prod_{i=1}^{t_h}\frac{v^{-2(a_{h-1,h-1}+t_{h-1}-i+1)}\pi^2-1}{v^{-2i}-1}\in\mathcal{R}_1;$$
If $h=1$, we define
\begin{align*}
\zeta_t(u,v,\pi)=&
u^{\sum_{l\leq0}t_l}v^{\beta''(t)}
\overline{\left(\prod_{l=1}^n\frac{[a_{0,l}+t_l+t_{-l}]!}{[a_{0,l}]![t_l]![t_{-l}]!}\right)}\\
&\cdot\prod_{i=1}^{t_0}\frac{(u^{-2}v^{-2(a_{00}^\natural+t_0-1-i)}\pi+1)(v^{-2(a_{00}^\natural+t_0-i)}\pi-1)}{v^{-2i}-1}\in\mathcal{R}_1.
\end{align*}
It is clear that both $\beta'(t)$ and $\beta''(t)$ are invariant if $A$ is replaced by $\p{A}$. Therefore the following formula holds for large enough even $p$ by \eqref{mult2}:
$$[\p{A}_1][\p{A}]=\sum_{t}\zeta_t(u,v,v^{-p})[\p{\widehat{A}}_{t,h}].$$
Hence the proposition is verified in the present case.

Using induction on $f$, we know that the proposition holds for general $f$ in the case where $A_1,\ldots,A_f$ are Chevalley matrices (i.e. of one of the two types considered above).
It follows from \eqref{def:mA} and \eqref{eq:mA} that
for any $A\in\Xi_{n,d}$, there exists Chevalley matrices $B_1, B_2,\ldots,B_M$ such that
$$[B_1][B_2]\cdots[B_M]=[A]+\mbox{lower terms}.$$
Then we can prove the proposition by using induction on $\Psi(A)=\sum_{i<j}\sigma_{ij}(A)$. We omit the subsequent argument here since it is totally as the same as those for \cite[Proposition 4.2]{BLM90}.
\endproof

%\rmk
%The finite type A analog of $\beta(A,S,T)$ is referred as $\beta(\mathbf{t}) = \beta(\mathbf{t})(A)$ in \cite[3.4]{BLM90},
%which satisfies that $\beta(\mathbf{t})(A) = \beta(\mathbf{t})(\p{A})$ for all $p$.
%This leads to a stronger statement \cite[Proposition 4.2]{BLM90} that the structure constants lie in a smaller ring
%(referred as $\mathcal{R}_2$), which is not essential for the stabilization procedure.
%The analogs of $\beta$ are referred as $\beta(t)$ for finite type B/C in \cite{BKLW14}, and as $P_{T,A}$ for affine type A
%in \cite{DF14}. These functions are also unchanged when $A$ is replaced by $\p{A}$.
%\endrmk

%--------------------------------------
%\subsection{Stabilization of bar involutions}
By an argument identical with \cite[Proposition 4.3]{BLM90}, we obtain below the stabilization of bar involution by allowing extra coefficients as seen in \eqref{def:R2}.
\begin{prop}\label{prop:stab2}
For any $A \in \Xit_n$, there exist matrices $T_1, \ldots, T_s\in \Xit_n$ and $\tau_i(u,v,\pi) \in \mathcal{R}_2$  such that,
for  even integer $p\gg 0,$
\eq\label{eq:stab2}
\={[\p{A}]} = \sum_{i=1}^s \tau_i(u,v, v^{-p}) [\p{T_i}].
\endeq
\end{prop}
%-----------------------------------------------------------------------------------------------------------
%\proof
%The proof is almost identical with \cite[\S~4.3]{BLM90}. It suffices to verify the identity \eqref{prop:stab2} for the case $A \in \Xi_n$. We again prove by an induction on $\Psi(A)$. Note that for the initial case $\Psi(A) = 0$, $A$ is diagonal and hence bar-invariant.
%For the inductive case $\Psi(A) > 0$, we recall matrices $A^{(i)}$ thanks to the existence of monomial basis. By applying Proposition~\ref{prop:stab1} on $[\p{A}^{(1)}] \cdots [A^{(x)}]$ (see \eqref{eq:pAA2}), we obtain
%\eq
%[\p{A}] = [\p{A}^{(1)}] \cdots [\p{A}^{(x)}] - \sum_{i=2}^m \zeta_i [\p{Z}_i],
%\endeq
%where $\zeta_i \in \mathcal{R}_1$ and each $Z_i \in \Xit_n$ satisfies that $\Psi(Z_i) < \Psi(A)$.
%By the inductive hypothesis we have
%\eq
%\={[\p{A}]} = [\p{A}^{(1)}] \cdots [\p{A}^{(x)}] - \sum_{i=2}^m \={\zeta}_i \={[\p{Z}_i]},
%\endeq
%\endproof

Let $\dKKj$ be the free $\bbA$-module with an $\bbA$-basis given by the symbols $[A]$ for $A\in \Xit_n$ (which will be
called a standard basis of $\dKKj$).
By Propositions~ \ref{prop:stab1}--\ref{prop:stab2} and applying a specialization at $\pi=1$ (note that $\zeta_i(u,v, 1)\in\bbA$), we have the following corollary.
\begin{cor}\label{cor:stab}
There is a unique associative $\bbA$-algebra structure on $\dKKj$ with multiplication given by
\[
[A_1]  [A_2]  \cdots  [A_f]
= \bc{
\sum_{i=1}^m \zeta_i(u,v, 1) [Z_i]
&\tif \co(A_i) = \ro(A_{i+1})\textup{ for all }i,
\\
0
&\textup{otherwise}.
}
\]
Moreover, the map $\bar{\empty}: \dKKj \rw \dKKj$ given by
$\={[A]} = \sum_{i=1}^s \tau_i(u,v, 1) [T_i]$
is an $\bbA$-linear involution.
\end{cor}
The following multiplication formula in $\dKKj$ follows directly from
Theorem \ref{thm:multformula2} by the stabilization construction.
\begin{prop}
 \label{prop:multK}
Let $A, B, C \in \Xit_{n}$ and $h\in[1,n]$.
\enu
\item[(1)] If $B-bE_{h,h-1}^\theta$ is diagonal and $\co(B)=\ro(A)$, then
\begin{equation}
[B] [A]=\sum_{t}u^{-\delta_{h,1}\sum_{l>0}{t_l}}v^{\beta(t)}\prod_{l=-n}^{n}\overline{\left[\begin{array}{cc}a_{h,l}+t_l\\t_l\end{array}\right]} [\widecheck{A}_{t,h}],
\end{equation} where
$t=(t_i)_{-n\leq i\leq n}\in\NN^N$
with $\sum_{i=-n}^n t_i=b$ such that
$$\left\{\begin{array}{ll}
t_i\leq a_{h-1,i} & \mbox{if~} i+1\neq h>1;\\
t_i+t_{-i}\leq a_{0,i} & \mbox{if~} h=1, i\neq0,
\end{array}
\right.$$
\item[(2)] Suppose $C-cE_{h-1,h}^\theta$ is diagonal and $\co(C)=\ro(A)$.
\\
If $h\neq1$ then
\begin{equation}
[C] [A]=\sum_{t}v^{\beta'(t)}\prod_{l=-n}^{n}\overline{\left[\begin{array}{cc}a_{h-1,l}+t_l\\t_l\end{array}\right]} [\widehat{A}_{t,h}],
\end{equation} where $t=(t_i)_{-n\leq i\leq n}\in\NN^N$ with $\sum_{i=-n}^n t_i=c$ such that $t_i\leq a_{h,i}$ if $i\neq h$.
\\
If $h=1$ then
\begin{equation}\label{mfinK3}
[C][A]=\sum_{t}u^{\sum_{l\leq0}t_l}v^{\beta''(t)}
\overline{\left(\frac{\prod_{k=a^\natural_{00}+1}^{a^\natural_{00}+t_{0}}[k](u^2v^{2(k-1)}+1)}
{\prod_{k=1}^{t_{0}}[k]}\prod_{l=1}^n\frac{[a_{0,l}+t_l+t_{-l}]!}{[a_{0,l}]![t_l]![t_{-l}]!}\right)}[\widehat{A}_{t,1}],
\end{equation}
where $t=(t_i)_{-n\leq i\leq n}\in\NN^N$ with $\sum_{i=-n}^n t_i=c$ such that $t_i\leq a_{1,i}$ if $i\neq 1$.
\endenu
\end{prop}

%=====================
\subsection{Monomial and canonical bases for $\dKKj$}
The proposition below follows from Proposition \ref{prop:mono} by the stabilization construction.
\begin{prop}  \label{prop:mAK'}
For any $A\in\Xit_{n}$, there exist Chevalley matrices $A^{(1)}, \ldots, A^{(x)}$ in $\Xit_{n}$
satisfying $\ro(A^{(1)})=\ro(A)$, $\co(A^{(x)})=\co(A)$, $\co(A^{(i)})=\ro(A^{(i+1)})$ for $1\le i\le x-1$
\eq \label{eq:mAK}
[A^{(1)}] [A^{(2)}] \cdots  [A^{(x)}]\in [A]+ \sum_{B <_\alg A} \bbA [B] \; \in \dKKj.
\endeq
\end{prop}
By abuse of notation, we denote the product in $\dKKj$ by
\eq
m_A = [A^{(1)}] [A^{(2)}] \cdots  [A^{(x)}] \in \dKKj.
\endeq
Hence $\{m_A~|~A\in\Xit_{n}\}$ forms a basis for $\dKKj$ (called a \emph{monomial basis}).
Similar to Section~\ref{sec:SCB}, we define, by abuse of notation, elements $[A]^\bL, m_A^\bL, \{A\}^\bL$ to be the according basis elements of $\dKKj$ at the specialization $u = \bv^{\bL(s_0)}, v= \bv^{\bL(s_1)}$.
\begin{thm}\label{thm:KCB}
There exists a canonical basis $\dBj = \{\{A\}^\bL\ |\ A \in \Xi_{n,d}\}$ for $\dKKj$ at the specialization  $u=\bv^{\bL(s_0)}, v=\bv^{\bL(s_1)}$,
which is characterized by the property \eqref{eq:SCB}.
\end{thm}

%=========================================================
\section{A different stabilization algebra $\dKKi$}\label{sec:i}

In this section we formulate a variant of Schur algebras and their corresponding stabilization algebras.
We construct the distinguished bases of these algebras.
Recall $N=2n+1$. %and we now set
%\eq
%\nn =n-1=2r \quad (r\ge 1).
%\endeq
%=========================================================
\subsection{$\imath$-Schur algebras }  \label{sec:Si}

Recall $\Xi_{n,d}$ from \eqref{def:Xi}. Let
\eq \label{eq:Xii}
\iXi = \{ A \in \Xi_{n,d}~|~\ro(A)_0 = 1 = \co(A)_0 \}.
\endeq
%A general element $A\in \iXi$ \eqref{eq:Xij} is of the form below
%(where $r\#\backslash c\#$ stands for row \# and column \# of the matrix):
%\eq\label{eq:AinXij}
%A=
%\ba{{c||cc:c:ccc:c:ccc}
%r\# \backslash c\#  &\cdots&{-1}&0&1&\cdots&{n-1}&n&{n+1}&\cdots
%\\
%\hline \hline
%\vdots&\ddots&&\vdots&&&&\vdots
%\\
%{-1}&&a_{n-1,n-1}&0&a_{-1,1}&&*&0&*
%\\
%\hdashline
%{0}&\cdots&0&1&0&\cdots&0&0&0&\cdots
%\\
%\hdashline
%{1}&&a_{1,-1}&0&a_{11}&&*&0&*
%\\
%\vdots&&&\vdots&&\ddots&&\vdots
%\\
%{n-1}&&*&0&*&&a_{n-1,n-1}&0&a_{-1,1}
%\\
%\hdashline
%{n}&\cdots&0&0&0&\cdots&0&1&0&\cdots
%\\
%\hdashline
%{n+1}&&*&0&*&&a_{1,-1}&0&a_{11}
%\\
%\vdots&&&\vdots&&&&\vdots&&\ddots
%\\
%}
%\endeq
Recall $\Ld_{n,d}$ \eqref{def:Ld}.
Let
\begin{align*}
\Ld_{n,d}^\imath =\{ \ld = (\ld_n,\ldots,\ld_1,1,\ld_1,\ldots,\ld_n)\in\Ld_{n,d} \}.
%\\
 %\D_{\nn,d}^{\imath} &= \{(\ld, g, \mu) ~|~ \ld,\mu\in\Ld^\imath, g\in\D_{\ld\mu} \}.
\end{align*}
The lemma below is the $\imath$-analog of Lemma~\ref{lem:kappa}, which follows by a similar argument.
%-----------------------------------------------------------------------------------------------------------
\begin{lem}\label{lem:kappai'}
%The restriction of $\kappa^{-1}$ on $\iXi$ is a bijection. In particular,
The map
$
\kappa^{\imath}:  \bigsqcup_{\ld,\mu\in\Ld_{n,d}^\imath} \{\ld\} \times \D_{\ld\mu} \times \{\mu\} \longrightarrow \iXi
$
sending $(\ld, g, \mu)$ to $(|R_i^\ld \cap g R_j^\mu|)$ is a bijection.
\end{lem}
%-----------------------------------------------------------------------------------------------------------
Now we define the {\em $\imath$-Schur algebra } as
\eq
\SSi = \textup{End}_{\HH}
\big(
\mathop{\oplus}_{\ld\in \Ld_{n,d}^\imath} x_\ld \HH \big).
\endeq
By definition the algebra $\SSi$ is naturally a subalgebra of $\SSj$.
Moreover, both  $\{  e_A ~|~ A \in \iXi \}$ and $\{  [A] ~|~ A \in \iXi \}$ are bases of $\SSi$ as a free $\bbA$-module.
%-----------------------------------------------------------------------------------------------------------
%================
\subsection{Monomial and canonical bases for $\SSi$}
\prop\label{thm:min gen'}
For each $A \in \iXi$, we have $m_A\in\SSi$. Hence the set $\{m_A~|~A \in \iXi \}$ forms an $\bbA$-basis of $\SSi$. Furthermore, we have
$m_A \in [A] +\sum_{B\in \iXi, B<_\alg A} \bbA [B]$.
\endprop
\proof
It follows from \cite[Proposition 5.6]{BKLW18} thanks to Remark~\ref{specialize}.
\endproof
%-----------------------------------------------------------------------------------------------------------

\begin{thm}\label{thm:SiCB}
At the specialization $u = \bv^\bL(s_0), v= \bv^\bL(s_1)$, there is a canonical basis
$\mathfrak B_{n,d}^{\imath} = \{ \{A\}^\bL~|~A \in \iXi \}$ of $\SSi$  such that $\overline{\{A\}^\bL } =\{A\}^\bL$ and
$\{A\}^\bL \in [A]^\bL + \sum_{B\in \iXi, B<_\alg A} \bv^{-\mathbf{c}} \ZZ[\bv^{-\mathbf{c}}] [B]^\bL$.
Moreover, we have
$\mathfrak B_{n,d}^{\imath} = \mathfrak B_{n,d}^{\jmath}\cap \SSi$.
\end{thm}
\proof
The first half statement on the canonical basis follows by Proposition~\ref{thm:min gen'} and a standard argument (cf. \cite[24.2.1]{Lu93}).
The second half statement follows from the uniqueness characterization of the canonical basis $\mathfrak B_{n,d}^{\imath}$.
\endproof

%=========================================================
\subsection{Stabilization algebra of type $\imath$}
 \label{sec:Ki}
We define two subsets of $\Xit_n$  \eqref{eq:Xitn} as follows:
\eq
\Xit_n^< = \{A =(a_{ij}) \in \Xit_n ~|~ a_{00} < 0\},
\quad
\Xit_n^> = \{A  =(a_{ij})  \in \Xit_n ~|~ a_{00} > 0\}.
\endeq
For any matrix $A \in \Xit_n$ and $p\in2\NN$, we define
\eq  \label{eq:ppA}
\pp{A} = A + p(I - E^{00}).
\endeq
%-----------------------------------------------------------------------------------------------------------
\lem \label{lem:stab-ij}
For $A_1, A_2,\ldots, A_f \in \Xit_n^>$, there exists $\mathcal{Z}_i \in \Xit_n^>$ and
$\zeta^\imath_i(u,v,\pi) \in \mathcal{R}_1$ such that for all even integers
$p\gg 0$, we have an identity in $\SSj$ of the form:
\[
[\pp{A_1}]  [\pp{A_2}]  \ldots   [\pp{A_f}] = \sum\limits_{i=1}^m \zeta^\imath_i(u, v, v^{-p}) [\pp{\mathcal Z}_i].
\]
\endlem
\proof
The proof is similar to the proof of Proposition~\ref{prop:stab1} where $\p{A} = A+pI$ is used instead of
$\pp{A}$.
\endproof
%-----------------------------------------------------------------------------------------------------------
Consequently, the vector space $\dKKjp$ over $\bbA$ spanned by the symbols
$[A]$, for $A \in \Xit_n^>$, is a stabilization algebra whose multiplicative structure is given by
(with $f=2$; associativity follows from $f=3$):
\eq
[A_1]  [A_2]  \cdots  [A_f]
= \bc{
\sum\limits_{i=1}^m \zeta^\imath_i(u,v, 1) [\mathcal{Z}_i]
&\tif \co(A_i) = \ro(A_{i+1}) \; \forall i,
\\
0
&\textup{otherwise}.
}
\endeq

Precisely, we have the following multiplication formulas for Chevalley generators in $\dKKjp$.
\begin{prop}\label{mfinKl}
Let $A, B, C \in \Xit_n^>$ and $h\in[1,n]$.
\enu
\item[(1)] If $B-bE_{h,h-1}^\theta$ is diagonal and $\co(B)=\ro(A)$, then
\begin{equation}
[B] [A]=\sum_{t}u^{-\delta_{h,1}\sum_{l>0}{t_l}}v^{\beta(t)}\prod_{l=-n}^{n}\overline{\left[\begin{array}{cc}a_{h,l}+t_l\\t_l\end{array}\right]} [\widecheck{A}_{t,h}],
\end{equation} where
$t=(t_i)_{-n\leq i\leq n}\in\NN^N$
with $\sum_{i=-n}^n t_i=b$ such that
$$\left\{\begin{array}{ll}
t_i\leq a_{h-1,i} & \mbox{if~} i+1\neq h>1;\\
t_i+t_{-i}\leq a_{0,i} & \mbox{if~} h=1, \forall i,
\end{array}
\right.$$
\item[(2)] Suppose $C-cE_{h-1,h}^\theta$ is diagonal and $\co(C)=\ro(A)$.
\\
If $h\neq1$ then
\begin{equation}
[C] [A]=\sum_{t}v^{\beta'(t)}\prod_{l=-n}^{n}\overline{\left[\begin{array}{cc}a_{h-1,l}+t_l\\t_l\end{array}\right]} [\widehat{A}_{t,h}],
\end{equation} where $t=(t_i)_{-n\leq i\leq n}\in\NN^N$ with $\sum_{i=-n}^n t_i=c$ such that $t_i\leq a_{h,i}$ if $i\neq h$.
\\
If $h=1$ then
\begin{equation}\label{mfinK3}
[C][A]=\sum_{t}u^{\sum_{l\leq0}t_l}v^{\beta''(t)}
\overline{\left(\frac{\prod_{k=a^\natural_{00}+1}^{a^\natural_{00}+t_{0}}[k](u^2v^{2(k-1)}+1)}
{\prod_{k=1}^{t_{0}}[k]}\prod_{l=1}^n\frac{[a_{0,l}+t_l+t_{-l}]!}{[a_{0,l}]![t_l]![t_{-l}]!}\right)}[\widehat{A}_{t,1}],
\end{equation}
where $t=(t_i)_{-n\leq i\leq n}\in\NN^N$ with $\sum_{i=-n}^n t_i=c$ such that $t_i\leq a_{1,i}$ if $i\neq 1$.
\endenu
\end{prop}
By arguments entirely analogous to those for Corollary~\ref{cor:stab} and
Theorem~\ref{thm:KCB},
$\dKKjp$ admits a (stabilizing) bar involution,
$\dKKjp$ admits  a monomial basis $\{m_A~|~ A \in \Xit_n^>\},$
and a  canonical basis $\dot{\mathfrak B}^{\jmath,>}$.
Let $\dKKi$ be the $\bbA$-submodule of $\dKKjp$ generated by $\{ [A]~|~A \in \iXit\}$, where
\eq  \label{dKij2}
\iXit = \{A \in \Xit_n^> ~|~ \co(A)_0 = \ro(A)_0 = 1\}.
\endeq
The goal of this subsection is to realize $\dKKi$ as a subquotient of $\dKKj$ with compatible bases
by following \cite[Appendix~A]{BKLW18}.
It follows from \eqref{dKij2} that $\dKKi$ is a subalgebra of $\dKKjp$.
Since  the bar-involution on $\dKKjp$ restricts to an involution on $\dKKi$, we reach the following conclusion.
\begin{lem}  \label{lem:Kp}
The set $\dKKi \cap \dot{\mathfrak B}^{\jmath,>}$ forms a canonical basis of $\dKKi$.
\end{lem}
The submodule of $\dKKj$ spanned by $[A]$ for $A \in \iXit$ is not a subalgebra.
This is why we need a somewhat different stabilization above to construct the canonical basis for $\dKKi$.
We shall see below the stabilization above is related to the stabilization used earlier.
%Now we realize $\dKKi$ as a subquotient of $\dKKj$.
Define $\JJ$ to be the $\bbA$-submodule of $\dKKj$ spanned by $[A]$ for all $A \in \Xit_n^<$.
%-----------------------------------------------------------------------------------------------------------
\lem\label{lem:2ideal}
The submodule $\JJ$ is a two-sided ideal of $\dKKj$.
\endlem

\proof
We note that $\JJ$  is clearly  invariant under the anti-involution for $\dKKj$ below:
\eq
[A] \mapsto u^{- \^\ell_\fc(A) + \^\ell_\fc({}^tA)} v^{- \^\ell_\fa(A) + \^\ell_\fa({}^tA)} [{}^t A].
\endeq
Hence the claim that $\JJ$ is a left ideal of $\dKKj$ is equivalent to that $\JJ$ is a right ideal of $\dKKj$.
We shall show that $\JJ$ is a left ideal of $\dKKj$. To that end,
it suffices to show that $[B]  [A] \in \JJ$ for arbitrary $A \in \Xit_n^<$ and $B\in \Xit_n$ such that $B-bE_{h,h-1}$ or $B-bE_{h-1,h}$ is diagonal for some $h\in[1,n]$ and $b\geq 0$.
Thanks to the multiplication formulas in Proposition \ref{prop:multK}, unless the case of $B-bE_{0,1}^\theta$ being diagonal, the $(0,0)$-entry of the terms arising in $[B][A]$ never exceeds $a_{0,0}$. Thus $[B][A]\in\JJ$ in these cases.

Consider the case that $B-bE_{0,1}^\theta$ is diagonal. Recall the formula \eqref{mfinK3}. If the $(0,0)$-entry $a_{0,0}+2t_0$ of the term $[\widehat{A}_{t,1}]$ is positive, then the coefficient of this term must be zero since
$$\frac{\prod_{k=a^\natural_{00}+1}^{a^\natural_{00}+t_{0}}[k](u^2v^{2(k-1)}+1)}
{\prod_{k=1}^{t_{0}}[k]}=0,$$
because of $a^\natural_{00}+1\leq 0< a^\natural_{00}+t_{0}$.
Therefore, we always have $[B][A] \in \JJ$.
\endproof

\lem
If $A \in \Xit_n^<$ then $m_A \in \JJ$.
%In other words, $T_{A,B} = 0$ unless $B \in \Xit_n^<$.
\endlem
%-----------------------------------------------------------------------------------------------------------
\proof
The proof is as the same as the one of \cite[Lemma A.6 (1)]{BKLW18}.
\endproof

Recall $\dKKj$ admits a canonical basis of $\dBj$ at the specialization  $u=\bv^{\bL(s_0)}, v=\bv^{\bL(s_1)}$ from Theorem~\ref{thm:KCB}.
\begin{thm}\label{thm:monoJ}
The ideal $\JJ$ admits a monomial basis $\{m_A ~|~ A \in \Xit_n^<\}$. Moreover, its specialization at $u=\bv^{\bL(s_0)}, v=\bv^{\bL(s_1)}$ (denoted by $\JJ^\bL$) has
 a canonical basis $\dBj \cap \JJ^\bL = \{\{A\}^\bL ~|~ A \in \Xit_n^<\}$.
\end{thm}

\proof
The first statement follows from the above lemma directly.
Since $m_A=[A]+\mbox{~lower terms}$, we know that $\JJ^\bL$ is bar invariant. Thus $\JJ^\bL$ does admit a canonical bases parameterized by $A \in \Xit_n^<$, which should be $\dBj \cap \JJ^\bL = \{\{A\}^\bL ~|~ A \in \Xit_n^<\}$ by the uniqueness of canonical basis.
\endproof

\prop  \label{prop:JK}
The following statements hold:
\enua
\item
The quotient algebra $\dKKj/\JJ$ admits a monomial basis $\{m_A + \JJ ~|~ A\in \Xit_n^>\}$.
\item
The specialization at $u = \bv^{\bL(s_0)}, v= \bv^{\bL(s_1)}$ of the quotient algebra $\dKKj/\JJ$ admits a canonical basis $\{\{A\}^\bL + \JJ^\bL ~|~ A \in \Xit_n^>\}$.
\item
The map $\sharp: \dKKj/\JJ \rw \dKKjp$ sending $[A] + \JJ \mapsto [A]$ is an isomorphism of $\bbA$-algebras,
which matches the corresponding monomial bases. It also matches the corresponding canonical bases at the specialization $u = \bv^{\bL(s_0)}, v= \bv^{\bL(s_1)}$.
\endenua
\endprop
%-----------------------------------------------------------------------------------------------------------
\proof
Parts (a) and (b) follow directly from Theorem~\ref{thm:monoJ}. Below we prove the Part (c).
Knowing that the map $\sharp$ is a linear isomorphism, we need to verify it is an algebraic homomorphism. Comparing the multiplication formulas for $\dKKj$ in Proposition \ref{prop:multK} with the ones for $\dKKjp$ in Proposition \ref{mfinKl}, we can see that the structure constants with respect to the Chevalley generators for $\dKKj/\JJ$ are as the same as those for $\dKKjp$. Therefore $\sharp$ is an algebraic homomorphism.

Since $\sharp$ matches the Chevalley generators, it matches the corresponding monomial bases. We also obtain that $\sharp$ commutes with the bar involution. Notice that the partial orders $<_\alg$ are compatible, hence $\sharp$ also matches the corresponding canonical bases at the specialization $u = \bv^{\bL(s_0)}, v= \bv^{\bL(s_1)}$.
\endproof

We summarize Lemma~\ref{lem:Kp} and Proposition~\ref{prop:JK} above as follows.

\begin{thm} \label{thm:KiCB}
As an $\bbA$-algebra, $\dKKi$ is isomorphic to a subquotient of $\dKKj$, with compatible standard, monomial basis.
They have compatible canonical bases at the specialization $u = \bv^{\bL(s_0)}, v= \bv^{\bL(s_1)}$.
\end{thm}

Let $\dot{\KK}_n^{\jmath,1}$ be the $\bbA$-submodule of $\dKKj$ spanned by $[A]$ where $A\in\Xit_{n}$ with $\ro(A)_{0}=\co(A)_0=1$. It is clear that $\dot{\KK}_n^{\jmath,1}$ is a subalgebra of $\dKKj$.
Let $\JJ^1=\JJ\cap \dot{\KK}_n^{\jmath,1}$, i.e.
$$\JJ^1=\mbox{span}_{\bbA}\{[A]~|~A\in\Xit_{n},\ro(A)_{0}=\co(A)_0=1,a_{00}<0\}.$$
Imitating the argument in \cite[\S A.3]{BKLW18}, we have the following.
\begin{prop}
\

\enua
\item
The monomial basis of ${\dKKj}$ restricts to the monomial basis of $\dot{\KK}_n^{\jmath,1}$; the monomial basis of $\dot{\KK}_n^{\jmath,1}$ restricts to the monomial basis of $\JJ^1$. So does the canonical basis at the specialization $u = \bv^{\bL(s_0)}, v= \bv^{\bL(s_1)}$.
\item
The quotient $\bbA$-subalgebra $\dot{\KK}_n^{\jmath,1}/\JJ^1$ admits a monomial basis $\{m_A + \JJ^1 ~|~ A\in \iXit\}$. It also admits a canonical basis $\{\{A\}^\bL + \JJ^{1,\bL} ~|~ A \in \iXit\}$ at the specialization $u = \bv^{\bL(s_0)}, v= \bv^{\bL(s_1)}$, where $\JJ^{1,\bL}=\JJ^1|_{u = \bv^{\bL(s_0)}, v= \bv^{\bL(s_1)}}$.
\item
There is an $\bbA$-algebra isomorphism $\dot{\KK}_n^{\jmath,1}/\JJ^1\cong \dKKi$, which matches the corresponding monomial bases. It also matches the corresponding canonical basis at the specialization $u = \bv^{\bL(s_0)}, v= \bv^{\bL(s_1)}$.
\endenua
\end{prop}

%=============================================================================================
\section{Quantum symmetric pairs}\label{sec:QSP}
%==============================================================================================

\subsection{The quantum symmetric pair $(\UU,\UUj)$}\label{sec:(U,Uj)}
We start with the quantum symmetric pairs of type AIII/AIV without fixed points nor black nodes, associated with the following Satake diagram:
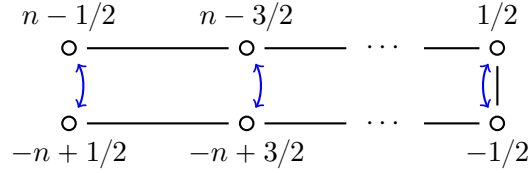
\[
\begin{tikzpicture}[thick]
\matrix [column sep={0.6cm}, row sep={0.5 cm,between origins}, nodes={draw = none,  inner sep = 3pt},ampersand replacement=\&]
{
	\node(U1) [draw, circle, fill=white, scale=0.6, label={above:$n-1/2$}] {};
	\&\node(U2) [draw, circle, fill=white, scale=0.6, label={above:$n-3/2$}] {};
	\&\node(U3) {$\cdots$};
	\&\node(U4) [draw, circle, fill=white, scale=0.6, label={above:$1/2$}] {};
\\
	\&\&\&
\\
	\node(L1) [draw, circle, fill=white, scale=0.6, label={below:$-n+1/2$}] {};
	\&\node(L2) [draw, circle, fill=white, scale=0.6, label={below:$-n+3/2$}] {};
	\&\node(L3) {$\cdots$};
	\&\node(L4) [draw, circle, fill=white, scale=0.6, label={below:$-1/2$}] {};
\\
};
\begin{scope}
\draw (U1) -- node  {} (U2);
\draw (U2) -- node  {} (U3);
\draw (U3) -- node  {} (U4);
\draw (U4) -- node  {} (L4);
\draw (L1) -- node  {} (L2);
\draw (L2) -- node  {} (L3);
\draw (L3) -- node  {} (L4);
\draw (L1) edge [color = blue,<->, bend right, shorten >=4pt, shorten <=4pt] node  {} (U1);
\draw (L2) edge [color = blue,<->, bend right, shorten >=4pt, shorten <=4pt] node  {} (U2);
\draw (L4) edge [color = blue,<->, bend left, shorten >=4pt, shorten <=4pt] node  {} (U4);
\end{scope}
\end{tikzpicture}
\]
Note that we use half integers for the index set following the convention in \cite{BW13}.
Set
\eq
\textstyle
\II_{2n}=\left\{
-n+\frac{1}{2},-n+\frac{3}{2},\ldots,n-\frac{1}{2}\right\}
\quad\mbox{and}\quad
\IIj_n=\left\{
\frac{1}{2},\frac{3}{2},\ldots,n-\frac{1}{2}
\right\}.
\endeq
Let $\mathbb{U}=\mathbb{U}(\fgl_{2n+1})$ be the algebra over $\QQ(u,v)$ generated by
$E_i, F_i$, $(i\in\II_{2n})$ and $D_a$, $(a\in[-n,n])$ subject to the following relations,
for $i,j\in\II_{2n}, a,b\in [-n,n]$:
\begin{align}
&D_aD_a^{-1}=D_a^{-1}D_a=1, \quad
D_aD_b=D_bD_a,
\\
&D_aE_jD_a^{-1}=v^{\delta_{a,j-\frac{1}{2}}-\delta_{a,j+\frac{1}{2}}}E_j,\quad D_aF_jD_a^{-1}=v^{-\delta_{a,j-\frac{1}{2}}+\delta_{a,j+\frac{1}{2}}}F_j,
\\
%K_iE_jK_i^{-1}&=v^{2\delta_{i,j}-\delta_{i,j+1}-\delta_{i,j-1}}E_j,\quad K_iF_jK_i^{-1}=v^{-2\delta_{i,j}+\delta_{i,j+1}+\delta_{i,j-1}}F_j,\\
&E_iF_j-F_jE_i=\delta_{i,j}\frac{K_{i}-K_{i}^{-1}}{v-v^{-1}},
\\
&E_i^2E_j+E_jE_i^2=(v+v^{-1})E_iE_jE_i,\quad F_i^2F_j+F_jF_i^2=(v+v^{-1})F_iF_jF_i, \quad
&(|i-j|=1),
\\
&E_iE_j=E_jE_i, \quad F_iF_j=F_jF_i,\quad
&(|i-j|>1).
\end{align}
(Here and below $K_i:=D_{i-\frac{1}{2}}D_{i+\frac{1}{2}}^{-1}$.)

Let $\UUj = \UUj(\fgl_{2n+1})$ be the $\QQ(u,v)$-algebra with generators
\[
e_i, f_i,\quad (i \in \IIj_n),\quad d_a^{\pm1} \quad(0 \leq a \leq n),
\]
subject to the following relations,  for $i\in\IIj_n,a,b\in [0,n]$:
\begin{align}
&d_ad_a^{-1}=1 =d_a^{-1}d_a, \quad d_ad_b=d_bd_a,\\
&d_0e_\frac{1}{2}d_0^{-1}=v^{2}e_\frac{1}{2},
\quad
d_0f_\frac{1}{2}d_0^{-1}=v^{-2}f_\frac{1}{2},
\\
& d_ae_jd_a^{-1}=v^{\delta_{a,j-\frac{1}{2}}-\delta_{a,j+\frac{1}{2}}}e_j,\quad d_af_jd_a^{-1}=v^{-\delta_{a,j-\frac{1}{2}}+\delta_{a,j+\frac{1}{2}}}f_j,
&\textstyle ((a,j)\neq (0,\frac{1}{2})),
\\
&e_if_j-f_je_i=\delta_{i,j}\frac{k_{i}-k_{i}^{-1}}{v-v^{-1}},
&\textstyle ((i,j)\neq (\frac{1}{2},\frac{1}{2})),
\\
&e_ie_j=e_je_i, \quad f_if_j=f_jf_i,
&\textstyle (|i-j|>1),
\\
&e_i^2e_j+e_je_i^2
=
(v+v^{-1})e_ie_je_i, \quad f_i^2f_j+f_jf_i^2=(v+v^{-1})f_if_jf_i,
&\textstyle (|i-j|=1),
\\
&e_{\frac{1}{2}}^2f_{\frac{1}{2}}+f_{\frac{1}{2}}e_{\frac{1}{2}}^2
=
(v+v^{-1})
\left(
  e_{\frac{1}{2}}f_{\frac{1}{2}}e_{\frac{1}{2}}
  - e_{\frac{1}{2}}(uvk_{\frac{1}{2}}+u^{-1}v^{-1}k_{\frac{1}{2}}^{-1})
\right),
\\
&f_{\frac{1}{2}}^2e_{\frac{1}{2}}+e_{\frac{1}{2}}f_{\frac{1}{2}}^2
=
(v+v^{-1})
\left(
  f_{\frac{1}{2}}e_{\frac{1}{2}}f_{\frac{1}{2}}
 -(uvk_{\frac{1}{2}}
 +u^{-1}v^{-1}k_{\frac{1}{2}}^{-1})f_{\frac{1}{2}}
 \right).
\end{align}
(Here $k_i=d_{i-\frac{1}{2}}d_{i+\frac{1}{2}}^{-1}$, $(i\not=\frac{1}{2})$, and $k_\frac{1}{2}=v^{-1}d_0d_1^{-1}$.)

It is known in \cite[\S 4.1]{BWW18} that there is a $\QQ(u,v)$-algebra homomorphism $\UUj \to \UU$ given by, for $i \in \IIj_n-\{\frac{1}{2}\}$, and for $1\leq  a \leq n$,
\eq\label{eq:imbeddingj}
\begin{array}{lll}
d_0 \mapsto v^{-1}D_0^{2},&e_i\mapsto E_{i}+F_{-i}K_{i}^{-1}& e_{\frac{1}{2}}\mapsto E_{\frac{1}{2}}+u^{-1}F_{-\frac{1}{2}}K_{\frac{1}{2}}^{-1},
\\
d_a\mapsto D_aD_{-a},&f_{i}\mapsto E_{-i}+K_{-i}^{-1}F_{i},  &  f_{\frac{1}{2}}\mapsto  E_{-\frac{1}{2}}+u K_{-\frac{1}{2}}^{-1}F_{\frac{1}{2}}.
\end{array}
\endeq
%------------------------------------------------------------------------------------------------------
\rmk
The (multiparameter) quantum symmetric pairs $(\UU,\UUj)$ in this paper are the $\fgl$-variant of the quantum symmetric pairs in \cite{BWW18}.
%A remark on the $\fsl$ versus $\fgl$ phenomenon for the equal parameter case can be found in \cite[Remark~4.3]{BKLW18}.
%We note that the notations used in this paper are different from the ones used in \cite{BWW18}, and the correspondence is as follows:
%\eq
%\begin{array}{cccc}
%u & \leftrightarrow p^{-1},\quad v\leftrightarrow q^{-1}, \\
%E_i & \leftrightarrow F_i, \quad F_i\leftrightarrow E_i, \quad K_i\leftrightarrow K_i,\\
%e_i &\leftrightarrow f_i, \quad f_i\leftrightarrow e_i, \quad k_i\leftrightarrow k_i.
%u  \leftrightarrow p^{-1},&E_i  \leftrightarrow F_i&F_i\leftrightarrow E_i&K_i\leftrightarrow K_i \\
%v \leftrightarrow q^{-1}, &e_i \leftrightarrow f_i,& f_i\leftrightarrow e_i,& k_i\leftrightarrow k_i.
%\end{array}
%\endeq
\endrmk
%-------------------------------------------------------------------------------------------------------
\subsection{Isomorphism $\dot{\mathbb{U}}^\jmath\simeq \dKKj$} \label{sec:KjUj}
%--------------------------------------------------------------------------------------------------------
Following \cite[\S 23.1]{Lu93}, it is routine to define the modified quantum algebra $\dot{\mathbb{U}}^\jmath$ from $\mathbb{U}^\jmath$.
Let $\Xit_n^\diag$ be the set of all diagonal matrices in $\Xit_n$. Denote by $\ld=\diag(\ld_{-n},\ld_{-n+1},\ldots, \ld_n)$ a diagonal matrix in $\Xit_n^\diag$. For $\ld,\ld'\in\Xit_n^\diag$, we set
\eq
{}_\ld\mathbb{U}^\jmath_{\ld'}=
\mathbb{U}^\jmath/\left(\sum_{a=0}^n(d_a-v^{\ld_a})\mathbb{U}^\jmath
+\sum_{a=0}^n\mathbb{U}^\jmath(d_a-v^{\ld'_a})\right).
\endeq
The modified quantum algebra $\dot{\mathbb{U}}^\jmath$ is defined by
\eq
\dot{\mathbb{U}}^\jmath=\bigoplus_{\ld,\ld'\in\Xit_n^\diag}{}_\ld\dot{\mathbb{U}}^\jmath_{\ld'}.
\endeq
Let $1_\ld=p_{\ld,\ld}(1)$, where $p_{\ld,\ld}: \mathbb{U}^\jmath\rightarrow{}_\ld\dot{\mathbb{U}}^\jmath_{\ld}$ is the
canonical projection. Thus the unit of $\mathbb{U}^\jmath$ is replaced by a collection of orthogonal idempotents $1_\ld$ in $\dot{\mathbb{U}}^\jmath$. It is clear that
\[
\dot{\mathbb{U}}^\jmath=\sum_{\ld\in\Xit_n^\diag}\mathbb{U}^\jmath 1_{\ld}
=\sum_{\ld\in\Xit_n^\diag}1_{\ld}\mathbb{U}^\jmath.
\]
For $\ld \in \Xit_n^\diag$ and $i\in\mathbb{I}_n^\jmath$, we use the following short-hand notations:
\eq
\ld+\alpha_i = \ld+E^\theta_{i-\frac{1}{2},i-\frac{1}{2}}-E^\theta_{i+\frac{1}{2},i+\frac{1}{2}},
\quad
\ld-\alpha_i = \ld - E^\theta_{i-\frac{1}{2},i-\frac{1}{2}} + E^\theta_{i+\frac{1}{2},i+\frac{1}{2}}.
\endeq
We also define, for $r \in \NN$,
\eq
\LR{r} = \frac{v^r - v^{-r}}{v-v^{-1}}.
\endeq
A multiparameter version of \cite[Proposition~4.6]{BKLW18} gives a presentation of $\dUUj$ as a $\QQ(u,v)$-algebra generated by the symbols, for $i \in \IIj_n, \ld \in \Xit_n^\diag,$
\[
1_\ld,
\quad
e_i 1_\ld,
\quad
1_\ld e_i,
\quad
f_i 1_\ld,
\quad
1_\ld f_i,
\]
subject to the following relations, for $i,j \in \IIj_n, \ld,\mu \in \Xit_n^\diag$, $x,y \in \{1, e_i, e_j, f_i, f_j\}$:
\begin{align}
&x 1_\ld 1_\mu y = \delta_{\ld,\mu} x 1_\ld y,\label{ex:xly}
\\
&e_i 1_\ld = 1_{\ld+\alpha_i} e_i,
\quad
f_i 1_\ld = 1_{\ld-\alpha_i} f_i,
\\
&e_i 1_\ld f_j = f_j 1_{\ld+\alpha_i+\alpha_j} e_i,
& (i \neq j),
\\
&(e_if_i-f_ie_i)1_\ld=\LR{\ld_{i-\frac{1}{2}} - \ld_{i+\frac{1}{2}} }1_{\ld},
&\textstyle (i \neq \frac{1}{2}),
\\
&e_ie_j1_\ld=e_je_i1_\ld, \quad f_if_j1_\ld=f_jf_i1_\ld,
&\textstyle (|i-j|>1),
\\
&(e_i^2e_j+e_je_i^2)1_\ld
=
\LR{2}e_ie_je_i1_\ld, \quad (f_i^2f_j+f_jf_i^2)1_\ld=\LR{2}f_if_jf_i1_\ld,
&\textstyle (|i-j|=1),
\\
&(\LR{2}e_\frac{1}{2}f_\frac{1}{2}e_\frac{1}{2}-e_{\frac{1}{2}}^2f_{\frac{1}{2}}-f_{\frac{1}{2}}e_{\frac{1}{2}}^{2})1_\ld
=\LR{2}(uv^{\ld_0-\ld_1}+u^{-1}v^{-\ld_0+\ld_1})e_\frac{1}{2} 1_\ld, \label{eq:eef}
\\
&
(\LR{2}f_{\frac{1}{2}}e_{\frac{1}{2}}f_{\frac{1}{2}}-f_{\frac{1}{2}}^2e_{\frac{1}{2}}-e_{\frac{1}{2}}f_{\frac{1}{2}}^2) 1_\ld
=\LR{2}
(uv^{\ld_0-\ld_1-3}+u^{-1}v^{-\ld_0+\ld_1+3}) f_{\frac{1}{2}}
 1_\ld.\label{ex:ffe}
\end{align}
Here and below we always write $x_1 1_{\ld^1} x_2 1_{\ld^2}\cdots x_k 1_{\ld^k}=x_1x_2\cdots x_k 1_{\ld^k}$, if the product is not zero; in this case such $\ld^1,\ld^2,\ldots,\ld^{k-1}$ are all uniquely determined by $\ld^k$.

For $\forall i\in\mathbb{I}^\jmath_n, \ld\in\Xit_n^\diag$,
write $$\be_i 1_\ld=[\ld -E^\theta_{i+\frac{1}{2},i+\frac{1}{2}}+E^\theta_{i-\frac{1}{2},i+\frac{1}{2}}]\in \dKKj \quad \mbox{and}\quad \bbf_i 1_\ld = [\ld -E^\theta_{i-\frac{1}{2},i-\frac{1}{2}}+E^\theta_{i+\frac{1}{2},i-\frac{1}{2}}]\in \dKKj.$$
Set ${}_\QQ\dKKj=\QQ(u,v)\otimes_\bbA \dKKj$.

\begin{thm}\label{thm:Kj=Uj}
There is an isomorphism of $\QQ(u,v)$-algebras $\aleph:\dot{\mathbb{U}}^\jmath\rightarrow {}_{\QQ}\dKKj$ such that, for $\forall i\in\mathbb{I}^\jmath_n, \ld\in\Xit_n^\diag$,
\[
e_i 1_\ld \mapsto \be_i 1_\ld,
\quad
f_i 1_\ld  \mapsto \bbf_i 1_\ld,
\quad
1_\ld  \mapsto [\ld].
\]
\end{thm}
\begin{proof}
A direct computation using Theorem~\ref{thm:multformula2} shows that relations~ \eqref{ex:xly}-\eqref{ex:ffe} also hold if we replace $e_i, f_i$'s by $\be_i, \bbf_i$'s.
Here we only present details for \eqref{eq:eef} regarding $\be_\frac{1}{2} 1_\ld$ and $\bbf_\frac{1}{2} 1_\ld$ as follows:
\begin{align*}
\be_{\frac{1}{2}}^2\bbf_{\frac{1}{2}}1_\ld
&=u^{-2}v^{-2\ld_0-\ld_1+4}[2](e_{\ld-2E_{1,1}^\theta+2E_{0,1}^\theta-E_{0,0}^{\theta}+E_{1,0}^{\theta}}
+[\ld_0-1]_\fc e_{\ld-E_{1,1}^{\theta}+E_{0,1}^{\theta}}),
\\
\bbf_{\frac{1}{2}}\be_{\frac{1}{2}}^{2}1_\ld
&=u^{-2}v^{-2\ld_0-\ld_1+2}[2](e_{\ld-2E_{1,1}^\theta+E_{0,1}^\theta+E_{1,-1}^{\theta}}
+e_{\ld-2E_{1,1}^\theta+2E_{0,1}^\theta-E_{0,0}^{\theta}+E_{1,0}^{\theta}}+[\ld_1-1]e_{\ld-E_{1,1}^{\theta}+E_{0,1}^{\theta}}),
\\
\be_\frac{1}{2}\bbf_\frac{1}{2}\be_\frac{1}{2}1_\ld
&=u^{-2}v^{-2\ld_0-\ld_1+3}(e_{\ld-2E_{1,1}^\theta+E_{0,1}^\theta+E_{1,-1}^{\theta}}+
[2]e_{\ld-2E_{1,1}^\theta+2E_{0,1}^\theta-E_{0,0}^{\theta}+E_{1,0}^{\theta}}
\\
\nonumber&+(u^2v^{2\ld_0-2}+[\ld_0-1]_\fc v^2+[\ld_1])e_{\ld-E_{1,1}^{\theta}+E_{0,1}^{\theta}}),
\end{align*}
Combining the identities above, we get
$(\LR{2}\be_\frac{1}{2}\bbf_\frac{1}{2}\be_\frac{1}{2}-\be_{\frac{1}{2}}^2\bbf_{\frac{1}{2}}-\bbf_{\frac{1}{2}}\be_{\frac{1}{2}}^{2})1_\ld
=\LR{2}(uv^{\ld_0-\ld_1}+u^{-1}v^{-\ld_0+\ld_1})\be_\frac{1}{2} 1_\ld$.
That is, $\aleph$ is indeed an algebra homomorphism.

We also know that $\aleph$ is a linear isomorphism. The argument is almost as the same as that for the case of specialization at $u=v$, which can be found in the proof of \cite[Theorem 4.7]{BKLW18}. Therefore $\aleph$ is an isomorphism of $\QQ(u,v)$-algebras.
\end{proof}

It has been shown in \cite[Lemma 4.1]{BWW18} that there exists a unique $\QQ$-linear bar involution on $\mathbb{U}^\jmath$ such that $\overline{u}=u^{-1}, \overline{v}=v^{-1}, \overline{d_a}=d_a^{-1}\ (0 \leq a \leq n), \overline{e_i}=e_i, \overline{f_i}=f_i\ (i\in\IIj_n)$. This bar involution on $\mathbb{U}^\jmath$ induces a compatible bar involution on $\dot{\mathbb{U}}^\jmath$, denoted also by ${}^-$, fixing all the generators $1_\ld$, $e_i 1_\ld$, $f_i 1_\ld$.

Note that $\be_i 1_\ld$, $\bbf_i 1_\ld$, $[\ld]$ are bar invariant elements in $\dKKj$, which implies that the isomorphism $\aleph$ intertwines the bar involution on $\dot{\mathbb{U}}^\jmath$ and on ${}_\QQ\dKKj$.

Set ${}_{\bbA}\dot{\mathbb{U}}^\jmath=\aleph^{-1}(\dKKj)$. It is an $\bbA$-subalgebra of $\dot{\mathbb{U}}^\jmath$. We have the following result.
\begin{prop}
The integral form ${}_{\bbA}\dot{\mathbb{U}}^\jmath$ is a free $\bbA$-submodule of $\dot{\mathbb{U}}^\jmath$. It is stable under the bar involution.
\end{prop}

%-------------------------------------------------------------------------------------------------------------
\subsection{The quantum symmetric pair $(\UU,\UUi)$}
%--------------------------------------------------------------------------------------------------------------
Below we formulate the counterparts of Sections~\ref{sec:(U,Uj)}--\ref{sec:KjUj}.
The proofs are very similar and will often be omitted.
We now work on quantum symmetric pairs of type AIII with fixed points associated with the Satake diagram below:
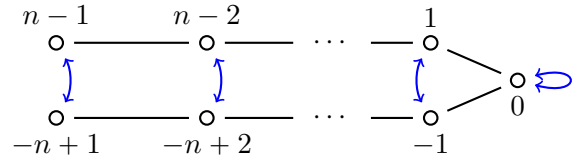
\[
\begin{tikzpicture}[thick]
\matrix [column sep={0.6cm}, row sep={0.5 cm,between origins}, nodes={draw = none,  inner sep = 3pt},ampersand replacement=\&]
{
	\node(U1) [draw, circle, fill=white, scale=0.6, label={above:$n-1$}] {};
	\&\node(U2)[draw, circle, fill=white, scale=0.6, label={above:$n-2$}] {};
	\&\node(U3) {$\cdots$};
	\&\node(U5)[draw, circle, fill=white, scale=0.6, label={above:$1$}] {};
\\
	\&\&\&\&
	\node(R)[draw, circle, fill=white, scale=0.6, label={below:$0$}] {};
\\
	\node(L1) [draw, circle, fill=white, scale=0.6, label={below:$-n+1$}] {};
	\&\node(L2)[draw, circle, fill=white, scale=0.6, label={below:$-n+2$}] {};
	\&\node(L3) {$\cdots$};
	\&\node(L5)[draw, circle, fill=white, scale=0.6, , label={below:$-1$}] {};
\\
};
\begin{scope}
\draw (U1) -- node  {} (U2);
\draw (U2) -- node  {} (U3);
\draw (U3) -- node  {} (U5);
\draw (U5) -- node  {} (R);
\draw (L1) -- node  {} (L2);
\draw (L2) -- node  {} (L3);
\draw (L3) -- node  {} (L5);
\draw (L5) -- node  {} (R);
\draw (R) edge [color = blue,loop right, looseness=40, <->, shorten >=4pt, shorten <=4pt] node {} (R);
\draw (L1) edge [color = blue,<->, bend right, shorten >=4pt, shorten <=4pt] node  {} (U1);
\draw (L2) edge [color = blue,<->, bend right, shorten >=4pt, shorten <=4pt] node  {} (U2);
\draw (L5) edge [color = blue,<->, bend left, shorten >=4pt, shorten <=4pt] node  {} (U5);
\end{scope}
\end{tikzpicture}
\]
Let $\mathbb{U}=\mathbb{U}(\fgl_{2n})$ be the algebra over $\QQ(u,v)$ generated by
$E_i, F_i$, $(i\in[-n+1,n-1])$ and $D_a$, $(a\in[-n+1,n])$ subject to the following relations,
for $i,j\in[-n+1,n-1], a,b\in [-n+1,n]$:
\begin{align}
&D_aD_a^{-1}=D_a^{-1}D_a=1, \quad
D_aD_b=D_bD_a,
\\
&D_aE_jD_a^{-1}=v^{\delta_{a,j}-\delta_{a,j+1}}E_j,\quad D_aF_jD_a^{-1}=v^{-\delta_{a,j}+\delta_{a,j+1}}F_j,
\\
%K_iE_jK_i^{-1}&=v^{2\delta_{i,j}-\delta_{i,j+1}-\delta_{i,j-1}}E_j,\quad K_iF_jK_i^{-1}=v^{-2\delta_{i,j}+\delta_{i,j+1}+\delta_{i,j-1}}F_j,\\
&E_iF_j-F_jE_i=\delta_{i,j}\frac{K_{i}-K_{i}^{-1}}{v-v^{-1}},
\\
&E_i^2E_j+E_jE_i^2=(v+v^{-1})E_iE_jE_i,\quad F_i^2F_j+F_jF_i^2=(v+v^{-1})F_iF_jF_i, \quad
&(|i-j|=1),
\\
&E_iE_j=E_jE_i, \quad F_iF_j=F_jF_i,\quad
&(|i-j|>1).
\end{align}
(Here and below $K_i:=D_{i}D_{i+1}^{-1}$.)

Let
$\UUi = \UUi(\fgl_{2n})$ be the $\QQ(u,v)$-algebra with generators
\[
t,\quad e_i, \quad f_i\quad (i \in [1, n-1]),\quad d_a^{\pm1} \quad(a\in[1,n]),
\]
subject to the following relations,  for $i,j\in [1, n-1], a,b\in [1,n]$:
\begin{align}
&d_ad_a^{-1}=1 =d_a^{-1}d_a,
\quad d_ad_b=d_bd_a,\\
&d_a t d_a^{-1}=t,\quad d_ae_jd_a^{-1}=v^{\delta_{a,j}-\delta_{a,j+1}}e_j,
\quad d_af_jd_a^{-1}=v^{-\delta_{a,j}+\delta_{a,j+1}}f_j,
\\
&e_if_j-f_je_i=\delta_{i,j}\frac{k_{i}-k_{i}^{-1}}{v-v^{-1}},
\\
&e_ie_j=e_je_i, \quad f_if_j=f_jf_i,
&\textstyle (|i-j|>1),
\\
&e_i^2e_j+e_je_i^2
=
(v+v^{-1})e_ie_je_i, \quad f_i^2f_j+f_jf_i^2=(v+v^{-1})f_if_jf_i,
&\textstyle (|i-j|=1),
\\
&
e_i t = te_i,\quad f_i t = t f_{i},&(i \neq 1),
\\
& t^2 e_1 + e_1 t^2 = (v+v^{-1})  t e_1 t + e_1,
\quad
 e_1^2 t + t e_1^2 = (v+v^{-1})  e_1 t e_1,
 \\
&t^2 f_1 + f_1 t^2 =(v+v^{-1})  t f_1 t + f_1,
\quad
f_1^2 t + tf_1^2
= (v+v^{-1})  f_1 t f_1.
\end{align}
(Here $k_i=d_{i}d_{i+1}^{-1}$.)

It has been known in \cite[\S 2.1]{BWW18} that there is a $\QQ(u,v)$-algebra homomorphism $\UUi \to \UU$ given by, for $i \in [1,n-1]$, and for $a\in[1,n]$,
\eq\label{eq:imbeddingj}
\begin{array}{ll}
d_a=D_aD_{-a},& t=E_0+vF_0K_0^{-1}+\frac{u-u^{-1}}{v-v^{-1}}K_0^{-1},
\\
e_i=E_{i}+F_{-i}K_{i}^{-1},& f_{i}=E_{-i}+K_{-i}^{-1}F_{i}.
\end{array}
\endeq

\rmk\label{rmk:embedding}
It was observed in \cite{Le99,BWW18} that the parameter $\omega\in\QQ(u,v)$ in the embedding $t=E_0+vF_0K_0^{-1}+\omega K_0^{-1}$, is irrelevant to the presentation of the algebra $\UUi$.
\endrmk

Let ${}^\imath\Xit_n^\diag$ be the set of all diagonal matrices in $\Xit^\imath$. Denote by $\ld=\diag(\ld_{-n},\ldots,\ld_{-1},1,\ld_{1}, \ldots, \ld_n)$ a diagonal matrix in ${}^\imath\Xit_n^\diag$.
We define the modified algebra $\dUUi$ similarly to the construction of $\dUUj$ as follows:
$$\dot{\mathbb{U}}^\imath=\bigoplus_{\ld,\ld'\in{}^\imath\Xit_n^\diag}{}_\ld\dot{\mathbb{U}}^\imath_{\ld'}=\sum_{\ld\in{}^\imath\Xit_n^\diag}\mathbb{U}^\imath 1_{\ld}
=\sum_{\ld\in{}^\imath\Xit_n^\diag}1_{\ld}\mathbb{U}^\imath,$$
where
$
{}_\ld\mathbb{U}^\imath_{\ld'}=
\mathbb{U}^\imath/\left(\sum_{a=1}^n(d_a-v^{\ld_a})\mathbb{U}^\imath
+\sum_{a=1}^n\mathbb{U}^\imath(d_a-v^{\ld'_a})\right)
$
and $1_\ld\in{}_\ld\mathbb{U}^\imath_{\ld}$ is the canonical projection image of the unit of $\mathbb{U}^\imath$.

For $\ld \in {}^\imath\Xit_n^\diag$ and $i\in [1,n-1]$, we use the following short-hand notations:
\eq
\ld+\alpha_i = \ld+E^\theta_{ii}-E^\theta_{i+1,i+1},
\quad
\ld-\alpha_i = \ld-E^\theta_{ii}+E^\theta_{i+1,i+1}.
\endeq
We thus obtain a presentation of $\dUUi$ as a $\QQ(u,v)$-algebra generated by the symbols, for $i \in [1,n-1], \ld \in {}^\imath\Xit_n^\diag,$
\[
1_\ld,
\quad
t1_\ld,
\quad
1_\ld t,
\quad
e_i 1_\ld,
\quad
1_\ld e_i,
\quad
f_i 1_\ld,
\quad
1_\ld f_i,
\]
subject to the following relations, for $i,j \in [1,n-1], \ld,\mu \in {}^\imath\Xit_n^\diag$, $x,y \in \{1, e_i, e_j, f_i, f_j, t\}$:
\begin{align}
&x 1_\ld 1_\mu y = \delta_{\ld,\mu} x 1_\ld y,\label{ex:kappi1}
\\
&e_i 1_\ld = 1_{\ld+\alpha_i} e_i,
\quad
f_i 1_\ld = 1_{\ld-\alpha_i} f_i,
\quad
t 1_\ld = 1_\ld t,
\\
&e_i 1_\ld f_j = f_j 1_{\ld+\alpha_i+\alpha_j} e_i,
& (i \neq j),
\\
&(e_if_i-f_ie_i)1_\ld=\LR{\ld_{i} - \ld_{i+1} }1_{\ld},
\quad
\\
&e_ie_j1_\ld=e_je_i1_\ld, \quad f_if_j1_\ld=f_jf_i1_\ld,
&\textstyle (|i-j|>1),
\\
&(e_i^2e_j+e_je_i^2)1_\ld
=
\LR{2}e_ie_je_i1_\ld, \quad (f_i^2f_j+f_jf_i^2)1_\ld=\LR{2}f_if_jf_i1_\ld,
&\textstyle (|i-j|=1),
\\
&f_i t1_\ld = t f_{i}1_\ld,
\quad
e_i t 1_\ld= te_i 1_\ld &(i \neq 1),
\\
&(t^2 f_1 + f_1 t^2)1_\ld = (\LR{2}t f_1 t + f_1)1_\ld,
\quad
(f_1^2 t + tf_1^2)1_\ld
= \LR{2} f_1 t f_1 1_\ld,
\\\label{ex:kappi2}
& (t^2 e_1 + e_1 t^2)1_\ld = (\LR{2}t e_1 t + e_1)1_\ld,
\quad
(e_1^2 t + t e_1^2)1_\ld = \LR{2}  e_1 t e_11_\ld.
\end{align}
For $i\in[1,n-1], \ld\in{}^\imath\Xit_n^\diag$,
write
\begin{align}
\be_i 1_\ld&=[\ld -E^\theta_{i+1,i+1}+E^\theta_{i,i+1}],\quad\quad\quad \bbf_i 1_\ld = [\ld -E^\theta_{i,i}+E^\theta_{i+1,i}],\nonumber\\
{\bf t}1_\ld&=[\ld-E_{1,1}^\theta+E_{-1,1}^\theta]+v^{-\ld_1}\frac{u - u\inv}{v - v\inv}[\ld].\label{eq:t1l}
\end{align}
Set ${}_\QQ\dKKj=\QQ(u,v)\otimes_\bbA \dKKi$.

\begin{thm}\label{thm:Ki=Ui}
There is an isomorphism of $\QQ(u,v)$-algebras $\aleph:\dot{\mathbb{U}}^\imath\rightarrow {}_\QQ\dKKi$ such that, for all $i\in[1,n-1], \ld\in\Xit_n^\diag$,
\[
t 1_\ld \mapsto {\bf t}1_\ld,
\quad
e_i 1_\ld \mapsto \be_i 1_\ld,
\quad
f_i 1_\ld  \mapsto \bbf_i 1_\ld ,
\quad
1_\ld  \mapsto [\ld].
\]
\end{thm}
\begin{proof}
By a direct computation using Theorem~\ref{thm:multformula2} one can show that the relations \eqref{ex:kappi1}-\eqref{ex:kappi2} for $t, e_i, f_i$'s also hold for ${\bf t}, \be_i, \bbf_i$'s. Hence $\aleph$ is a homomorphism of $\QQ(u,v)$-algebras.
Here we only present the details for the first relation in \eqref{ex:kappi2} as follows.
Note that as an element in $\dKKjp$,
\begin{equation}\label{eq:t}
{\bf t}1_\ld=\bbf_0\be_01_\ld-\frac{u^{-1}v^{\ld_1}-uv^{-\ld_1}}{v-v^{-1}}1_\ld,
\end{equation}
where  $\be_0 1_\ld=[\ld-E_{1,1}^\theta+E_{0,1}^\theta]\quad\mbox{and}\quad\bbf_0 1_{\ld+E_{0,0}^\theta-E_{1,1}^\theta}=[\ld-E_{1,1}^\theta+E_{1,0}^\theta]\in\dKKjp.$
Moreover, we have
\begin{align*}
{\bf t}^2 1_\ld=&\LR{2}u^{-1}v^{-2\ld_1+2}\frac{u-u^{-1}}{v-v^{-1}}e_{\ld-E_{1,1}^\theta+E_{-1,1}^{\theta}}
+\left(v^{-2\ld_1+2}[\ld_1]+v^{-2\ld_1}\frac{(u-u^{-1})^2}{(v-v^{-1})^2}\right)e_\ld\\
&+u^{-2}v^{-2\ld_1+2}[2]e_{\ld-2E_{1,1}^\theta+2E_{-1,1}^{\theta}}.
\end{align*}
Hence
\begin{align*}
\be_1{\bf t}^2 1_\ld=&\LR{2}u^{-1}v^{-3\ld_1+2}\frac{u-u^{-1}}{v-v^{-1}}
e_{\ld-E_{1,1}^\theta+E_{-1,1}^{\theta}-E_{2,2}^\theta+E_{1,2}^\theta}\\
&+\left(v^{-3\ld_1+2}[\ld_1]+v^{-3\ld_1}\frac{(u-u^{-1})^2}{(v-v^{-1})^2}\right)
e_{\ld-E_{2,2}^\theta+E_{1,2}^\theta}
+u^{-2}v^{-3\ld_1+2}[2]e_{\ld-2E_{1,1}^\theta+2E_{-1,1}^{\theta}-E_{2,2}^\theta+E_{1,2}^\theta},
\end{align*}
and
\begin{align*}
{\bf t}^2 \be_1 1_\ld=&\LR{2}u^{-1}v^{-3\ld_1}\frac{u-u^{-1}}{v-v^{-1}}
(e_{\ld-E_{1,1}^\theta+E_{-1,1}^{\theta}-E_{2,2}^\theta+E_{1,2}^\theta}+
e_{\ld-E_{2,2}^\theta+E_{-1,2}^\theta})\\
&+\left(v^{-3\ld_1}[\ld_1+1]+v^{-3\ld_1-2}\frac{(u-u^{-1})^2}{(v-v^{-1})^2}\right)
e_{\ld-E_{2,2}^\theta+E_{1,2}^\theta}\\
&+u^{-2}v^{-3\ld_1}[2]
(e_{\ld-2E_{1,1}^\theta+2E_{-1,1}^{\theta}-E_{2,2}^\theta+E_{1,2}^\theta}
+e_{\ld-E_{1,1}^\theta+E_{-1,1}^{\theta}-E_{2,2}^\theta+E_{-1,2}^\theta}).
\end{align*}
Finally, using \eqref{eq:t} again, we compute that
\begin{align*}
{\bf t}\be_1{\bf t} 1_\ld
=&\LR{2}u^{-1}v^{-3\ld_1+1}\frac{u-u^{-1}}{v-v^{-1}}
e_{\ld-E_{1,1}^\theta+E_{-1,1}^{\theta}-E_{2,2}^\theta+E_{1,2}^\theta}+
u^{-1}v^{-3\ld_1}\frac{u-u^{-1}}{v-v^{-1}}e_{\ld-E_{2,2}^\theta+E_{-1,2}^\theta}\\
&+\left(v^{-\ld_1}\frac{(1-v^{-2\ld_1})(v+v^{-1})}{v-v^{-1}}+v^{-3\ld_1-1}\frac{(u-u^{-1})^2}{(v-v^{-1})^2}\right)
e_{\ld-E_{2,2}^\theta+E_{1,2}^\theta}\\
&+u^{-2}v^{-3\ld_1+1}e_{\ld-E_{1,1}^\theta+E_{-1,1}^{\theta}-E_{2,2}^\theta+E_{-1,2}^\theta}
+u^{-2}v^{-3\ld_1+1}[2]e_{\ld-2E_{1,1}^\theta+2E_{-1,1}^{\theta}-E_{2,2}^\theta+E_{1,2}^\theta}.
\end{align*}
Combining the identities above, we see that indeed $({\bf t}^2 \be_1 + \be_1 {\bf t}^2)1_\ld = (\LR{2}{\bf t} \be_1 t + \be_1)1_\ld.$

An argument similar to the proof of \cite[Theorem A.15]{BKLW18} also shows $\aleph$ is a linear isomorphism. Therefore $\aleph$ is an isomorphism of $\QQ(u,v)$-algebras.
\end{proof}

Thanks to \cite[Lemma 2.1]{BWW18}, we know there exists a unique $\QQ$-algebra bar involution on $\dot{\mathbb{U}}^\imath$ such that
$\overline{u}=u^{-1}, \overline{v}=v^{-1}, \overline{d_a}=d_a^{-1}\ (a\in[1,n]), \overline{e_i}=e_i, \overline{f_i}=f_i\ (i\in[1,n-1]), \overline{t}=t$. This bar involution on $\mathbb{U}^\imath$ induces a compatible bar involution on $\dot{\mathbb{U}}^\imath$, denoted also by ${}^-$, fixing all the generators $1_\ld$, $e_i 1_\ld$, $f_i 1_\ld, t$.

Set ${}_{\bbA}\dot{\mathbb{U}}^\imath=\aleph^{-1}(\dKKi)$. It is an $\bbA$-subalgebra of $\dot{\mathbb{U}}^\imath$.
\begin{prop}
The integral form ${}_{\bbA}\dot{\mathbb{U}}^\imath$ is a free $\bbA$-submodule of $\dot{\mathbb{U}}^\imath$. It is stable under the bar involution.
\end{prop}

\rmk
Theorem \ref{thm:KCB} (resp. Theorem \ref{thm:monoJ}) provides a canonical basis for the modified form of $\mathbb{U}^{\jmath}$ (resp. $\mathbb{U}^{\imath}$) at the specialization $u = \bv^{\bL(s_0)}, v= \bv^{\bL(s_1)}$. A general theory of canonical bases for quantum symmetric pairs with parameters of arbitrary finite type was developed in \cite{BW16}.
\endrmk

%\rmk
%Similar to the phenomenon mentioned in Remark \ref{rmk:embedding}, the parameters %$\omega_1,\omega_2\in\QQ(u,v)$ for
%${\bf t}1_\ld=[\ld-E_{1,1}^\theta+E_{-1,1}^\theta]+(\omega_1 v^{\ld_1}+\omega_2 %v^{-\ld_1})[\ld]$ make no effect of Theorem \ref{thm:embedding}. We choose the formulation %\eqref{eq:t1l} for ${\bf t} 1_\ld$ in order to fit the formulation in \cite{BKLW18} (resp. %\cite{FL15}) for the specialization at $u=v$ (resp. $u=1$).
%\endrmk

%%%%%%%%%%%%%%%%%%%%%%%%%%%%%%%%%%%%%%%%%%%%%%%%%%%%%%%%%%%%%%%%%%%%%%%%%%%%%%%%%%%%%%%%%%%%%%%%%%%%%%%%%%%%
\newpage

\appendix

\section{An algebraic approach to Schur algebras of type D}
%==================================================================================
As we mentioned in Section~\ref{rmk:typeD}, at the specialization $u=1$ the multiparameter Schur duality yields a weak Schur duality of type D that is used in \cite{Bao17} to formulate the Kazhdan-Lusztig theory for classical and super type D.
These algebras $\mathbb{S}^\bullet_{n,d}|_{u=1}$ ($\bullet = \imath$ or $\jmath$), however, are not the Schur algebras introduced in \cite{FL15}.
While bases of Schur algebras of finite type A/B/C and affine type A/C can be parametrized by a matrix set (cf. $\Xi_{n,d}$ in \ref{def:Xi}), for finite type D Fan and Li showed that a matrix set is not enough -- a notion of signed matrices that indexes a larger algebra is needed.
From a geometric point of view, this reflects the fact that there are two connected components for the maximal isotropic Grassmannian associated to $\mrm{SO}(2d)$.
In this appendix, we provide an algebraic approach to Fan-Li's construction parallel to our multiparameter results.
The arguments are very similar to the multiparameter counterpart, so we will omit the easy proofs in this appendix.

\subsection{Weyl groups of type $\bD$}
Fix $d\in\mathbb{N}$, and we set
set
\eq
J_d=\{-d,\ldots,-1,1,\ldots,d\}.
\endeq
Let $W_{\bD}$ be the Weyl group of type $\bD_d$. It is known (c.f. \cite{BB05}) that $W_{\bD}$ can be identified as a permutation subgroup of $J_d$ which consists of those permutations $g$ satisfying that
\[
{}^\#
\{ i\in J_d \mid i>0, g(i)<0 \} \in 2\NN,
\quad
g(-i)=-g(i)
\quad
(1\leq i \leq d).
\]
Let $S_\bD=\{\varsigma_0,\varsigma_1,\ldots,\varsigma_{d-1}\}$, where $\varsigma \in W_\bD$ are given by the following products of transpositions:
\[
\varsigma_0=(1,-2)(2,-1) \quad\mbox{and}\quad \varsigma_i=(i,i+1)(-i-1,-i) \quad\mbox{for $i=1,\ldots,d$.}
\]
It is also known (see \cite[(8.18),(8.19)]{BB05}) that $(W_\bD, S_\bD)$ is a Coxeter group associated with the length function as below:
\begin{lemma}
The length of $g\in W_\bD$ is given by
\[
\ell(g)= {}^\#\{ (i,j)\in J_d^2 \mid |i|<j,g(i)>g(j) \}.
\]
\end{lemma}
%=================================================================================
\subsection{Signed compositions}
%=================================================================================
Fix $n\in\NN$.
Recall that \eqref{def:Ld} first $\Lambda_{n,d}$ is the set of weak compositions of $d$ into $n+1$ parts.
Set
\eq
\Lambda^0=\{\lambda\in \Lambda_{n,d}~|~\lambda_0>0\}\times\{0\},
\quad
\Lambda^\epsilon =\{\lambda\in \Lambda_{n,d}~|~\lambda_0=0\}\times\{\epsilon\},
\quad
(\epsilon = + \textup{ or }-).
\endeq
In below we abbreviate $(\lambda,\alpha)\in \Lambda^\alpha$ by $\lambda^\alpha$ where $\alpha\in\{0,+,-\}$.
We further set
\eq
\Lambda_{\bD}=\Lambda^0\sqcup\Lambda^+\sqcup\Lambda^-.
\endeq
Elements in $\Lambda_{\bD}$ will be called {\em signed compositions}.
%---------------------------------------------------------------------------
Recall that $\ld_{0,i} =\ld_0 + \ld_1 + \cdots + \ld_i$ for $i\in[0,n], \lambda\in\Lambda_{n,d}$. We define positive integer intervals associated to $\lambda^\alpha$ by
\begin{equation}
R_i^{\lambda^0}=
\left\{\begin{array}{ll}
[-\lambda_0,\lambda_0]\setminus\{0\} & \mbox{if $i=0$};\\
{[\ld_{0,i-1}+1,\ld_{0,i}]} & \mbox{if $i\in[1,n]$},
\end{array}\right.
\end{equation}
\eq
R_i^{\lambda^+}=
\left\{\begin{array}{ll}
\emptyset & \mbox{if $i=0$};\\
{[1,\lambda_1]}& \mbox{if $i=1$};\\
{[\ld_1+1,\ld_{0,i}]} & \mbox{if $i\in[2,n]$},
\end{array}\right.
\quad
R_i^{\lambda^-}=
\left\{\begin{array}{ll}
\emptyset & \mbox{if $i=0$};\\
\{-1,2,\ldots,\lambda_1\}& \mbox{if $i=1$};\\
{[\ld_1+1,\ld_{0,i}]} & \mbox{if $i\in[2,n]$}.
\end{array}\right.
\endeq
For $-n\leq i \leq 1$, we set $R_{i}^{\lambda^\alpha}=\{-x|x\in R_{-i}^{\lambda^\alpha}\}$.
We remark that the sets $\{R_i^{\lambda^\alpha}\}_{i\in[-n,n]}$ partition the set $J_d$.

%===================================================================
%\subsection{Parabolic subgroups and cosets}
%===============================================================================
For any $\lambda^\alpha\in\Lambda_\bD$, let $W_{\ld^\alpha}$ be the parabolic subgroup of $W_\bD$ generated by
\begin{equation}
\left\{
\begin{array}{ll}
S_\bD \setminus\{\varsigma_{\ld_0},\varsigma_{\ld_{0,1}},\ldots,\varsigma_{\ld_{0,n-1}}\} & \mbox{if $\alpha=0$},\\
S_\bD \setminus\{\varsigma_0,\varsigma_{\ld_{0,1}},\ldots,\varsigma_{\ld_{0,n-1}}\} & \mbox{if $\alpha=+$},\\
S_\bD \setminus\{\varsigma_1,\varsigma_{\ld_{0,1}},\ldots,\varsigma_{\ld_{0,n-1}}\} & \mbox{if $\alpha=-$}.
\end{array}
\right.
\end{equation}
Denote by $\mbox{Stab}(X)$ the stabilizer of $J_d$ in $W_\bD$, for any $X\subset J_d$.
%-------------------------------------------
\begin{lem}\label{lem:stabR}
For any $\ld^\alpha \in \Lambda_\bD$, we have
$
W_{\ld^{\alpha}} = \bigcap_{i = 0}^{n} \Stab (R_i^{\ld^\alpha}).
$
\end{lem}
Denote the set of minimal length right coset representatives of $W_{\ld^\alpha}$ in $W_\bD$ by
\eq
\D_{\ld^\alpha} = \big\{g\in W_\bD ~|~ \ell(wg) = \ell(w) + \ell(g), \forall  w\in W_{\ld^\alpha} \big\}.
\endeq
Hence, the set $\D_{\ld^\alpha\mu^\beta} = \D_{\ld^\alpha} \cap \D_{\mu^\beta}^{-1}$  is the set of minimal length double coset representatives for $W_{\ld^\alpha} \backslash W_\bD /W_{\mu^\beta}$.
%-----------------------------------------------------------------------------------------------------------
%-----------------------------------------------------------------------------------------------------------
\begin{lemma}
Let $g \in W_\bD$ and $\ld^\alpha \in \Ld_\bD$.
\enua
\item If $\alpha=\pm$, then $g\in\D_{\ld^\alpha}$ if and only if $g^{-1}$ is order-preserving on $R^{\ld^{\alpha}}_i$, for all $i \in [1,n]$;

\item
If $\alpha=0$, then $g\in\D_{\ld^\alpha}$ if and only if $g^{-1}$ is order-preserving on $R^{\ld^{\alpha}}_i$ for all $i \in [1,n]$ and
\[
g^{-1}(-2)<g^{-1}(1)<g^{-1}(2)<\cdots<g^{-1}(\lambda_0).
\]
\endenua
\end{lemma}

%-------------------------------------------------------------------
By a similar argument for \cite[Proposition 4.16, Lemma 4.17 and Theorem 4.18]{DDPW08}, we have the following facts.
\begin{prop}\label{prop:doublecoset}
Let $\ld^\alpha,\mu^\beta \in \Ld_\bD$ and $g \in \D_{\ld^\alpha\mu^\beta}$.
\enua
\item There is a weak composition $\delta = \delta(\ld^\alpha, g, \mu^\beta) \in \Ld_{n',d}$ for some $n'$ such that
$W_{\delta^\beta} = g^{-1} W_{\ld^\alpha} g \cap W_{\mu^\beta}$.
\item The map $W_{\ld^\alpha} \times (\D_\delta \cap W_{\mu^\beta}) \rw W_{\ld^\alpha} g W_{\mu^\beta}$ sending $(x,y)$ to $xgy$ is a bijection;
moreover, we have $\ell(xgy) = \ell(x) + \ell(g) + \ell(y)$.
\item The map $(\D_\delta\cap W_{\mu^\beta}) \times W_\delta \rw W_{\mu^\beta}$ sending $(x,y)$ to $xy$ is a bijection;
moreover, we have $\ell(x) + \ell(y) = \ell(xy)$.
\endenua
\end{prop}

%==============================================================================
\subsection{Schur algebras} \label{sec:Sj}
%==========================================================================================
The Hecke algebra $\bH = \bH(W_\bD)$ over $\bfA =\ZZ[v,v^{-1}]$
is an $\bfA$-algebra with basis $\{T_g ~|~ g\in W_\bD\}$ satisfying that
\[
\ba{{lllll}
T_w T_{w'} = T_{ww'}
&
\tif
\ell(ww') = \ell(w) + \ell(w'),
\\
(T_s+1) (T_s - v^2) = 0,
&\tfor
s \in S_\bD.
}
\]
For any finite subset $X \subset W_\bD$ and for $\ld^\alpha\in\Ld_\bD$ , set
\eq  \label{eq:xD}
T_X = \sum_{w\in X} T_w
\quad\textup{and}\quad
x_{\ld^\alpha} = T_{W_{\ld^\alpha}}.
\endeq
For $\ld^\alpha,\mu^\beta\in\Ld_\bD$ and $g\in \D_{\ld^\alpha\mu^\beta}$, we consider a right $\bH$-linear map
$
\phi_{\ld^\alpha\mu^\beta}^g \in \Hom_\bH(x_{\mu^\beta} \bH, \bH)$,
sending $x_{\mu^\beta}$ to $T_{W_{\ld^\alpha} g W_{\mu^\beta}}.$
Thanks to Proposition~ \ref{prop:doublecoset}~(b),
we have $T_{W_\ld g W_\mu} = x_\ld T_g T_{\D_\delta \cap W_\mu}$ for some $\delta\in \Ld_{n',d}$,
and hence we have constructed a right $\bH$-linear map
\eq  \label{phiD}
\phi_{\ld^\alpha\mu^\beta}^g \in \Hom_\bH(x_{\mu^\beta}\bH, x_{\ld^\alpha}\bH),
\qquad x_{\mu^\beta} \mapsto T_{W_{\ld^\alpha} g W_{\mu^\beta}} = x_{\ld^\alpha} T_g T_{\D_\delta \cap W_{\mu^\beta}}.
\endeq
%-----------------------------------------------------------------------------------------------------------
We define the {\em Schur algebra $\bS_{n,d}$ of type $\bD$} as
\eq
\label{def:Sjj}
\bS_{n,d} = \textup{End}_{\bH}
\Bp{
\mathop{\oplus}_{\ld^\alpha\in\Ld_\bD} x_{\ld^\alpha} \bH
}
= \bigoplus_{\ld^\alpha,\mu^\beta \in \Ld_\bD} \Hom_{\bH} (x_{\mu^\beta} \bH, x_{\ld^\alpha} \bH)
.
\endeq
Introduce the following subset of $\Ld_\bD \times W_\bD \times \Ld_\bD$:
\eq  \label{Dnd}
 \D_{n,d} %=\{ (\ld^\alpha, g, \mu^\beta) \in \Ld_\bD \times W \times \Ld_\bD~|~ g \in \D_{\ld\mu} \}
 =\bigsqcup_{\ld^\alpha, \mu^\beta \in \Ld_\bD} \{\ld^\alpha\} \times \D_{\ld^\alpha\mu^\beta} \times \{\mu^\beta\}.
\endeq

%-----------------------------------------------------------------------------------------------------------
\begin{lemma}
\label{lem:basis}
The set $\{\phi_{\ld^\alpha\mu^\beta}^g ~|~ (\ld^\alpha,g, \mu^\beta) \in \D_{n,d} \}$
forms an $\bfA$-basis of $\bS_{n,d}$.
\end{lemma}

%============================================================================
\subsection{Signed matrices}
%==================================================================================
From now on, we fix
\[
N=2n+1, \quad\quad D=2d.
\]
Notice that $D$ is even and is different from the convention \eqref{NdDd}.
Set
\begin{align}
\label{eq:Xind}
\begin{split}
\Xi= \Big\{A=(a_{ij})_{-n\leq i, j\leq n} \in \text{Mat}_{N\times N}(\NN)~\big |~
& a_{-i,-j} =a_{ij}, \forall i, j\in [-n,n];
\textstyle\sum_{i,j=-n}^n a_{ij}= D
 \Big \}.
 \end{split}
\end{align}
Recall $\ro(T)$ and $\co(T)$ in \eqref{def:ro}, we set
\begin{equation}
\begin{array}{l}
\Xi^0=\{A\in\Xi~|~ \ro(A)_0>0\ \mbox{and}\ \co(A)_0>0\}\times\{0\},\\
\Xi^+=\{A\in\Xi~|~ \ro(A)_0=0\ \mbox{or}\ \co(A)_0=0\}\times\{+\},\\
\Xi^-=\{A\in\Xi~|~ \ro(A)_0=0\ \mbox{or}\ \co(A)_0=0\}\times\{-\}.
\end{array}
\end{equation}
In below we abbreviate $(A,\alpha)\in \Xi^\alpha$ by $A^\alpha$ where $\alpha\in\{0,+,-\}$.
We further set
\begin{equation}
\Xi_\bD=\Xi^0\sqcup\Xi^+\sqcup\Xi^-,
\end{equation}
whose elements are called {\em signed matrices}.
Define a sign map
$\sgn:\{0,+,-\}^2\rightarrow\{0,+,-\}$ by
\begin{equation}
\sgn(\alpha,\beta)=
\left\{
\begin{array}{ll}
0, &\tif (\alpha,\beta)=(0,0);\\
+, &\tif (\alpha,\beta)=(0,+),(+,0),(+,+),(+,-);\\
-, &\tif (\alpha,\beta)=(0,-),(-,0),(-,-),(-,+).
\end{array}
\right.
\end{equation}
Define a map
$
\kappa:\D_{n,d}\rightarrow\Xi_\bD
$ by
$
\kappa(\ld^\alpha,g,\mu^\beta)=\left(|R_i^{\ld^\alpha}\cap gR_j^{\mu^\beta}|\right)^{\sgn(\alpha,\beta)}.
$

%--------------------------------------------------------------------------
\begin{lemma}
The map $\kappa:\D_{n,d}\rightarrow\Xi_\bD$ is a bijection.
\end{lemma}
For each $\mA=\kappa(\lambda^\alpha,g,\mu^\beta)\in\Xi_\bD$, we write
%\begin{equation}\label{basisnotation}
$e_\mA=\phi_{\ld^\alpha\mu^\beta}^g$,
and hence $\{e_\mA\mid A \in \Xi\}$ forms a basis of $\bS_{n,d}$.
%\end{equation}
%=======================================
%\subsection{The length formula}
%=========================================
For any $A=(a_{ij})\in\Xi$, we set
\begin{equation}
a'_{ij}=
\left\{\begin{array}{ll}
\frac{1}{2}a_{00} & \mbox{if $(i,j)=(0,0)$};\\
a_{ij} & \mbox{otherwise},\end{array}\right.
\quad \mbox{and} \quad
a''_{ij}=
\left\{\begin{array}{ll}
a_{00}-1 & \mbox{if $(i,j)=(0,0)$};\\
a_{ij} & \mbox{otherwise}.\end{array}\right.
\end{equation}
Let
$
I^+=(\{0\}\times [0,n])\sqcup([1,n]\times [-n,n])
$
be the index set corresponding to the ``positive half part'' of matrices in $\Xi$.

\begin{lemma}\label{lengthformula}
If $A^{\sgn(\alpha,\beta)}=\kappa(\lambda^\alpha,g,\mu^\beta)\in\Xi_\bD$ where $A=(a_{ij})\in\Xi$, then the length of $g\in W_\bD$ is
\begin{equation}\label{length}
\ell(g)=\frac{1}{2}\left(\sum_{(i,j)\in I^+}\left(\sum_{x>i,y<j}+\sum_{x<i,y>j}\right)a'_{ij}a''_{xy}\right).
\end{equation}
\end{lemma}
In particular, the length is independent of the sign $\sgn(\alpha,\beta)$. Thus we write,
for $\mA=A^{\sgn(\alpha,\beta)}=\kappa(\lambda^\alpha,g,\mu^\beta)\in\Xi_\bD$,
\eq
\ell(A)=\ell(g) \quad\mbox{or}\quad \ell(\mA)=\ell(g)
\endeq

%================================================================
%\subsection{Multiplication formulas}
%================================================================
%======================================================================================
%\subsection{Further notations for signed matrices}
%=======================================================================================
For each signed matrix $\mA=A^{\sgn(\alpha,\beta)}=\kappa(\ld^\alpha,g,\mu^\beta)\in\Xi_\bD$ with $A=(a_{ij})\in\Xi$, we introduce the following notations:
\eq
\begin{split}
&\sgn(\mA)=\sgn(\alpha,\beta),
\quad
s_l(\mA)=\alpha,
\quad
s_r(\mA)=\beta,
\\
&\ro(\mA)=\ro(A),
\quad
\co(\mA)=\co(A),
\quad
p(\mA)=\bc{
- &\mbox{if $\sum_{i<0,j>0}a_{ij}$ is odd};\\
+ &\mbox{otherwise},
}
\\
&\mA\pm B=A\pm B, \quad\mbox{for any $N\times N$ matrix $B$}.
\end{split}
\endeq
Note that $\mA\pm B$ is a matrix instead of a signed matrix.
The following lemmas follows immediately from definition.
\begin{lemma}
Let $\mA=\kappa(\ld^\alpha,g,\mu^\beta)\in\Xi_\bD$, then
$p(\mA)=+$ (resp. $-$) if and only if $g(1)>0$ (resp. $<0$).
\end{lemma}

\begin{lemma}
For a signed matrix $\mA\in\Xi_\bD$, we have
\eq
s_l(\mA)
=
\bc{
    0 &\tif  \ro(\mA)_0>0;
    \\
    \sgn(\mA) & \tif \ro(\mA)_0=0,
    }
\quad
s_r(\mA)=
\bc{
    0 &\tif \co(\mA)_0>0;
    \\
    -\sgn(\mA) &\tif \co(\mA)_0=\ro(\mA)_0=0, p(\mA)=-;
    \\
    \sgn(\mA) &\otw.
}
\endeq
\end{lemma}

%===============================================================
%\subsection{Poincare polynomials for $W_{\delta(\mA)}$}
%=============================================================
Let $\mA=\kappa(\ld^\alpha,g,\mu^\beta)\in\Xi_\bD$. We define a signed weak composition as below:
\begin{equation}\label{delta}
\delta(\mA)=(\frac{a_{00}}{2},a_{10},\ldots,a_{n0},a_{-n,1},a_{-n+1,1},\ldots,a_{n1},\ldots,\ldots,a_{-n,n},a_{-n+1,n},\ldots,a_{nn})^\beta.
\end{equation}

A direct computation shows that $\delta(\mA)$ is indeed a weak composition $\delta$ in Proposition \ref{prop:doublecoset}(a).
\begin{prop}\label{prop:delta}
Let $\mA=\kappa(\ld^\alpha,g,\mu^\beta)\in\Xi_\bD$. Then $W_{\delta(\mA)} = g^{-1} W_{\ld^\alpha} g \cap W_{\mu^\beta}.$
\end{prop}
We define type D quantum factorials by
\[
[0]^!_\fd=[2]^!_\fd=1,
\quad
[2k]^!_\fd=[k][2][4]\cdots[2(k-1)],
\quad
(k \geq 2).
\]
We further define,
for $A=(a_{ij})\in\Xi$,
\eq
[A]_\fd^!=[a_{0,0}]^!_\fd\prod_{(i,j)\in I^+\setminus\{(0,0)\}}[a_{ij}]!.
\endeq
We write $[\mA]^!_\fd=[A]^!_\fd$ if $\mA=A^{\sgn{\mA}}$.
The type D quantum factorials are defined in the sense that the following identity on the Poincare polynomial for $W_{\delta(\mA)}$ holds:
\begin{lemma}
For any $\mA=A^\alpha\in\Xi_\bD$ with $A=(a_{ij})$, we have
$
\sum_{w\in W_{\delta(\mA)}}v^{2\ell(w)}=[A]_\fd^!.
$
\end{lemma}

%===========================================================
\subsection{Multiplication formulas}
%===========================================================
The proofs of Lemma~\ref{lem:xTxD}--\ref{lem:multAwD} are very similar to their counterparts (Lemma~\ref{lem:xTx}, \eqref{eq:mult1} and  Lemma~\ref{lem:multAw}) so we omit.
\begin{lemma}\label{lem:xTxD}
Let $\mA=\kappa(\ld^\alpha,g,\mu^\beta)$ for $\ld^\alpha,\mu^\beta \in \Ld_{\bD}, g\in \D_{\ld^\alpha\mu^\beta}$.  Then
$x_{\ld^\alpha} T_{g} x_{\mu^\beta} = [A]^!_\fd \, e_\mA(x_{\mu^\beta}).$
\end{lemma}

\begin{lemma}\label{lem:mult1D}
Let $\mB = \kappa(\ld^\alpha,g_1,\mu^\beta)$ and $\mA = \kappa(\mu^\beta,g_2,\nu^\gamma)$, where
$\ld^\alpha,\mu^\beta, \nu^\gamma \in \Ld_\bD$, $g_1 \in \D_{\ld^\alpha\mu^\beta}$, and $g_2 \in \D_{\mu^\beta\nu^\gamma}$.
Write $\delta = \delta(\mA)$. Then we have
$
e_\mB   e_\mA(x_{\nu^\gamma}) = \frac{1}{[A]^!_\fd} x_{\ld^\alpha} T_{g_1} T_{(\D_{\delta} \cap W_{\mu^\beta})g_2} x_{\nu^\gamma}.
$
\end{lemma}

\begin{lem}\label{lem:multAwD}
Let $\mB = \kappa(\ld^\alpha,1,\mu^\beta), \mA = \kappa(\mu^\beta,g,\nu^\gamma)$.
Let $y^{(w)}$ be the shortest double coset representative for $W_\lambda wg W_\nu$, and let $\mA^{(w)}=\kappa(\lambda^\alpha,y^{(w)},\nu^\gamma)$. Then
\[
e_\mB e_\mA
=\sum_{w\in\D_{\delta} \cap W_{\mu^\beta}}
v^{2(\ell(w)+\ell(g)-\ell(y^{(w)}))}
\frac{[\mA^{(w)}]^!_\fd}{[A]^!_\fd}
e_{\mA^{(w)}}
\]
\end{lem}
%\proof
%By Lemma \ref{lem:mult1D}, we have
%\[
%e_\mB e_\mA (x_{\nu^\gamma})=\sum_{w\in\D_{\delta} \cap W_{\mu^\beta}}\frac{1}{[A]!_\fd}x_{\lambda^\alpha}T_{wg}x_{\nu^\gamma}.
%\]
%Note that $wg_2=w_1 y^{(w)} w_2$ with $y^{(w)} \in \D_{\ld^\alpha\nu^\gamma}$, $w_1 \in W_{\ld^\alpha}$, $w_2 \in W_{\nu^\gamma}$.
%We obtain
%\[
%x_{\lambda^\alpha}T_{wg_2}x_{\nu^\gamma}
%=v^{2\ell(w_1)+2\ell(w_2)}x_{\lambda^\alpha}T_{y^{(w)}}x_{\nu^\gamma}
%=v^{2\ell(wg_2)-2\ell(y^{(w)})}x_{\lambda^\alpha}T_{y^{(w)}}x_{\nu^\gamma}.
%\]
%The Lemma then follows from Lemma \ref{lem:xTxD}.
%\endproof
In the multiplication formulas below, we regard $e_\mA=0$ if $\mA\not\in\Xi_{\bD}$.
\begin{prop}\label{prop:multformula1D}
Suppose that $\mA=A^{\sgn(\mA)}, \mB, \mC\in\Xi_{\bD}$ and $h\in[1,n]$.
Let $\Gamma_r=\{t=(t_i)_{-n\leq i\leq n}\in\NN^N~|~\sum_{i=-n}^nt_i=r\}.$
\enu
\item[(1)] If $h\neq1$, $\mB-rE_{h,h-1}^\theta$ is diagonal, $\co(\mB)=\ro(\mA)$, and $s_r(\mB)=s_l(\mA)$, then
\begin{equation}
e_\mB e_\mA=\sum_{t\in\Gamma_r}v^{2\sum_{k<p}t_pa_{h,k}}\prod_{p=-n}^{n}\left[\begin{array}{cc}a_{h,p}+t_p\\t_p\end{array}\right]e_{\widecheck{\mA}_{t,h}},
\end{equation} where $\widecheck{\mA}_{t,h}=(A+t_pE_{h,p}^\theta-t_pE_{h-1,p}^\theta,\sgn(s_l(\mB),s_r(\mA)))$, $s_l(\widecheck{\mA}_{t,h})=s_l(\mB)$ and $s_r(\widecheck{\mA}_{t,h})=s_r(\mA)$.

\item[(2)] If $\mB-rE_{1,0}^\theta$ is diagonal, $\co(\mB)=\ro(\mA)$, and $s_r(\mB)=s_l(\mA)$,
then
\begin{equation}\label{eq7}
e_\mB e_\mA=
\sum_{t\in\Gamma_r}v^{2\sum_{k<p}t_pa_{1,k}}(1+(1-\delta_{r,\frac{1}{2}\ro(\mA)_0})(1-\delta_{a'_{0,0},0})\delta_{a'_{0,0},t_0})\prod_{p=-n}^{n}\left[\begin{array}{cc}a_{1,p}+t_p\\t_p\end{array}\right]e_{\widecheck{\mA}_{t,1}}.
\end{equation}

\item[(3)] If $h\neq1$, $\mC-rE_{h-1,h}^\theta$ is diagonal, $\co(\mC)=\ro(\mA)$, and $s_r(\mC)=s_l(\mA)$, then
\begin{equation}
e_\mC e_\mA=\sum_{t\in\Gamma_r}v^{2\sum_{k>p}t_pa_{h-1,k}}\prod_{p=-n}^{n}\left[\begin{array}{cc}a_{h-1,p}+t_p\\t_p\end{array}\right]e_{\widehat{\mA}_{t,h}},
\end{equation} where $\widehat{\mA}_{t,h}=(A-t_pE_{h,p}^\theta+t_pE_{h-1,p}^\theta,\sgn(s_l(\mC),s_r(\mA)))$, $s_l(\widehat{\mA}_{t,h})=s_l(\mC)$ and $s_r(\widehat{\mA}_{t,h})=s_r(\mA)$.

\item[(4)] If $\mC-rE_{0,1}^\theta$ is diagonal, $\co(\mC)=\ro(\mA)$, and $s_r(\mC)=s_l(\mA)$, then
\begin{equation}
e_\mC e_\mA=\sum_{t\in\Gamma_r}v^{2\sum_{k>p}a_{0,k}t_p+2\sum_{p<k<-p}t_pt_k+\sum_{p<0}t_p(t_p-1)}\frac{[a_{0,0}+2t_0]_\fd^!}{[a_{0,0}]_\fd^![t_0]!}\prod_{p=1}^n\frac{[a_{0,p}+t_p+t_{-p}]!}{[a_{0,p}]![t_p]![t_{-p}]!}e_{\widehat{\mA}_{t,1}}.
\end{equation}
\endenu
\end{prop}
\begin{proof}
Here we only prove Parts (2) and (4) while omitting the easier parts (1) and (3).
For Part (2), let $\mA=\kappa(\mu^\beta,g_2,\nu^\gamma)$, and let $\delta=\delta(\mB)$.
Take any $t\in\Gamma_r$, we consider two cases: $r<\frac{1}{2}\ro(\mA)_0$ or $r=\frac{1}{2}\ro(\mA)_0$.

\item[Case 1: $r<\frac{1}{2}\ro(\mA)_0$:]
Let $w_t$ be the minimal length element in the set $\{ w\in\D_{\delta} \cap W_{\mu^\beta} \mid \mA^{(w)}=\widecheck{\mA}_{t,1}\}$.
A direct computation shows that its length is give by
\begin{equation}\label{eq1}
\begin{array}{rcl}
\ell(w_t)&=&\sum\limits_{\substack{k>p\geq 0 \\ \text{or} \\ k\geq-p>0}}t_{p}(a_{0,k}-t_k)+\sum_{|k|<-p}t_p(a_{0,k}-t_k-t_{-k})-\sum_{p<0}\frac{(t_p+1)t_p}{2}\\
&=&\sum_{k>p}(a_{0,k}-t_k)t_p-\sum_{p<k<-p}t_pt_k-\frac{1}{2}\sum_{p<0}t_p(t_p+1).
\end{array}
\end{equation}
By a combinatorial argument, we calculate that
\eq
\sum_{\substack{w\in\D_{\delta} \cap W_{\mu^\beta},\\ \mA^{(w)}=\widehat{\mA}_{t,1}}}v^{2\ell(w)}
=
v^{2\ell(w_t)}
\left(\sum_{x+y=t_{0}}
\lrb{a'_{0,0}}{x}
\lrb{a'_{0,0}-x}{y}
(v^2)^{\frac{x(x-1)}{2}+x(a'_{0,0}-t_{0})}\right)\prod_{p=1}^n\lrb{a_{0,p}}{t_p}\lrb{a_{0,p}-t_p}{t_{-p}}.
\endeq
Note that
\eq
\begin{split}
\sum_{x+y=t_{0}}
\lrb{a'_{0,0}}{x}
\lrb{a'_{0,0}-x}{y}
(v^2)^{\frac{x(x-1)}{2}+x(a'_{0,0}-t_{0})}
&=\lrb{ a'_{0,0}}{t_{0}}
\sum_{x=0}^{t_{0}}
\lrb{t_{0}}{x}
v^{x(x-1)} (v^{a'_{0,0}-t_{0}})^{2x}
\\
&\stackrel{(\diamondsuit)}{=}\lrb{ a'_{0,0}}{t_{0}}
\prod_{i=1}^{t_{0}}(1+v^{2(i-1)}v^{2(a'_{0,0}-t_{0})})
\\
&= (1+(1-\delta_{a'_{0,0},0})\delta_{a'_{0,0},t_0})\frac{[a_{0,0}]^!_\fd}{[a_{0,0} - 2t_{0}]^!_\fd [t_{0}]!},
\end{split}
\endeq
where ($\diamondsuit$) is due to the quantum binomial theorem
$
\sum_{r=0}^n
\lrb{n}{r}
v^{r(r-1)}x^r
=\prod_{k=0}^{n-1}(1+v^{2k}x)
$.
Hence
\begin{equation}\label{eq2}
\sum_{w\in\D_{\delta} \cap W_{\mu^\beta},\mA^{(w)}=\widehat{\mA}_{t,1}}v^{2\ell(w)}=
v^{2\ell(w_t)}\frac{[a_{0,0}]^!_\fd}{[a_{0,0} - 2t_{0}]^!_\fd [t_{0}]!}\prod_{p=1}^n\lrb{a_{0,p}}{t_p}\lrb{a_{0,p}-t_p}{t_{-p}}.
\end{equation}
Moreover, using \eqref{length}, we obtain
\begin{align}\label{eq3}
%\begin{array}{l}
\ell(\mA)-\ell(\widehat{\mA}_{t,1}) &=
-\sum_{k>p}(a_{0,k}-t_k)t_p+\frac{1}{2}\sum_{p<0}t_p+\sum_{k<p}t_pa_{1,k}+\frac{1}{2}\sum_{k<-p}t_pt_k\nonumber\\
&=\sum_{k<p}t_pa_{1,k}-\sum_{k>p}(a_{0,k}-t_k)t_p+\sum_{p<k<-p}t_pt_k+\frac{1}{2}\sum_{p<0}t_p(t_p+1).
\end{align}
Combining Lemma~\ref{lem:multAwD}, \eqref{eq1}, \eqref{eq2} and \eqref{eq3}, we obtain that, if $r<\frac{1}{2}\ro(\mA)_0$,
\begin{equation*}
e_\mB e_\mA=
\sum_{t\in\Gamma_r}v^{2\sum_{k<p}t_pa_{1,k}}(1+(1-\delta_{a'_{0,0},0})\delta_{a'_{0,0},t_0})\prod_{p=-n}^{n}\left[\begin{array}{cc}a_{1,p}+t_p\\t_p\end{array}\right]e_{\widecheck{\mA}_{t,1}}.
\end{equation*}

\item[Case 2: $r=\frac{1}{2}\ro(\mA)_0$:]
In this case, each term $e_{\widecheck{\mA}_{t,1}}=0$ unless
$a_{0,p}=t_p+t_{-p}$ for all  $p\in[-n, n]$. (Particularly, $a'_{0,0}=t_0$.)
For the non-vanishing terms, we have
\begin{equation*}
\sum_{w\in\D_{\delta} \cap W_{\mu^\beta},\mA^{(w)}=\widehat{\mA}_{t,1}}v^{2\ell(w)}=v^{2\ell(w_t)}
\left(\sum_{x}
\lrb{a'_{0,0}}{x}
(v^2)^{\frac{x(x-1)}{2}}\right)\prod_{p=1}^n\lrb{a_{0,p}}{t_p},
\end{equation*}
where $x$ runs over all integers such that $0\leq x\leq a'_{0,0}$ and $x+\sum_{p<0}t_p\in2\NN$.
Note that
$$\sum_{a'_{0,0}\geq x\in2\NN}\lrb{a'_{0,0}}{x}v^{x(x-1)}=\sum_{a'_{0,0}\geq x\in2\NN+1}\lrb{a'_{0,0}}{x}v^{x(x-1)}=\prod_{i=1}^{a'_{0,0}-1}(1+v^{2i}).$$
Hence
\begin{equation}\label{eq6}
\sum_{w\in\D_{\delta} \cap W_{\mu^\beta},\mA^{(w)}=\widehat{\mA}_{t,1}}v^{2\ell(w)}=v^{2\ell(w_t)}
\prod_{i=1}^{a'_{0,0}-1}(1+v^{2i})\prod_{p=1}^n\lrb{a_{0,p}}{t_p}.
\end{equation}
Combining Lemma~\ref{lem:multAwD}, \eqref{eq1}, \eqref{eq3} and \eqref{eq6}, we obtain, if $r=\frac{1}{2}\ro(\mA)_0$,
\begin{equation*}
e_\mB e_\mA=
\sum_{t\in\Gamma_r}v^{2\sum_{k<p}t_pa_{1,k}}\prod_{p=-n}^{n}\left[\begin{array}{cc}a_{1,p}+t_p\\t_p\end{array}\right]e_{\widecheck{\mA}_{t,1}}.
\end{equation*}
Part (2) concludes.

For Part (4), Let $\mA=\kappa(\mu^\beta,g_2,\nu^\gamma)$, $\delta=\delta(\mC)$ and take any $t\in\Gamma_r$.  Let $w_t$ be the shortest element in the set $\{w\in\D_{\delta} \cap W_{\mu^\beta} \mid \mA^{(w)}=\widehat{\mA}_{t,1}\}$. Its length is given by
\begin{equation}\label{eq4}
\sum_{w\in\D_{\delta} \cap W_{\mu^\beta},\mA^{(w)}=\widehat{\mA}_{t,1}}v^{2\ell(w)}=
v^{2\ell(w_t)}\prod_{p=-n}^n\lrb{a_{1,p}}{t_p}=v^{2\sum_{k<p}t_p(a_{1,k}-t_k)}\prod_{p=-n}^n\lrb{a_{1,p}}{t_p}.
\end{equation}
Moreover, using \eqref{length}, we obtain
\eq\begin{split}\label{eq5}
\ell(\mA)-\ell(\widehat{\mA}_{t,1})&=
\sum_{k>p}a_{0,k}t_p-\frac{1}{2}\sum_{p<0}t_p-\sum_{k<p}t_p(a_{1,k}-t_k)+\frac{1}{2}\sum_{k<-p}t_pt_k
\\
&=\sum_{k>p}a_{0,k}t_p-\sum_{k<p}t_p(a_{1,k}-t_k)+\sum_{p<k<-p}t_pt_k+\frac{1}{2}\sum_{p<0}t_p(t_p-1).
\end{split}
\endeq
Combining Lemma~\ref{lem:multAwD},\eqref{eq4} and \eqref{eq5}, we finally get that
\begin{eqnarray*}
e_\mC e_\mA
&=&\sum_{t\in\Gamma_r}v^{2\sum_{k>p}a_{0,k}t_p+2\sum_{p<k<-p}t_pt_k+\sum_{p<0}t_p(t_p-1)}\left(\prod_{p=-n}^n\lrb{a_{1,p}}{t_p}\frac{[a_{1,p}-t_p]!}{[a_{1,p}]!}\right)\\
&&\cdot\left(\frac{[a_{0,0}+2t_0]_\fd^!}{[a_{0,0}]_\fd^!}\prod_{p=1}^n\frac{[a_{0,p}+t_p+t_{-p}]!}{[a_{0,p}]!}\right)e_{\widehat{\mA}_{t,1}}
\\&=&\sum_{t\in\Gamma_r}v^{2\sum_{k>p}a_{0,k}t_p+2\sum_{p<k<-p}t_pt_k+\sum_{p<0}t_p(t_p-1)}
\frac{[a_{0,0}+2t_0]_\fd^!}{[a_{0,0}]_\fd^![t_0]!}\prod_{p=1}^n\frac{[a_{0,p}+t_p+t_{-p}]!}{[a_{0,p}]![t_p]![t_{-p}]!}e_{\widehat{\mA}_{t,1}}.
\end{eqnarray*}
\end{proof}

Take $r=1$ in Proposition \ref{prop:multformula1D}, we have the following corollary.
\begin{cor}\label{coro:multformula2}
Suppose that $\mA=A^{\sgn(\mA)}, \mB, \mC\in\Xi_{\bD}$ and $h\in[1,n]$.

\enu
\item[(1)] If $h\neq1$, $\mB-E_{h,h-1}^\theta$ is diagonal, $\co(\mB)=\ro(\mA)$, and $s_r(\mB)=s_l(\mA)$, then
\begin{equation}
e_\mB e_\mA=\sum_{p=-n}^nv^{2\sum_{k<p}a_{h,k}}[a_{h,p}+1]e_{\mA_p},
\end{equation} where $\mA_p=(A+E_{h,p}^\theta-E_{h-1,p}^\theta,\sgn(s_l(\mB),s_r(\mA)))$.

\item[(2)] If $\mB-E_{1,0}^\theta$ is diagonal, $\co(\mB)=\ro(\mA)$, and $s_r(\mB)=s_l(\mA)$, then
\begin{equation}
e_\mB e_\mA=\sum_{p\neq0}v^{2\sum_{k<p}a_{1,k}}[a_{1,p}+1]e_{\mA_p}
+v^{2\sum_{k<0}a_{1,k}}(2-\delta_{2,\ro(A)_0})[a_{1,0}+1]e_{\mA_0}.
\end{equation}

\item[(3)] If $h\neq1$, $\mC-E_{h-1,h}^\theta$ is diagonal, $\co(\mC)=\ro(\mA)$, and $s_r(\mC)=s_l(\mA)$, then
\begin{equation}
e_\mC e_\mA=\sum_{p=-n}^nv^{2\sum_{k>p}a_{h-1,k}}[a_{h-1,p}+1]e_{\mA(h,p)},
\end{equation} where $\mA(h,p)=(A-E_{h,p}^\theta+E_{h-1,p}^\theta,\sgn(s_l(\mC),s_r(\mA)))$.

\item[(4)] If $\mC-E_{0,1}^\theta$ is diagonal, $\co(\mC)=\ro(\mA)$, and $s_r(\mC)=s_l(\mA)$, then
\begin{equation}
e_\mC e_\mA=\sum_{p\neq0}v^{2\sum_{k>p}a_{0,k}}[a_{0,p}+1]e_{\mA(1,p)}
+v^{2\sum_{k>0}a_{0,k}}\left([a_{0,0}+1]+(1-\delta_{0,a_{0,0}})v^{a_{0,0}}\right)e_{\mA(1,0)}.
\end{equation}
\endenu
\end{cor}

\begin{rmk}
The multiplication formulas with $e_{\mA}$ (Proposition \ref{prop:multformula1D} and Corollary \ref{coro:multformula2}) match Fan-Li's multiplication formulas (\cite[Proposition~4.3.2 and Corollary~4.3.4]{FL15}.) with $e^{\mrm{geo}}_{\mA}$, via the following correspondence:
\begin{equation}
e_\mA\rightarrow\left\{\begin{array}{ll}
\frac{1}{2}e^{\mathrm{geo}}_\mA, & \mbox{if $a_{0,0}=0$, $\ro(A)_0\neq0$ and $\co(A)_0\neq0$};\\
e^{\mathrm{geo}}_\mA & \mbox{otherwise}.
\end{array}\right.
\end{equation}
\end{rmk}
\begin{rmk}
An immediate application of the multiplication formulas is to demonstrate a stabilization property for $\{\bS_{n,d}\mid d\in \NN\}$, and further construct an algebra $\mathcal{K}_n$ so that the multiplication rules on $\mathcal{K}_n$ are compatible with the rules on any $\bS_{n,d}$.
The algebras $\mathcal{K}_n$ have been introduced by Fan and Li in {\it loc. cit}.
\end{rmk}

%%%%%%%%%%%%%%%%%%%%%%%%%%%%%%%%%%%%%%%%%%%%%%%%%%%%%%%%%%%%%%%%%%%%%%%%%%%%%%%%%%%%%%%%%%%%%%%%%%%%%%%%%%%%%%%%
\subsection{Schur duality}
Let $\mathfrak{g}$ be the simple Lie algebra of type $\bD_d$, and let $\rho$ be the half sum of the positive roots of $\mathfrak{g}$. It was mentioned in a framework \cite{LW17} that $\Ld_\bD$ can be viewed as the set of orbits of $W$ on a (truncated) $\rho$-shifted weight lattice of $\mathfrak{g}$. Then the $v$-tensor space $\bigoplus_{\ld^\alpha\in\Ld_\bD}x_{\ld^\alpha}\bH$ can be viewed as the quantum version of the Grothendieck groups of the category $\mathcal{O}$ of $\mathfrak{g}$-modules.

This picture is also valid when $\Ld_\bD$ is replaced by its subset. Each subset $\Ld_f\subset\Ld_\bD$ corresponds to a Schur algebra
$$\bS_f = \textrm{End}_{\bH}
\Bp{
\mathop{\oplus}_{\ld\in\Ld_f} x_{\ld} \bH.
}
$$
A Schur duality is also obtained in {\it loc. cit.} for each pair $(\bS_f, \bH)$ on the tensor space $\mathop{\oplus}_{\ld\in\Ld_f} x_{\ld} \bH$.

\rmk
If $\Ld_f=\Lambda^+\sqcup\Lambda^-$, then $\bS_f$ is the algebra $\mathcal{S}^m$ in \cite[\S 6.1]{FL15}.
The stabilization procedure affords a different quantum algebra $\mathcal{K}^m$ in {\it loc. cit.}
\endrmk
\rmk
Fan and Li told the authors in private conversations that they have also been aware of the Schur algebra $\bS_f$ and the related Schur duality for $\Ld_f=\Lambda^+$ or $\Lambda^0\sqcup\Lambda^+$ although they did not write it down.
\endrmk

%%%%%%%%%%%%%%%%%%%%%%%%%%%%%%%%%%%%%%%%%%%%%%%%%%%%%%%%%%%%%%%%%%%%%%%%%%%%%%%%%%%%%%%%%%%%%%%%%%%%%%%%%%%%

%\printindex{Index of Notation}
%\addcontentsline{toc}{chapter}{Index}

\end{document}